\documentclass[12pt]{amsart}
\usepackage{amsmath, amsthm, amssymb,esint}
\usepackage{fullpage}
\usepackage{color}
\usepackage{hyperref}
\usepackage{soul}
\usepackage{enumitem}
\usepackage{tikz}
\usetikzlibrary{positioning}

\usepackage{pgfplots}

\pgfplotsset{compat=1.10}
\usetikzlibrary{patterns}
\usetikzlibrary{optics}

\newtheorem{theorem}{Theorem}

\newtheorem{lemma}{Lemma}
\newtheorem{proposition}[lemma]{Proposition}
\newtheorem{corollary}[lemma]{Corollary}
\newtheorem{definition}[lemma]{Definition}
\newtheorem{remark}[lemma]{Remark}

\numberwithin{lemma}{section}

\DeclareMathOperator{\supp}{\mathrm{supp}}

\numberwithin{equation}{section}

\newcommand{\R}{{\mathbb R}}
\newcommand{\N}{{\mathbb N}}

\renewcommand{\H}{{\mathcal H }}
\newcommand{\bfH}{{\mathbf H }}

\renewcommand{\R}{\mathbb R}
\newcommand{\tOmega}{{\tilde \Omega}}

\newcommand{\tw}{{\tilde w}}
\newcommand{\tr}{{\tilde r}}
\newcommand{\tv}{{\tilde v}}
\newcommand{\ts}{{\tilde s}}

\newcommand{\dr}{{\delta r}}
\newcommand{\dx}{{\delta x}}

\newcommand{\tF}{{\tilde F}}
\newcommand{\tG}{{\tilde G}}
\newcommand{\tA}{{\tilde A}}
\newcommand{\tB}{{\tilde B}}

\newcommand{\tC}{{\tilde C^{0,\frac12}}}

\newcommand{\norm}{{\mathbf N}}

\renewcommand{\div}{\mbox{ div }}
\newcommand{\curl}{\mbox{\,curl }}
\renewcommand{\norm}{|\!|\!|}

\renewcommand{\rm}{r_-}
\newcommand{\vm}{v_-}
\newcommand{\rp}{r_+}
\newcommand{\vp}{v_+}

\begin{document}

\title{The compressible Euler equations in a physical vacuum: a comprehensive Eulerian approach}

\author{Mihaela Ifrim}
\address{Department of Mathematics, University of Wisconsin, Madison}
\email{ifrim@wisc.edu}

\author{ Daniel Tataru}
\address{Department of Mathematics, University of California at Berkeley}
\email{tataru@math.berkeley.edu}

\begin{abstract}
This article is concerned with  the local well-posedness problem 
for the compressible Euler equations in gas dynamics. For this system we 
consider the free boundary problem which corresponds to a physical vacuum.

Despite the clear physical interest in this system, the prior work  on this problem
is limited to Lagrangian coordinates, in high regularity spaces. Instead, the objective of the
present work is to provide a new, fully Eulerian approach to this problem, 
which provides a complete, Hadamard style well-posedness theory for this problem in low 
regularity Sobolev spaces.  In particular we give new proofs for  both existence, uniqueness, and continuous dependence on the data with sharp, scale invariant 
energy estimates, and continuation criterion.   
 \end{abstract}

\maketitle

\setcounter{tocdepth}{1}
\tableofcontents
\section{Introduction}

In this article we study the dynamics of the free boundary problem for
a compressible gas.  In the simplest form, the gas is contained in a
moving domain $\Omega_t$ with boundary $\Gamma_t$, and is described
via its \emph{density} $\rho \geq 0$ and \emph{velocity} $v$. 
The evolution of the Eulerian variables  $(\rho,v)$  is given by the 
compressible Euler equations
\begin{equation} \label{free-bd-euler}
\left\{
\begin{aligned}
& \rho_t + \nabla(\rho v) = 0
\\
& \rho (v_t + (v \cdot \nabla) v) + \nabla p = 0,
\end{aligned}
\right.
\end{equation}
with the constitutive law
\[
p = p(\rho).
\]
In the present paper we will consider constitutive laws of the form\footnote{Here, for expository reasons, we use $\kappa+1$ rather than $\kappa$ as the exponent, as it is more common in the literature.} 
\begin{equation}\label{kappa}
p(\rho) = \rho^{\kappa+1}, \qquad \kappa > 0.
\end{equation}

Heuristically one can view this system as a coupled system consisting of a wave equation 
for the pair $(\rho,\nabla \cdot v)$ and a transport equation for $\omega = \curl v$.
In this interpretation, a key physical quantity is the propagation speed $c_s$ for the wave 
component. This is called the \emph{speed of sound}, and is
given by 
\begin{equation}\label{sound-speed}
c_s^2 = p'(\rho) .
\end{equation}

We consider this system in the presence of vacuum states, i.e. the density $\rho$ is allowed to vanish. The gas is located in the domain $\Omega_t := \{ (t,x) \, |\,  \rho (t,x) > 0 \}$, whose boundary $\Gamma_t$
is moving. The defining characteristic  in the case of a gas, versus the fluid case, is that the 
density vanishes on the free boundary $\Gamma_t$, which is thus described by
\[
\Gamma_t = \partial \Omega_t := \{ (t,x)\, |\,  \rho (t,x) = 0 \}.
\]
In this context, the decay rate of the sound speed near the free boundary
plays a fundamental role both in the gas dynamics and in the analysis. In essence, one expects that there 
is a single stable, nontrivial physical regime, which is called \emph{physical vacuum},
and corresponds to the sound speed decay rate
\begin{equation}\label{sound}
c_s^2(t,x) \approx d(x, \Gamma_t).
\end{equation}
The property \eqref{sound} will propagate in time for as long as $\nabla v \in L^\infty$,
which will be the case for all solutions considered in this article. We remark that
 in particular such a bound guarantees  a bilipschitz fluid flow.

To provide some intuition for this we note that the acceleration of particles on the free boundary
is exactly given by $-\kappa^{-1} \nabla c_s^2$, which is normal to the boundary. Heuristically, because 
of this, the property \eqref{sound} yields the correct balance which allows the free boundary to move 
with a bounded velocity and acceleration while interacting with the interior, as follows:
\begin{itemize}
\item A faster fallout rate for the sound speed would cause the
  boundary particles to simply move independently and linearly  with the outer particle
  speed. This can only last for a short time, until the faster waves
  inside overtake the boundary and likely lead to a more stable regime
  where \eqref{sound} holds. See for instance the results in this direction in \cite{Serre}, but also the dispersive scenario discussed in \cite{Grassin}.

\item A slower fallout rate would cause an infinite initial acceleration of the boundary, 
likely leading again to the same pattern.
\end{itemize}
A fundamental observation concerning physical vacuum 
is that the relation \eqref{sound} guarantees that 
linear waves with speed $c_s$ can reach the free boundary $\Gamma_t$ in finite time.
Because of this, in the above flow the motion of the boundary is strongly coupled to 
the wave evolution and is not just a self-contained evolution at leading order.

There are two classical approaches in fluid dynamics, using either Eulerian coordinates,
where the reference frame is fixed and the fluid particles are moving,  or using Lagrangian
coordinates, where the particles are stationary but the frame is moving.  Both of these 
approaches have been extensively developed in the context of the compressible Euler 
equations, where the local well-posedness problem is very well understood.

By contrast, the free boundary problem corresponding to the physical vacuum  
has been far less studied and understood. Because of the difficulties related
to the need to track the evolution of the free boundary, all the prior work is 
in the Lagrangian setting and in high regularity spaces which are only indirectly defined.

Our goal in this paper is to provide a \emph{new, complete, low regularity approach} 
for this free boundary problem which is \emph{fully within the Eulerian framework}. 
In particular, our work contains the following steps, each of which represents
original, essential advances in the study of this problem:

\begin{enumerate}[label=\alph*)]
\item We prove the \emph{uniqueness} of solutions with very limited regularity\footnote{In an appropriately weighted sense in the case of $\rho$, see Theorem~\ref{t:unique}.}  $v \in Lip$, $\rho \in Lip$. More generally, at the same regularity level we prove \emph{stability},
by showing that bounds for a certain distance between different solutions can be propagated in time.

\item We develop the Eulerian Sobolev \emph{function space structure} where this problem should be considered, providing the correct, natural scale of spaces for this evolution.

\item We prove sharp, \emph{scale invariant energy estimates} within the above mentioned scale of spaces,
which show that the appropriate  Sobolev regularity of solutions can be continued for as long as we have uniform bounds at the same scale  $v \in Lip$. 

\item We give a simpler, more elegant proof of \emph{existence} for regular solutions, fully within the Eulerian setting, based on  the above energy estimates.

\item We devise a nonlinear Littlewood-Paley type method to obtain \emph{rough solutions} as unique limits of smooth solutions, also proving the \emph{continuous dependence} of the solutions on the initial data.
\end{enumerate}

At a conceptual level, we also remark that in our approach the study of the linearized problem plays the main role, whereas the energy bounds for the full system are seen as secondary, derived estimates. This is unlike in prior works, where the linearized equation is relegated to a secondary role 
if it appears at all.

\subsection{The material derivative and the Hamiltonian}

The derivative along the particle trajectories $D_t$ is called the material derivative and is defined as
\[
D_t = \partial_t + v \cdot  \nabla .
\]
With this notation the system \eqref{free-bd-euler} is rewritten as
\begin{equation} \label{free-bd-euler-dt}
\left\{
\begin{aligned}
& D_t \rho + \rho \nabla v = 0
\\
& \rho D_t v  + \nabla p = 0.
\end{aligned}
\right.
\end{equation}
Differentiating once more in the first equation we obtain
\[
D_t ^2\rho  - \rho \nabla (\rho^{-1} p'(\rho) \nabla \rho)  = \rho [ (\nabla\cdot v)^2 - Tr (\nabla v)^2 ],
\]
which at leading order is a wave equation for $\rho$ with propagation speed $c_s$, and
where $\nabla \cdot v$ can be viewed as a dependent variable.

On the other hand, for the vorticity $\omega = \curl v$ one can use the second equation
to obtain the transport equation
\[
D_t \omega = - \ \omega\cdot  \nabla v -(\nabla v)^T \omega .
\]

The last two equations show that indeed one can interpret the Euler equations
as a coupled system consisting of a wave equation  for the pair $(\rho,\nabla v)$ and a transport equation for $\omega = \curl v$.

This problem admits a conserved energy, which in a suitable setting can be interpreted
as a Hamiltonian, see \cite{Marsden-Chorin}, \cite{ratiu}, \cite{ebin1969},
\[
E = \int_{\Omega_t} e \, dx,
\] 
where the energy density $e$ is given by 
\[
e = \frac12 \rho v^2 + \rho h(\rho),
\]
with the specific enthaply $h$ defined by
\[
h(\rho) =  \int_{0}^\rho \frac{p(\lambda)}{\lambda^2} \, d\lambda.
\]

\subsection{The good variables}
The pair of variables $(\rho,v)$ is convenient to use if $\kappa = 1$. However, for other values of $\kappa$ 
in \eqref{kappa} we can make a better choice.  To understand that, we compute the sound speed
\[
c_s^2 = (\kappa+1) \rho^\kappa.
\]
This should have linear behavior near the boundary. Because of this, it is more convenient 
to use $r = r(\rho)$ defined by 
\[
r' = \rho^{-1} p'(\rho),
\]
which gives 
\[
r = \dfrac{\kappa +1}{\kappa} \rho^\kappa
\]
as a good variable instead of $\rho$.  

Written in terms of $(r,v)$ the equations become
\begin{equation} \label{free-bd-euler-a}
\left\{
\begin{aligned}
& r_t +  v \nabla r + \rho r' \nabla v = 0
\\
& v_t + (v \cdot \nabla) v  + \nabla r = 0.
\end{aligned}
\right.
\end{equation}
In our case we have $\rho r' = \kappa r$ so we rewrite the above system as
 \begin{equation} \label{free-bd-euler-b}
\left\{
\begin{aligned}
& r_t +  v \nabla r + \kappa r \nabla v = 0
\\
& v_t + (v \cdot \nabla) v  + \nabla r = 0
\end{aligned}
\right.
\end{equation}
or, using material derivatives,
\begin{equation} \label{free-bd-euler-b-dt}
\left\{
\begin{aligned}
& D_t r + \kappa r \nabla v = 0
\\
& D_t v  + \nabla r = 0.
\end{aligned}
\right.
\end{equation}
We will work with this system for the rest of the paper.

\subsection{Energies and function spaces} 

Given the constitutive law \eqref{kappa}, the  conserved energy is
\begin{equation}\label{energy-kappa}
E = \int  \frac{1}{\kappa} \rho^{\kappa+1}  + \frac12 \rho  v^2   \,  dx .
\end{equation}
Switching to the $(r,v)$ variables and adjusting constants, we obtain
\begin{equation}\label{energy-rv}
E = \int  r^{\frac{1-\kappa}{\kappa}} \left( r^2 + \frac{\kappa+1}2 r v^2 \right)\,  dx .
\end{equation}

This will not be directly useful in solving the equation, but will give us a good idea 
for the higher order function spaces we will have to employ. 
Based on this, we introduce  the energy space $\H$ with norm
\begin{equation}\label{def-H}
\|(s,w)\|^2_{\H} =  \int r^{\frac{1-\kappa}{\kappa}}  \left(|s|^2 + \kappa r |w|^2\right) \,  dx
\end{equation}
for functions $(s,v)$ defined a.e. within the fluid domain $\Omega_t$. Importantly, we note that the constants above do not match \eqref{energy-rv}, and instead have been adjusted to match the energy functional for the linearized equation, which is discussed in Section~\ref{s:linearize}. The two components of the $\H$ 
space as weighted $L^2$ spaces,
\[
\H = L^2(r^{\frac{1-\kappa}{\kappa}}) \times  L^2(r^{\frac{1}{\kappa}}). \]

For higher regularity, we take our cue from the second order wave equation, which has 
the leading operator $c_s^2 \Delta = r \Delta$, which is naturally associated to the 
acoustic metric\footnote{Technically one should add a $k^{-1}$ factor here.} 
\begin{equation}\label{metric}
g = r^{-1} dx^2 \qquad \text{ in \ \ } \Omega_t.
\end{equation}
 Correspondingly, we define the higher order Sobolev spaces $\H^{2k}$ 
for distributions within the fluid domain $\Omega_t$ to have norms
\[
\| (s,w)\|^2_{\H^{2k}} =  \sum_{|\beta| \leq 2k}^{|\beta|-\alpha \leq k}  \|r^{\alpha} \partial^\beta(s,w)\|^2_{\H},
\]
where $\alpha$ is implicitly restricted to $0\leq \alpha \leq k$.
More generally, for all real $k\geq 0$ one can define by interpolation the 
spaces $\H^{2k}$. These  spaces and their properties are further discussed  in the next section.

\subsection{Scaling and control parameters}

The equation \eqref{free-bd-euler-b} admits the 
scaling law
\begin{equation}\label{scaling}
(r(t,x),v(t,x)) \to (\lambda^{-2} r(\lambda t, \lambda^2 x), \lambda^{-1} v(\lambda t, \lambda^2 x) ).
\end{equation}
We use this scaling in order to  track the \emph{order} of factors in multilinear expressions,
introducing  a counting device based on scaling:
\begin{enumerate}[label=\roman*)]
\item $r$ and $v$ have degree $-1$, respectively $- \frac12$.

\item  $\nabla$ has order $1$ and $D_t$ has order $\frac12$.
\end{enumerate}
The order of a multilinear expression is defined as the sum of the orders of each factors. 
In this way, all terms in each of the  equations have the same order. This property remains valid if we 
either differentiate the equations in $x$, $t$ or apply the material derivative $D_t$.

Corresponding to the above spaces and scaling we identify the \emph{critical space} $\H^{2k_0}$ where 
$k_0$ is given by\footnote{In general this will not be an integer.} 
\[
2k_0 =  d + 1 + \frac{1}{\kappa}.
\]
This has the property that its (homogeneous) norm is invariant with respect to the above scaling.

Associated to this Sobolev exponent we introduce the following scale invariant time dependent pointwise 
control norm
\begin{equation}
    \label{E:A_norm}
  A = \| \nabla r - N \|_{L^\infty} + \| v \|_{\dot C^\frac12},  
\end{equation}
where $N$ is a given nonzero vector. Here $N$ can be chosen as $N =
\nabla r(x_0)$ for some fixed point $x_0$ where $r(x_0) = 0$. The
motivation for using such an $N$, rather than just $\| \nabla
r\|_{L^\infty}$, is that the latter is a scale invariant quantity of fixed,  unit size.
On the other hand the $A$ defined above can be harmlessly assumed to be small simply by  working  in a small neighbourhood of the reference point $x_0$. Such a localization  is allowed in the study of compressible Euler system  because of the finite speed of propagation.  
 The control parameter $A$ will play a leading role in elliptic estimates at fixed time,
 and, in order to avoid cumbersome notations, will be implicitly assumed to be small in all of our analysis.

For the energy estimates we will also introduce a second time dependent  control norm
which is associated with the space $\H^{2k_0+1}$, namely
\begin{equation}
\label{E:B_norm}
B = \| \nabla r \|_{\tC} + \| \nabla v \|_{L^\infty},
\end{equation}
where the $\tC$ norm is given by
\[
\| f \|_{\tC} = \sup_{x,y \in \Omega_t}  \frac{ |f(x) -f(y)|}{r(x)^\frac12+ r(y)^\frac12 + {|x-y|^\frac12}}. 
\]
This scales like he $\dot C^\frac12$ norm, but it is weaker in that it only uses one derivative of $r$ away from the free
boundary.

The role of  $B$ will be to control the growth rate for our energies, while also allowing for a secondary dependence of the implicit constants on $A$.

\subsection{The main results}

Our main result is a well-posedness result for the compressible Euler evolution \eqref{free-bd-euler-b}. However,
it is more revealing to break the result down into several components. We begin with the 
uniqueness result, which requires least regularity.

\begin{theorem}[Uniqueness]\label{t:unique}
For every Lipschitz initial data $(r_0,v_0)$ satisfying the nondegeneracy condition
$|\nabla r_0| > 0$ on $\Gamma_0$, the system  \eqref{free-bd-euler-b}  admits at most one solution $(r,v)$ in the class
\begin{equation}\label{uniqueness-class}
v \in C^1_x, \qquad \nabla r \in \tC_x.
\end{equation}
\end{theorem}
In other words, uniqueness holds in the class of solutions $(r,v)$ for which $B$ remains finite. One can further relax this
to $B \in L^1_t$.  We note that only the spatial regularity is specified in the theorem, as the time regularity can then be obtained from the equations. Also the nondegeneracy condition
is only given at the initial time, but it can be easily propagated to later times given our 
regularity assumptions.

To the best of our knowledge, this is the first uniqueness proof for this problem which  applies directly in the Eulerian setting, and also the 
first uniqueness result at low, scale invariant\footnote{Scale invariance corresponds to the assumption $B \in L^1_t$.} regularity.

\begin{remark}
The result in Theorem~\ref{t:unique} can be seen as a subset of Theorem~\ref{t:Diff} in Section~\ref{s:uniqueness}. There we go one step further, and prove that a suitable nonlinear
distance between two solutions is propagated along the flow, under the same assumptions as 
in Theorem~\ref{t:unique}.
\end{remark}

Next we consider the well-posedness question. Here we define the phase space
\begin{equation}\label{phase-space}
\bfH^{2k} = \{ (r,v)\, |\,  \  (r,v)  \in \H^{2k} \}.
\end{equation}
One should think of this in a nonlinear fashion, as an infinite dimensional manifold, as the $\H^{2k}$ norms depend on $\Omega_t$ and thus on $r$. The topology on this manifold is discussed in the next section. Now we can state our main well-posedness result:

\begin{theorem}[Well-posedness]\label{t:lwp}
The system \eqref{free-bd-euler} is locally well-posed in the space $\bfH^{2k}$ 
for $k \in \R$ with
\begin{equation}\label{k-range}
2k > 2k_0+1.
\end{equation}
\end{theorem}

The well-posedness result should be interpreted in a quasilinear fashion, i.e. including:

\begin{itemize}
\item Existence of solutions $(r,v) \in C[0,T;\bfH^{2k}]$.

\item Uniqueness of solutions in a larger class, see Theorem~\ref{t:unique} above.

\item Weak Lipschitz dependence on the initial data, relative to a new, nonlinear distance 
functional introduced in Section~\ref{s:uniqueness}.

\item Continuous dependence of the solutions on the initial data in the $\bfH^{2k}$ topology.
\end{itemize}

The last question we consider is that of continuation of the solutions, which is where our control norms 
are critically used. This is closely related to the energy estimates for our system:

\begin{theorem}\label{t:energy}
For each integer $k \geq 0$ there exists an energy functional $E^{2k}$ with the following properties:

a) Coercivity: as long as\footnote{Recall that we can harmlessly
assume $A$ small.} $A \ll 1$, we have 
\begin{equation}\label{e-equiv}
E^{2k}(r,v) \approx \| (r,v)\|_{\H^{2k}}^2.
\end{equation}

b) Energy estimates for solutions to \eqref{free-bd-euler}
\begin{equation}\label{en-growth}
\frac{d}{dt} E^{2k}(r,v) \lesssim_A B \| (r,v)\|_{\H^{2k}}^2.
\end{equation}
\end{theorem}
By Gronwall's inequality this implies the bound
\begin{equation}\label{e-continue}
    \| (r,v)(t)\|_{\H^{2k}}^2 \lesssim  e^{\int_0^T C(A)B(s)\, ds}  \| (r,v)(t)(0)\|_{\H^{2k}}^2.
\end{equation}

\begin{remark} 
These energies are constructed in an explicit fashion only for integer $k$. Nevertheless, as a consequence in our analysis in the last section of the paper, it follows  that bounds of the form \eqref{e-continue} hold also for all noninteger $k> 0$. However,  we do this using a mechanism which is akin to a paradifferential expansion, without constructing an explicit energy functional as provided by the above theorem in the integer case. 
\end{remark}

A consequence of the last result is the following continuation criteria for solutions to \eqref{free-bd-euler}, which holds regardless of whether $k$ is an integer:

\begin{theorem}\label{t:continuation}
Let $k$ be as in \eqref{k-range}. Then the $\bfH^{2k}$ solutions  to \eqref{free-bd-euler} given by Theorem~\ref{t:lwp} can be continued for as long as $A$ remains bounded
and $B \in L^1_t$.
\end{theorem}

Here we implicitly make a topological assumption
and exclude the possibility that two gas bubbles
at some point touch each other, or that the 
free boundary self-intersects. This latter 
possibility is prohibited at small scales by our result, but certainly not at large scales.

This result is consistent with the standard continuation results 
for quasilinear hyperbolic systems in the absence of a free the boundary.
But for the physical vacuum free boundary problem, this work is the \emph{first} where  anything close to such a continuation result has been proved.

\subsection{Historical comments} 

The study of the compressible Euler evolutions has a long history, and also considerable interest from the physical side.
Allowing for vacuum states introduces many added layers of difficulty to the problem, whose nature greatly depends on the behavior of the sound speed near the vacuum boundary. Within this realm, physical vacuum represents the natural boundary condition for compressible gasses.  Below we begin with a brief discussion of the broader context, and then we focus on the problem at hand.

\subsubsection{Compressible Euler flows} 
The compressible Euler equations are classically considered as a symmetric hyperbolic system, 
and as such, local well-posedness has long been known, see e.g. \cite{Kato-1975}, and also 
the Euler oriented analysis in \cite{Majda-1984}. The local solutions can be obtained using the energy method, and relying solely on the energy requires initial data local regularity 
$(\rho_0,v_0) \in H^s$ with $s > \dfrac{d}{2}+1$, with the continuation criteria
\[
\int_0^\infty \| \nabla (\rho,v) \|_{L^\infty} < \infty.
\]
By now it is known that these results can be improved by taking advantage of Strichartz estimates for wave equations. In the irrotational case, for instance, 
the result of \cite{Smith-Tataru} applies directly and yields the sharp local well-posedness result, for\footnote{Here $d = 3,4,5$.} $s > \dfrac{d+1}{2}$. In the rotational case,
it is not yet clear what would be the optimal condition on the vorticity which would allow 
for a similar improvement; see the results in \cite{2019arXiv190902550D} and \cite{2019arXiv191105038W}.

\subsubsection{Vacuum states in compressible Euler flows} 
Vacuum states correspond to allowing for the density to vanish in some regions. Here one should think of 
having a particle region  $\Omega_{t}$, and a vacuum region, separated by a moving \emph{free boundary}  
$\Gamma _{t} = \partial \Omega_{t}$. There are two major physical scenarios, distinguished by the boundary behaviour of 
the density $\rho$, or equivalently of  the sound speed $c_s$:

\begin{enumerate}
    \item Fluid flows, where the pressure is constant on the free boundary, describing a balance of forces, and the  density and implicitly the sound speed are assumed to have a nonzero, positive limit  
there. 

 \item Gas flows, where the density decay to zero near the free boundary; this is our main focus in this paper.
\end{enumerate}

Both are free boundary problems associated to compressible Euler, but their nature is very different in the two cases.
Fluid flows were considered in \cite{Christodoulou} and \cite{Lindblad-2005}, and also the incompressible limit was investigated in \cite{Lindblad-Luo}. 
\bigskip

Now we turn our attention to our present interest, namely the gas flows. Heuristically one distinguishes several potential scenarios when comparing the sound speed $c_s$ with the distance $d_{\Gamma}$ to the vacuum boundary. 

\bigskip
a) Rapid decay corresponds to
\[
c_s \lesssim d_{\Gamma_t}.\]
In this case the vacuum boundary evolves linearly, and internal waves cannot reach the boundary arbitrarily fast. Thus this geometry persists at least for a short time, and the local well-posedness problem can be even studied using the standard tools of symmetric hyperbolic systems; see for instance \cite{DiPerna-1983}, \cite{Chen-1997} and \cite{Lions-1998}, as 
well as the alternative approach in \cite{Makino-1986},\cite{Chemin-1990} and the one dimensional analysis in \cite{Liu-1997}. Thus this case cannot be thought of as a true free boundary problem.
Furthermore, after a finite time, the internal waves will reach the boundary \cite{Liu-1997}, and this geometry breaks down.

\bigskip
b) Slow decay, 
\[
c_s \gg d_{\Gamma_{t}}.
\]
This is where the problem indeed becomes a genuine free boundary problem, as internal waves can reach the boundary arbitrarily fast, and then the 
flow of the free boundary becomes strongly coupled with the internal flow.
One might think that there might be a range of possible decay rates, for instance like
\[
c_s \approx d^\beta_{\Gamma_t}, \qquad 0 < \beta < 1.
\]
However, both physical and mathematical considerations seem to indicate 
that among these there is a single stable decay rate, which corresponds 
to $\beta = \frac12$. This is commonly referred to as \emph{physical vacuum}.
The other values of $\beta$ are expected to be unstable, with the solutions 
instantly falling into the stable regime; but this is all a conjecture at this point, and likely there will be significant differences between the cases 
$\beta < \frac12$ and $\beta > \frac12$.

\bigskip

\subsubsection{The physical vacuum scenario}
We turn now our attention to the problem at hand, i.e. the physical vacuum scenario.  The easier one dimensional setting was considered first, in \cite{Coutand-2011}
followed by \cite{Jang-2009}. While some energy estimates are formally obtained in 
\cite{Coutand-2011} and a procedure to construct solutions is provided, the functional structure there does not provide a direct description of the initial data space. This issue is remedied in \cite{Jang-2009}, which first introduces the Lagrangian counterparts of the scale of spaces we are also using here, and provides both existence and uniqueness results in sufficiently regular spaces.

More recently, the three dimensional case was considered in several papers.
Energy estimates for $\kappa = 1$ were formally derived in \cite{Coutand-2010}. This was 
followed by an existence proof proposed in \cite{Coutand-2012}, which is based a parabolic regularization. However, the functional setting is similar to their prior one dimensional paper,  and some steps are merely claimed rather than proved; for instance the difference bound, which also, as stated, requires additional regularity for the solutions compared to the existence result.
Independently, \cite{Jang-2015} offers an alternative existence and uniqueness proof for arbitrary $\kappa > 0$, this time within the correct scale of weighted Sobolev spaces, using an iterative argument for the existence part, and with a different approach to the energy estimates.

All the results described above are in the Lagrangian setting, and aim to give 
existence and uniqueness results in sufficiently regular function spaces. 
In addition to the limitations mentioned above, no attempt is made to provide any continuous dependence results, nor to transfer the results to the physical, Eulerian framework.

By contrast, our results in the present paper are fully developed within the Eulerian setting, at low regularity, in all dimensions and for all $\kappa > 0$.
In this context we provide completely new arguments for existence, uniqueness,
and continuous dependence of the solutions on the initial data, i.e. a full 
well-posedness theory in the Hadamard sense. In addition we prove a family of sharp, scale invariant energy estimates, which in particular yield optimal continuation criteria at the level of $\|\nabla v\|_{L^\infty}$, consistent
with the well-known results for hyperbolic systems in the absence of the free boundary. Despite the fact that we only construct energy functionals 
corresponding only  to integer Sobolev spaces, we nevertheless are able to use 
these estimates in order to obtain energy estimates in fractional Sobolev spaces as well, nicely completing the theory up to the optimal Sobolev thresholds.

\subsection{An outline of the paper} 
The article has a modular  structure, where, for the essential part, only the main results of each section are used later.

\subsubsection{Function spaces and interpolation}
The starting point of our analysis, in the next section,  is to describe the appropriate
functional setting for our analysis, represented by the $\H^{2k}$ scale of weighted Sobolev spaces.
These are associated to the singular Riemannian metric \eqref{metric} under the sole assumption 
that the boundary $\Gamma_{{t}}$ is Lipschitz, with $r$ as a nondegenerate defining function. 
A similar scale of spaces was introduced in \cite{Jang-2015} in the Lagrangian setting, though under 
more regularity assumptions. However, since in the Eulerian setting the boundary is moving, 
the corresponding state space $\bfH^{2k}$ for $(r,v)$ is seen here akin to an infinite dimensional manifold. 

We remark on the dual  role of $r$, as a defining function of the boundary implicitly as a weight
on one hand, and as one of the dynamical variables on the other hand; for our low regularity analysis we carefully decouple these two roles, in order to avoid cumbersome boootstrap loops.

Interpolation plays a significant role in our study. First this occurs at the level of 
the $\H^{2k}$ scale of spaces, and it allows us to work with fractional Sobolev spaces
without having to directly prove energy estimates in the fractional setting, using expansions
which are akin to paradifferential ones but done at the level of the nonlinear flow.
Secondly, we also interpolate between the $\H^{2k}$ spaces and the pointwise bounds 
captured by our control parameters $A$ and $B$. It is this last tool which allows us 
to work at low regularity and to obtain sharp, scale invariant energy estimates.

\subsubsection{The linearized equation and transition operators} In Section~\ref{s:linearize}
we consider the linearized equation, which is modeled as a linear evolution in the 
time dependent weighted $L^2$ space $\H$. We view this as the main tool in the analysis of the nonlinear evolution, rather than the direct nonlinear energy estimates as in all prior work (except for \cite{Jang-2015}, to some extent). 
This later helps us not only to prove nonlinear energy estimates for single solutions, 
but also to compare different solutions, which is critical both for our uniqueness proof 
and for our construction of rough solutions as strong limits of smooth solutions.
We remark that at the level of the linearized variables $(s,w)$ there is no longer any boundary condition on the moving free boundary $\Gamma_t$; this is closely related to the prior comment about uncoupling the roles of $r$.

Next, using the linearized equation, we obtain the \emph{transition operators} $L_1$ and $L_2$, which 
act at the level of the two linearized variables $s$, respectively $w$, and should be though of as the degenerate elliptic leading spatial part of the wave evolution for $s$, respectively $\nabla \cdot w$.
We call them transition operators because they tie the successive spaces $\H^{2k}$ and $H^{2k+2}$
on our scale in a coercive, invertible fashion. These operators play a leading role in both the higher order energy estimates and in the regularization used for our construction of regular solutions.

\subsubsection{Difference estimates and the uniqueness result} The aim of Section~\ref{s:uniqueness}
is to construct a nonlinear difference functional which allows us to track the distance between two 
solutions roughly at the level of the $\H$ norm. This is akin to the difference bounds in a weaker topology which are common in the study of quasilinear problems.

This is one of the centerpieces of our analysis,
and to the best of our knowledge this is the first time such a construction was successfully carried
out in a free boundary setting. The fundamental difficulty is that we are seeking to not only compare functions on different domains, but also to track the evolution in time of this distance. This difficulty is 
translated into the nonlinear character of our difference functional; some delicate, careful choices
are made there, which ultimately allow us to propagate this distance forward in time.

\subsubsection{Higher order energy estimates}
The aim of Section~\ref{s:ee} is to establish energy estimates in integer index Sobolev spaces 
on our $\H^{2k}$ scale. We define the nonlinear energy functionals $E^{2k}$ using suitable vector fields applied to the equation. This energy has two components, a wave component and a transport component, which correspond to the heuristic (partial) decoupling of the evolution into a wave part for  $r$ and $\nabla \cdot v$ and a transport part for the vorticity $\omega$. 
Our proof of the energy estimates is split in a modular fashion into two parts, where we succesively 
 (i) prove the coercivity of our energy functional and (ii) track the time evolution of the energy.
 
 The coercivity bound is obtained inductively in $k$, using the transition operators $L_1$ and $L_2$
 as key tools. The main part of the proof of the propagation bound
happens at the level of the wave component, where we identify Alihnac style ``good variables" $(s_{2k},w_{2k})$, which are shown to solve the linearized equation modulo perturbative source terms.

Our energy functionals are to some extent the Eulerian counterparts of energies previously constructed 
in \cite{Coutand-2012}, \cite{Jang-2015} in the Lagrangian setting and at higher regularity. They are closer 
in style to \cite{Coutand-2012}, though the coercivity part is largely missing there and as a consequence
some of the functional setting is incomplete/incorrect.  The analysis in \cite{Jang-2015}, on the other hand,
corresponds to combining the two steps above together. This leads to a more comprehensive energy functional, where the coercivity part is relatively straightforward, but instead moves 
the difficulty to the propagation part, which becomes considerably more complex.

\subsubsection{Existence of regular solutions}
The aim of Section~\ref{s:reg-exist} is to prove the existence theorem in the context of regular solutions. The scheme we propose here is constructive, using a time discretization 
via an Euler type method to produce good approximate solutions. However a naive implementation 
of Euler's method looses derivatives; to rectify this we precede the Euler step by 
(i) a regularization on a suitable scale, and (ii) a separate transport part\footnote{This bit is 
optional but does simplify the analysis.}. The challenge is to control the energy growth at each step
of the way. This is more delicate for the regularization, which has has to be done carefully 
using the elliptic transition operators $L_1$ and $L_2$.

We note that our construction is very different from  any other approaches previously used in analyzing this problem; they all relied on parabolic regularizations. Our construction is simpler and more direct, though not  without interesting subtleties. It is also better  tailored to the physical structure of the equations, which makes this approach more robust and also successful in the relativistic counterpart of our problem. 

\subsubsection{Rough solutions as limits of regular solutions}
The last section of the paper aims to construct rough solutions as strong limits of smooth solutions.
This is achieved by considering a family of dyadic regularizations of the initial data, which 
generates corresponding smooth solutions. For these smooth solutions we control on one hand 
higher Sobolev norms $\H^{2N}$, using our energy estimates, and on the other hand the $L^2$ type distance between consecutive ones, which is at the level of the $\H$ norms.  
Combining the high and the low regularity bounds directly yields rapid convergence in all 
$\bfH^{2k_1}$ spaces below the desired threshold, i.e. for $k_1 < k$. To gain strong convergence 
in $\bfH^{2k}$ we use frequency envelopes to more accurately  control both the low and the  
high Sobolev norms above. This allows us to bound differences in the strong $\bfH^{2k}$ topology.
A similar argument yields continuous dependence of the solutions in terms of the initial data also in the strong topology,
as well as our main continuation result in Theorem~\ref{t:continuation}.

\subsection{Acknowledgements} 
  The first author was supported by a Luce Assistant Professorship, by the Sloan Foundation, and by an NSF CAREER grant DMS-1845037. The second author was supported by the NSF grant DMS-1800294 as well as by a Simons Investigator grant from the Simons Foundation. 
  
Both authors thank Marcelo Disconzi for introducing them to this class
of problems. In particular, the relativistic counterpart of this problem 
is considered jointly with him in forthcoming work.

\section{Function spaces}

The aim of this section is to introduce the main function spaces where
we will consider the free boundary problem for the compressible
gas. These are Sobolev type spaces of functions inside the gas domain
$\Omega_t$, with weights depending on $r$, or equivalently on the
distance to the free boundary. We begin with a more general discussion of 
weighted Sobolev spaces in Lipschitz domains, and then specialize 
to the function spaces that are needed in our problem.

\subsection{Weighted Sobolev spaces}
As a starting point, in a domain $\Omega \subset \R^d$ with Lipschitz boundary $\Gamma$
and nondegenerate defining function $r$ we introduce a two  parameter family of weighted
Sobolev spaces  (see \cite{triebel2010theory, triebel2010theory2} for a more general take on this): 

\begin{definition}
Let $\sigma > -\frac12$ and $j \geq 0$. Then the space $H^{j,\sigma}=H^{j,\sigma}(\Omega) $ is defined as the space of all distributions in 
$\Omega$ for which the following norm is finite:
\begin{equation}
\| f \|_{H^{j,\sigma}}^2 := \sum_{|\alpha| \leq j} \| r^\sigma \partial^\alpha f\|_{L^2}^2.
\end{equation}
\end{definition}
\noindent
By complex interpolation, one also defines corresponding fractional Sobolev spaces $H^{s,\sigma}$ for $s \geq 0$ and $\sigma > -\frac12$.
This yields a double family of interpolation spaces.
\medskip

Some comments are in order here:
\begin{itemize}
    \item At this point, all we assume about the geometry of the problem
    is that the boundary $\Gamma$ is Lipschitz, and that $r$ is 
    a non-degenerate defining function for $\Gamma$, i.e. proportional
    to the distance to $\Gamma$. Different choices for $r$ yield
    the same space with different but equivalent norms. Without any restriction in generality, we can assume that $r$ is Lipschitz continuous.

\item The requirement $\sigma > -\frac12$ corresponds to the fact that no 
vanishing assumptions on the boundary $\Gamma$ are made for any of the elements in our function spaces.

\item If $\sigma = 0$ then one recovers the classical Sobolev spaces
$H^{k,0} = H^k$.

\item If $j = 0$ these are weighted $L^2$ spaces, $H^{0,\sigma} = L^2(r^{2\sigma})$.
\end{itemize}

Next, we establish some key properties of these spaces.
First, we have the  Hardy type embeddings (see the book \cite{Kufner-2007} for a broader view):
\begin{lemma}\label{l:hardy}
Assume that $s_1 > s_2 \geq 0$ and  $\sigma_1 > \sigma_2 > -\frac12$ 
with $s_1-s_2 = \sigma_1-\sigma_2$.
Then we have 
\begin{equation}\label{eq:hardy}
H^{s_1,\sigma_1} \subset H^{s_2,\sigma_2}.
\end{equation}
\end{lemma}

\begin{proof}
By interpolation and reiteration it suffices to prove the result when 
$s_1-s_2=1$, both integers. Thus we will show that
\begin{equation}\label{eq:hardy-a}
 H^{j,\sigma} \subset H^{j-1,\sigma-1}, \qquad j \geq 1, \sigma > \frac12.
\end{equation}
 It suffices to prove the result in dimension $n=1$; then all the higher dimensions will follow by considering foliations of $\Omega $ with parallel one dimensional lines which are transversal to $\Gamma $.  
 
 Here $r$ is the distance function to the boundary of $\Omega $. 
 Setting $\Omega  = [0,\infty)$, $r$ is pointwise equivalent to $x$, and in particular gives
\[
\int_{\Omega }\left( r^{\sigma -1}\right)^2 \vert \partial_x^{j-1} f\vert^2\, dx \approx \int_{\Omega_t}\left( x^{\sigma -1}\right)^2 \vert \partial_x^{j-1} f\vert^2\, dx. 
\]
The inclusion follows from the following integration by parts
\[
\begin{aligned}
 \int_{\Omega }\left( x^{\sigma -1}\right)^2 \vert \partial_x^{j-1} f\vert^2\, dx  &= \int_{\Omega_t}\left( \frac{x^{2\sigma -1}}{2\sigma -1}\right)^{'} \vert \partial_x^{j-1} f\vert^2\, dx\\
 &=\left.\vert \partial_x^{j-1} f\vert^2\left( \frac{x^{2\sigma -1}}{2\sigma -1}\right)  \right|_{x\in \partial\Omega } -\frac{2}{2\sigma -1}\int_{\Omega } x^{2\sigma -1} \vert \partial_x^{j-1} f\vert \vert \partial_x^{j} f\vert\, dx.\\
 \end{aligned}
\]
The boundary term vanishes, and we can now apply Cauchy-Schwartz's inequality to obtain
\[
\Vert f\Vert _{H^{j-1, \sigma-1}}\leq \frac{2}{2\sigma -1}  \Vert f\Vert _{H^{j, \sigma}}.
\]

\end{proof}

As a corollary of the above lemma we have embeddings into standard Sobolev spaces:
\begin{lemma}\label{l:sobolev}
Assume that $\sigma > 0$ and $\sigma \leq j$.
Then we have 
\begin{equation} \label{eq:sobolev}
H^{j,\sigma} \subset H^{j-\sigma}.
\end{equation}
\end{lemma}

In particular, by standard Sobolev embeddings,  we also have Morrey type  embeddings into $C^s$ spaces:
\begin{lemma} \label{l:morrey}
We have 
\begin{equation}\label{morrey}
H^{j,\sigma}_r \subset C^{s}, \qquad   0 \leq s \leq  j- \sigma - \frac{d}2,
\end{equation}
where the equality can hold only if $ s$ is not an integer.
\end{lemma}

\subsection{Weighted Sobolev  norms for compressible Euler}
Our starting point here is the conserved energy for our problem, namely
\[
E(r,v) = \int_{\Omega_t}  r^{\frac{1-\kappa}{\kappa}} \left( r^2 + \frac{\kappa+1}2 r v^2\right)\,  dx .
\]

Even more importantly, in our study of  the linearized equation (see Section~\ref{s:linearize}) 
for linearized variables $(s, w)$  we use the weighted $L^2$ type energy functional
\[
E_{lin}(s,w) = \int _{\Omega_t}  r^{\frac{1-\kappa}{\kappa}} (  |s|^2 +  \kappa r |w|^2 )  \, dx.
\]
Based on this, we define our baseline space $\H$ with norm
\[
\|(s,w)\|_{\H}^2 = E_{lin}(s,w). 
\]
In terms of the $H^{s,\sigma}$ spaces discussed earlier, or weighted $L^2$ spaces, we have
\begin{equation}
\H =  H^{0, \frac{1-\kappa}{2\kappa}} \times H^{0,\frac{1}{2\kappa}}= L^2(r^{\frac{1-\kappa}{\kappa}})
\times L^2(r^{\frac{1}{\kappa}}).
\end{equation}

Next we define a suitable scale of higher order Sobolev spaces for our problem. To understand the balance between weights and derivatives
we consider the leading order operator, if we write the wave part of our system as a second order equation for $r$. At leading order this yields
the wave operator
\[
D_t^2  -  \kappa  r  \Delta,
\]
which is naturally associated with the Riemannian metric \eqref{metric} in $\Omega_t$.

So, to the above $L^2$ type space $\H$ we need to add Sobolev regularity based on powers of $r  \Delta$, or equivalently, relative to the metric $g$
defined above.  Hence we define the higher order Sobolev spaces $\H^{2k}$
\[
\H^{2k} :=  H^{2k, k+\frac{1-\kappa}{2\kappa}} \times H^{2k, k+\frac{1}{2\kappa}}, \qquad k \geq 0
\]
of pairs functions defined inside $\Omega_t$. These form a one parameter family of interpolation spaces. The $\H^{2k}$ spaces  have 
the obvious norm if $k$ is a nonnegative integer; for instance one can set
\begin{equation}\label{H2k-norm}
\|(s,w)\|_{\H^{2k}}^2 := \sum_{|\beta| \leq 2k }^{|\beta|-\alpha \leq k} 
\| r^\alpha \partial^\beta (s,w)\|_{\H}^2   , 
\end{equation}
where $\alpha$ is also restricted to nonnegative integers.

On the other hand, if $k$ is not an integer then the corresponding norms are Hilbertian norms defined by interpolation. Since in the Hilbertian case 
all interpolation methods yield the same result, for the $\H^{2k}$ norm we 
will use a characterization which is akin to a 
Littlewood-Paley decomposition, or to a discretization of the 
$J$ method of interpolation. Precisely, we have

\begin{lemma}\label{l:interp-spaces}
Let $0 < k < N$. Then $\H^{2k}$ can be defined as the space of distributions
$(s,v)$ which admit a representation
\begin{equation}
(s,w) = \sum_{l=0}^\infty (s_l,w_l)     
\end{equation}
with the property that the following norm is finite:
\begin{equation}\label{eq:interp-spaces}
\norm \{ (s_l,w_l)\} \norm^2 := \sum_{l = 0}^{\infty} 2^{2kl} \| (s_l,w_l) \|_{\H}^2 +   
2^{2l(k-N)} \| (s_l,w_l) \|_{\H^{2N}}^2,
\end{equation}
and with equivalent norm defined as
\begin{equation}
\|(s,w) \|^2_{\H^{2k}} := \inf \norm \{ (s_l,w_l)\} \norm^2,
\end{equation}
where the infimum is taken with respect to all representations as above.
\end{lemma}

\subsection{The state space \texorpdfstring{$\bfH^{2k}$}{}.}

As already mentioned in the introduction, the state space $\bfH^{2k}$
is defined for $k > k_0$ (i.e. above scaling) as the set of 
pairs of functions $(r,v)$ defined in a domain $\Omega_t$ in $\R^n$ with boundary $\Gamma_t$ with the following properties:

\begin{enumerate}[label=\roman*)]
\item Boundary regularity: $\Gamma_t$ is a Lipschitz surface.

\item Nondegeneracy: $r$ is a Lipschitz function in $\bar \Omega_t$, positive
inside $\Omega_t$ and vanishing simply on the boundary $\Gamma_t$.

\item Regularity: The functions $(r,v)$ belong to $\H^{2k}$.

\end{enumerate}

Since the domain $\Omega_t$ itself depends on the function $r$, one cannot think of $\bfH^{2k}$ as a linear space, but rather as an infinite dimensional 
manifold. As time varies in our evolution, so does the domain, so we 
are interested in allowing the domain to vary in $\bfH^{2k}$. However, 
describing a manifold structure for $\bfH^{2k}$ is beyond the purposes
of our present paper, particularly since the trajectories associated with our flow are merely expected to be continuous with values in $\bfH^{2k}$.
For this reason, here we will limit ourselves to defining a topology on $\bfH^{2k}$.

\begin{definition}\label{d:convergence}
A sequence $(r_n,v_n)$ converges to $(r,v)$ in $\bfH^{2k}$ if the following 
conditions are satisfied:
\begin{enumerate}[label=\roman*)]
\item Uniform nondegeneracy, $|\nabla r_n| \geq c > 0$.

\item Domain convergence, $\|r_n - r\|_{Lip} \to 0$.

\item Norm convergence: For each $\epsilon > 0$ there exist 
smooth functions $(\tr_n,\tv_n)$ in $\Omega_n$, respectively $(\tr,\tv)$
in $\Omega$ so that 
\[
(\tr_n,\tv_n) \to (\tr,\tv) \qquad \text{ in } C^\infty
\]
while 
\[
\| (\tr_n,\tv_n) - (r_n, v_n)\|_{\H^{2k}(\Omega_n)} \leq \epsilon.
\]
\end{enumerate}
\end{definition}

We remark that the last condition in particular provides both a uniform bound for the sequence $(r_n,v_n)$ in $\H^{2k}(\Omega_n)$ as well as an equicontinuity type property, which ensures that a nontrivial 
portion of their $\H^{2k}$ norms cannot concentrate on thinner layers
near the boundary. This is akin to the the conditions in the Kolmogorov-Riesz theorem for compact sets in $L^p$ spaces.

This definition will enable us to achieve two key properties of our flow:
\begin{itemize}
    \item Continuity of solutions $(r,v)$ as functions of $t$ with values
    in $\bfH^{2k}$.
    \item Continuous dependence of solutions $(r,v) \in \bfH^{2k}$ as functions of the initial data $(r_0,v_0) \in \bfH^{2k}$.
\end{itemize}

\subsubsection{Sobolev spaces and control norms}

An important threshold for our energy estimates corresponds to the 
uniform control parameters $A$ and $B$
given by \eqref{E:A_norm} and \eqref{E:B_norm}, 
respectively.
Of these $A$ is at scaling,
while $B$ is one half of a derivative above scaling. Thus, by 
Lemma~\ref{l:morrey} we will have the bounds
\begin{equation}\label{sobolev-A}
A \lesssim \| (r,v)\|_{\bfH^{2k}}, \qquad k > k_0 = \frac{d+1}2 +\frac{1}{2\kappa},    
\end{equation}
respectively 
\begin{equation}\label{sobolev-B}
B \lesssim \| (r,v)\|_{\bfH^{2k}}, \qquad k > k_0+\frac12 = \frac{d+2}2 +\frac{1}{2\kappa}  .  
\end{equation}

\subsubsection{The regularity of the free boundary}
Another property to consider for our flow, in dimension $n \geq 2$, is the regularity of the free boundary, as well as the regularity of the velocity restricted to the free boundary. This is given by trace theorems and the embedding \eqref{eq:sobolev}:

\begin{lemma}
Suppose that $(r,v) \in \bfH^{2k}$ and that $2k- \frac{1}{\kappa}$ is not an even integer. Then $\Gamma_t$ has regularity 
\[
\Gamma_t \in H^{k - \frac{1}{2\kappa}}.
\]
If in addition $\frac{1}{\kappa}$ is also not an odd integer then
the velocity restricted to $\Gamma_t$ has class 
\[
v \in H^{\frac{k-1}2 - \frac{1}{2\kappa}}(\Gamma_t).
\]
\end{lemma}

These properties are provided here only for comparison purposes,
and play no role in the sequel.
This is because in this problem one cannot view the evolution of the free 
boundary as a stand alone flow, not even at leading order.
In particular, a-priori this velocity does not suffice in order to
transport the regularity of $\Gamma_t$; instead the boundary evolution should be viewed as a part of the interior evolution.
Indeed, we will see that there
is some interesting cancellation arising from the structure of the
equations which facilitates this.

\subsection{Regularization and good kernels}

An important ingredient in our construction of solutions to our free boundary evolution
is to have good regularization operators associated to each dyadic frequency scale $2^h$, $h \geq 0$. These operators will need to accomplish  two distinct goals:
\begin{itemize}
    \item Fixed domain regularization. Given $(s,v) \in \H^{2k}(\Omega)$, construct 
    regularizations $(s^h,w^h)$ within the same $\H^{2j}(\Omega)$ scale of spaces. 
    \item State and domain regularization. Given $(r,v) \in \bfH^{2k}$, where the 
    first component defines a domain $\Omega$, construct regularizations
    $(r^h,v^h)$ within the  $\bfH^{2j}$ scale of spaces, where the regularized 
    domains $\Omega_h$ are defined by $r^h$, $\Omega_h := \{x \in \R^d\, | \, r^h(x) > 0 \}$.
\end{itemize}

We begin with some heuristic considerations and notations. Given a dyadic 
frequency scale $h$, our regularizations will need to select frequencies $\xi$
with the property that $r \xi^2 \lesssim 2^{2h}$, which would require 
kernels on the scale 
\[
\delta x \approx r^{\frac12} 2^{-h}.
\]
However, if we are too close to the boundary, i.e. $r \ll 2^{-2h}$,
then we run into trouble with the uncertainty principle, as we would have
$\delta x \gg r$. Because of this, we select the spatial scale 
$r \lesssim 2^{-2h}$ and the associated frequency scale $2^{2h}$
as cutoffs in this analysis. 

To describe this process, it is convenient to decompose a neighbourhood of the boundary
$\Gamma$ into boundary layers. We denote the dyadic boundary layer associated to the frequency $2^h$ by
\begin{equation}
\Omega^{[h]} = \{ x \in \Omega, \ r(x) \approx 2^{-2h}\},    
\end{equation}
the corresponding full boundary strip by
\begin{equation}
\Omega^{[> h]} = \{ x \in \Omega, \ r(x) \lesssim 2^{-2h}\}    ,
\end{equation}
and the corresponding interior region by
\begin{equation}
\Omega^{[< h]} = \{ x \in \Omega, \ r(x) \gtrsim 2^{-2h}\}    .
\end{equation}

We will also use dyadic enlargements of $\Omega$, denoted by 
\begin{equation}
\tOmega^{[h]} = \{ x \in \R^d, \ \ d(x,\Omega) \leq c 2^{-2h}\}   , 
\end{equation} 
with a small universal constant $c$, and 
\begin{equation}
\tOmega^{[>h]} = \{ x \in \R^d, \ \ d(x,\Gamma) \leq c 2^{-2h}\}   . 
\end{equation} 

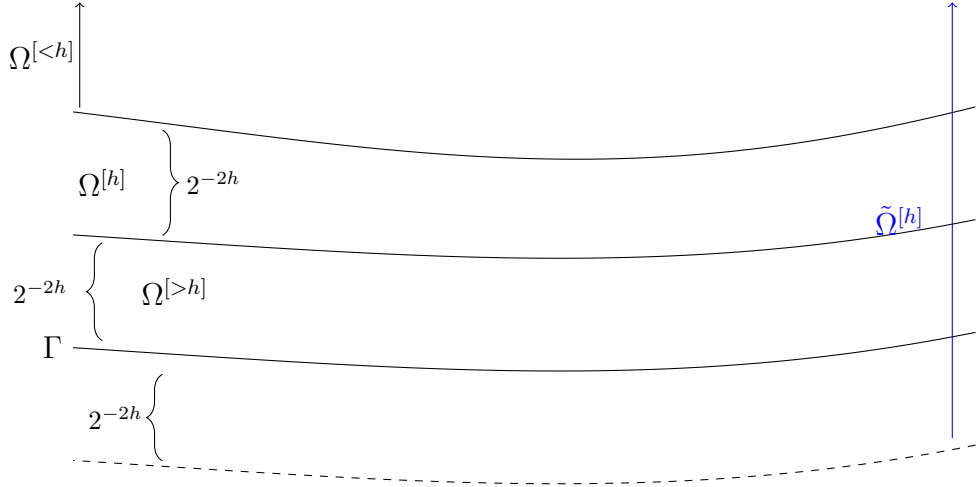
\begin{figure}
    \centering

\begin{tikzpicture}[use optics]

\draw[black] [samples=200,domain=0:12] plot (-\x,{-sin(\x r/4)*exp(-\x/16)+0.5 }) node[left, align=left] (script){};

\draw[black] [samples=200,domain=0:12] plot (-\x,{-sin(\x r/5)*exp(-\x/10)-1 }) node[left, align=left] (script){} ;

\draw[black] [samples=200,domain=0:12] plot (-\x,{-sin(\x r/5)*exp(-\x/10)-2.5 }) node[left, align=left] (script){$\Gamma$};

\draw[black, dashed] [samples=200,domain=0:12] plot (-\x,{-sin(\x r/5)*exp(-\x/10)-4 }) node[left, align=left] (script){} ;

\draw [decorate,decoration={brace,amplitude=6pt,mirror},xshift=-60pt,yshift=-60pt]
(-9.5,.8) -- (-9.5,-.5) node [black,midway,xshift=-0.5cm] {\footnotesize $\, \ 2^{-2h}\ \qquad  $}
node [black, midway, xshift=1cm] { $\, \Omega^{[>h]}\ $};

\draw [decorate,decoration={brace,amplitude=6pt,mirror},xshift=-60pt,yshift=-60pt]
(-8.7,.9) -- (-8.7,2.3) node [black,midway,xshift=0.5cm] {\footnotesize $\, \quad 2^{-2h\, } $} node [black,midway,xshift=-1cm] { $\, \quad \Omega^{[h]}\, $} ;

\draw [decorate,decoration={brace,amplitude=6pt,mirror},xshift=-60pt,
yshift=-60pt]
(-8.7,-.95) -- (-8.7,-2.1) node [black,midway,xshift=-.8cm] {\footnotesize $\, \quad 2^{-2h\, } $} ;

\draw [->={at=(-9.8,4)},decorate, decoration={amplitude=6pt,mirror},xshift=-60pt,yshift=-60pt](-9.8,2.6) -- (-9.8,4) node [black,midway,xshift=-.5cm] {$\Omega^{[<h]}$} ;

\draw [<-={at=(1.8,4)},blue, decorate,decoration={amplitude=6pt,mirror}, xshift=-60pt, yshift=-60pt](1.8,4) -- (1.8,-1.8) node [blue,midway,xshift=-.7cm] {$\tilde{\Omega}^{[h]}$} ;
\end{tikzpicture}
\label{layer.pic}
    \caption{Boundary layers associated to frequency scale $2^{h}$.}
    \label{fig:my_label}
\end{figure}

Given a domain $\Omega$ with a nondegenerate Lipschitz defining 
function $r$, and $(s,w)$ functions in $\Omega$, we will define 
regularizations $(s^h,w^h)$ associated to the $h$ dyadic scale 
using smooth kernels $K^h$,
\[
(s^h,w^h)(x) = \Psi^h(s,v) : = \int K^h(x,y) (s,w)(y) \, dy.
\]
The heuristic discussion above leads to the following notion
of \emph{good kernels}:

\begin{definition}
The family of kernels $K^h$ are called good regularization kernels if
the following properties are satisfied:

\begin{enumerate}[label=\roman*)]
\item Domain and localization:
\begin{equation}\label{kernel-domain}
K^h : \tOmega^{[h]} \times \Omega \to \R    
\end{equation}
with support properties
\begin{equation}\label{kernel-shift}
\text{supp } K^h \subset \{ (x,y) \in \tOmega^{[h]} \times \Omega^{< h}, \ \ 
|x-y| \lesssim  \delta y^h := 2^{-2h}  + 2^{-h} r(y)^{\frac12} \} .
\end{equation}

\item Size and regularity  
\begin{equation}\label{kernel-size}
| \partial_x^\alpha \partial_y^\beta K^h(x,y)| \lesssim (2^{-2h}  + 2^{-h} r(y)^{\frac12})^{-N - |\alpha| - |\beta|}, \qquad |\alpha|+|\beta| \leq 4N,
\end{equation}
where $N$ is large enough.

\item
Approximate identity,
\begin{equation}\label{kernel-moments0}
\int K^h(x,y)\, dy = 1,
\end{equation}
\begin{equation}\label{kernel-moments}
\int K^h(x,y) (x-y)^\alpha \, dy = 0, \qquad 1 \leq |\alpha| \leq 2N.
\end{equation}
\end{enumerate}

\end{definition}

Notably, the first property will allow us to define the regularizations $(s^h,w^h)$
in the extended domain $\tOmega^{[h]}$, with dyadic mapping properties as follows:
\begin{itemize}
    \item For $j < h$, the regularizations $(s^h,w^h)$ in $\Omega^{[j]}$ are determined by 
$(s,w)$ also in $\Omega^{[j]}$. 
    \item For the $h$ layers, the regularizations $(s^h,w^h)$ in $\tOmega^{[>h]}$ 
    are determined by  $(s,w)$ only in  $\Omega^{[h]}$.
\end{itemize}
Thus our regularization operators use their inputs only outside the $2^{-2h}$ 
boundary layer, but provide outputs in a $2^{-2h}$ enlargement of the domain $\Omega$.
Such a property is critical in order to have good domain regularization properties. 

The role of the third property on the other hand is to ensure that polynomials of 
sufficiently small degree are reproduced by our regularizations. This will later 
provide good low frequency bounds for differences of successive regularizations.

Regularization kernels with these properties ca be easily constructed:

\begin{lemma}
Good regularization kernels exist.
\end{lemma}
\begin{proof}
\bigskip
We outline the steps in the kernel construction, leaving the details for the reader:
\medskip

a) We consider a unit vector $e$ which is uniformly transversal
to the boundary, outward oriented. Such an $e$ can be chosen locally, and kernels
constructed based on a local choice of $e$ can be assembled together 
using a partition of unity in the first variable.

\medskip

b) Given such an $e$, we consider a smooth bump function $\phi$ with properties as follows:
\begin{itemize}
    \item the support of $\phi$ is such that
\[
\supp \phi \subset B(e,\delta), \qquad \delta \ll 1,
\]
\item its average is $1$:
\[
\int \phi(x) \, dx = 1,
\]

\item and, it has zero moments
\[
\int x^\alpha \phi(x)\,  dx = 0, \qquad 1 \leq |\alpha| \lesssim N .
\]
\end{itemize}

c) For each dyadic scale $m$ we consider a shifted regularizing kernel 
\[
K^m_0(x-y) = 2^{2md} \phi(2^{2m}(x-y))
\]
on the $2^{-2m}$ scale, which is accurate to any order. 

Correspondingly we also consider a partition of unity in $\Omega$,
\[
1 = \sum_{m=0}^\infty \chi_m ,
\]
where the functions $\chi_m$ select the region $\Omega^{[m]}$
and are smooth on the $2^{-2m}$ scale. Given a fixed dyadic scale $h$,
we adapt this partition of unity to $h$,
\[
1= \chi_{> h}+  \sum_{m=0}^h \chi_m ,
\]
where the first term $\chi_{>h}$ can be extended by $1$ to the exterior of $\Omega$.

d) We define the regularization kernels
\[
K^h(x,y) := \chi_{>h}(x) K^h_0(x-y) + \sum_{m=0}^{ h } \chi_{m}(x) K^m_0(x-y),
\]
which are still accurate to any order. It is easily verified that these kernels have the  desired properties.

\end{proof}

Next we prove bounds for our regularizations in $\H^{2k}$ spaces:

\begin{proposition}\label{p:reg}
The following estimates hold for good regularization kernels 
whenever $r_1$ is a nondegenerate defining function with $|r - r_1| \ll 2^{-2h}$:

a) Regularization bound:
\begin{equation}\label{reg+}
\| \Psi^h (s,w) \|_{\H^{2k+2j}_{r_1}} \lesssim 2^{2jh} \| (s,w) \|_{\H^{2k}_r}, \qquad \qquad j \geq 0  ,
\end{equation}

b) Difference bound:
\begin{equation}\label{reg-}
\| (\Psi^{h+1} - \Psi^h) (s,w) \|_{\H^{2k+2j}_{r_1}} \lesssim 2^{2jh} \| (s,w) \|_{\H^{2k}_r}, \qquad -k \leq j \leq 0   ,
\end{equation}

c) Error bound
\begin{equation}\label{reg:err}
\| (I - \Psi^h) (s,w) \|_{\H^{2k+2j}_{r}} \lesssim 2^{2jh} \| (s,w) \|_{\H^{2k}_r}, \qquad -k \leq j \leq 0.
\end{equation}
\end{proposition}

Here we recall that the regularized functions $\Psi^h (s,v)$ are defined on the larger domain 
$\tOmega^{[h]}$. This is what allows us to measure them with respect to a perturbed domain 
$\Omega_1 = \{ r_1 > 0\}$ as long as the two boundaries are within $O(2^{-2h})$ of each other.

\begin{proof}
By interpolation we can assume that $k$ and $j$ are both integers.
Because of the support properties of $K_h$, we can prove the desired estimate separately in each  boundary layer $\Omega^{[l]}$, for $0 \leq l \leq h$, and then separately for $\tOmega^{[>h]}$.
For instance in the case of \eqref{reg+} we will show that
\begin{equation}
\| \Psi^h (s,w) \|_{\H^{2k+2j}(\Omega^{[l]})} \lesssim 2^{2hj} \| (s,w)\|_{\H^{2k}(\Omega^{[l]})} ,
\end{equation}
where the domain restricted norms are interpreted as the square integral of the appropriate
quantities over the restricted domains\footnote{In a standard fashion, we also need to allow the domain on the right to be a slight enlargement of the domain on the left.}.

 The above localization allows us to fix the $r$ dependent localization scale $\dx = 2^{-(h+l)}$ for $\Psi^h$, which becomes akin to a scaling parameter. Even better, we can localize further to a ball $B_\dx \subset \Omega^{[l]}$ and show that
\[
\| \Psi^h (s,w) \|_{\H^{2k+2j}(B_{\dx})} \lesssim 2^{2jh} \|(s,w)\|_{\H^{2k}(2B_{\dx})}.
\]
Consider one component of the norm on the left, namely the maximal one, and show that
\begin{equation}\label{Lk-bd}
\| r_1^{k+j} \partial^{2(k+j)} \Psi^h (s,w) \|_{\H_{r_1}(B_\dx)} \lesssim \| r^{k} \partial^{2k}  (s,w) \|_{\H_r(B_\dx)}.
\end{equation}
 
To avoid distracting technicalities, consider first the case
$l < h$, where the weights are constant and can be dropped. Then the above inequality becomes
\begin{equation}\label{Lk-bd-a}
\|\partial^{2(k+j)} \Psi^h u \|_{L^2(B_{\dx})} \lesssim 2^{2j(h+l)} \|  \partial^{2k}  u \|_{L^2(2B_\dx)}.
\end{equation}
The difficulty here is that we only have control over the derivatives of $u$ (here $u$ can be replaced by either $s$ or $w$). We can bypass this difficulty using  (a higher order version of) Poincare's inequality in $B_\dr$, which allows us to find a polynomial $P$ of degree $2k-1$ so that 
\[
\| \partial^{b}(u - P)\|_{L^2 (B_\dx)}\lesssim \dx^{2k-b} \| \partial^{2k}  u \|_{L^2(B_{\dx})}, 
\qquad 0 \leq b < 2k.
\]
The property \eqref{kernel-moments} shows that $K^h P = P$, therefore in \eqref{Lk-bd-a} we can replace $u$ by $u-P$, for which  we have better control of the lower Sobolev norms. Then the estimate \eqref{Lk-bd-a} easily follows.

Minor adjustments to this argument are needed in $\Omega^{[h]}$. Then 
$\dx \approx 2^{-2h}$, and we can still freeze $r$ in the input region to 
$r = 2^{-2h}$.  On the other hand in the output region we have 
$r_1 \lesssim 2^{-2h}$, which still allows us to drop the $r_1^{k}$ weight.
The Poincare inequality still applies. The only difference is that
the weight  in the $\H$ norm on the left might be singular.  However, this weight is nevertheless square integrable near the boundary, which suffices
due to the fact that in effect in $B_{\dx}$ we can obtain pointwise control  for 
$\partial^{2k+2j} \Psi^h (u-P)$. 

Now we consider the case (b). There the same localization applies, and the main
difference in the proof is that now for a polynomial $P$ of degree at most $2N$ we have 
\[
(\Psi^{h+1} - \Psi^h) P = 0.
\]
This in turn allows us to also substitute $u$ by $u-P$ in \eqref{Lk-bd-a} when $j$ 
is negative. The rest of the argument is identical.

Finally, for the bound \eqref{reg:err} we simply add up \eqref{reg-} for scales $> h$.
\end{proof}

Given a rough state $(r,v) \in \bfH^{2k}$, we can use the above Lemma to construct 
a regularized state $(r^h,v^h)$ as follows:
\begin{enumerate}[label=\alph*)]
\item We define the regularized functions $(r^{h},v^h)$ in the larger domain 
$\tOmega^{[h]}$ by 
\[
(r^{h},v^h) = \Psi^h (r,v).
\]

\item We restrict $(r^{h},v^h)$ to the set\footnote{Here and below we use subscripts for $\Omega$ as in $\Omega_{*} = \{ r^* > 0\}$ to indicate the domain associated to a function $r^*$, and the superscripts $\Omega^{[*]}$ to select various boundary layers.} 
$\Omega_{h} := \{ r^h > 0\}$.
\end{enumerate}
Such a  strategy works provided that the domain $\tOmega^{[h]}$ is large enough in order 
to allow $r^h$ to transition to negative values before reaching the boundary of its domain.
We will see that this is indeed true provided that $k$ is above the scaling exponent $k_0$.
Our main result is stated below. For better accuracy, we  use 
the language of frequency envelopes to state it.

\begin{proposition}\label{p:reg-state}
Assume that $k > k_0$. Then given a state $(r,v) \in \bfH^{2k}$, there 
exists  a family of regularizations $(r^{h},v^h) \in \bfH^{2k}$,
so that the following properties hold for 
a slowly varying frequency envelope $c_h \in \ell^2$ which satisfies
\begin{equation}\label{app-fe}
\|c_h\|_{\ell^2} \lesssim_A    \| (r,v) \|_{\bfH^{2k}}   .
\end{equation}

\begin{enumerate}[label=\roman*)]
\item  Good approximation, 
\begin{equation}\label{app-point}
   (r^{h},v^h) \to (r,v) \quad \text{  in }  C^1 \times C^\frac12 
   \quad \text{ as } h \to \infty,
\end{equation}
and 
\begin{equation}\label{app-point0}
   \|r^{h} - r\|_{L^\infty(\Omega)} \lesssim 2^{-2(k-k_0+1) h }.
\end{equation}

\item Uniform bound,
\begin{equation}\label{app-uniform}
\| (r^h,v^h) \|_{\bfH^{2k}}   \lesssim_{A} \| (r,v) \|_{\bfH^{2k}}  .
\end{equation}

\item  Higher regularity
\begin{equation}\label{app-high}
\| (r^h,v^h) \|_{\bfH^{2k+2j}_h}   \lesssim  2^{2hj}  c_h, \qquad j > 0.
\end{equation}

\item Low frequency difference bound:
\begin{equation}\label{app-diff}
\| (r^{h+1},v^{h+1}) - (r^h,v^h) \|_{\H_{\tilde r}}   \lesssim  2^{-2hk}  c_h \qquad 
|\tilde r - r| \ll 2^{-2h}. 
\end{equation}
\end{enumerate}
\end{proposition}

\begin{proof}
To start with, we will assume that $(r^h,v^h)$ are defined in the larger set $\tOmega^{[h]}$ using good regularization kernels $K_h$,
\[
(r^h,v^h) = \Psi^h(r,v).
\]
By Sobolev embeddings we know that 
\[
(r,v) \in C^{1+k-k_0} \times C^{\frac12+k-k_0} (\Omega).
\]
This easily implies the uniform bound for $(r^h,v^h)$ in $C^{1} \times C^{\frac12}(\tOmega^{[h]})$,
as well as the convergence in the same topology to $(r,v)$ in $\Omega$.
It also implies the pointwise bound  \eqref{app-point0}. This in turn shows that 
on the boundary $\Gamma$ we have $|r^h| \lesssim 2^{-2(k-k_0+1) h}$, therefore the zero set 
$\Gamma_h= \{r^h = 0\} $ is within distance $2^{-2(k-k_0+1) h}$ from $\Gamma$, and thus within
$\tOmega^{[h]}$. This ensures that $(r^h,v^h)$ restricted to $\Omega_h = \{r^h > 0\}$
is a well defined state.

Next we consider the bound \eqref{app-uniform}. In view of the difference bound 
\eqref{app-point0}, this is a consequence  of \eqref{reg+} with $r_1 = r^h$ and $j=0$.

It remains to prove \eqref{app-high} and \eqref{app-diff}. If we were to replace 
$c_h$ by $1$ on the right, this would also follow from Proposition~\ref{p:reg}.
To gain the extra decay associated with a frequency envelope,
for the functions $(r,v)$ we will use the interpolation space 
representation given by Lemma~\ref{l:interp-spaces} with $N$ sufficiently large, 
\begin{equation}\label{rv-interp}
(r,v) = \sum_{l=0}^\infty (s_l, w_l),
\end{equation}
for which the norm in \eqref{eq:interp-spaces} is finite.
Accordingly, we can choose a slowly varying
frequency envelope $c_l$ so that
\begin{equation}\label{rv-interp-est}
  \| (s_l,w_l) \|_{\H} \leq 2^{-2lk} c_l, \qquad    
\| (s_l,w_l) \|_{\H^{2N}} \leq 2^{2l(N-k)} c_l.
\end{equation}
with
\[
\sum c_l^2 \lesssim \| (r,v)\|_{\bfH^{2k}}^2.
\]
The frequency envelope $c_l$ above is the one we will use 
in the Proposition. The property \eqref{app-fe} is then automatically satisfied.

\bigskip

iii) \emph{Proof of \eqref{app-high}.}
Our starting point is again the  decomposition \eqref{rv-interp}-\eqref{rv-interp-est} for $(r,v)$, but now we 
separate the contributions of $l \leq k$ and $l > k$.

\medskip

\emph{ a) Low frequency components $l < k$.} Using the $\Psi^h$ bounds in Proposition~\ref{p:reg}, the bounds for $(r_l, v_{l})$ carry over to $\Psi^h(r_l, v_{l})$, namely 
\[
 \| \Psi^h(s_l,w_l) \|_{\H} \leq 2^{-2lk} c_l, \qquad    
\| \Psi^h (s_l,w_l) \|_{\H^{2N}} \leq 2^{2l(N-k)} c_l.
\]
Then by interpolation we have
\begin{equation}\label{Kh-low}
\| \Psi^h (s_l,w_l) \|_{\H^{2k+2j}} \lesssim 2^{2lj} c_l.
\end{equation}

\medskip

\emph{ b) High frequency components $l \geq k$.}
Here we discard the $\H^{2N}$ bound, and instead estimate directly 
\begin{equation}\label{Kh-hi}
\| K_h (s_l,w_l) \|_{\H^{2k+2j}} \lesssim 2^{2h(j+k)}
\| (s_l,w_l) \|_{\H} \lesssim 2^{2jh} 2^{2(h-l)j} c_l.
\end{equation}

Combining \eqref{Kh-low} and \eqref{Kh-hi}, we obtain
\[
\| K_h(r,v)\|_{\H^{2k+2j}} \lesssim \sum_{l \leq h} 2^{2lj} c_l
+ \sum_{l > h} 2^{2jh} 2^{2(h-l)j} c_l \lesssim c_h
\]
as needed.

\bigskip

v) \emph{Proof of \eqref{app-diff}.}
We follow the same strategy as above, where we still can use all the $\Psi^h$ bounds
in Proposition~\ref{p:reg}, but with the difference that now we also have access to 
the difference bound in \eqref{reg-}.

Starting with the  decomposition \eqref{rv-interp}-\eqref{rv-interp-est} for $(r,v)$, we observe that the $\H$ bound for $(r_l,v_l)$ suffices in the high frequency case $l \geq h$. It remains to consider the low frequency case $l < h$, where we will have to rely instead on the $\H^{2N}$ norm. Precisely, by \eqref{reg-} we have
\begin{equation}\label{Lk-bd5}
\| \partial_h K^h(r_l,v_l)\|_{\H} \lesssim 2^{-2Nh} \| (r_l,v_l)\|_{\H^{2N}},
\end{equation}
which again, combined with \eqref{rv-interp-est}, suffices after dyadic $l$ summation. 
\end{proof}

\bigskip

\subsection{Interpolation inequalities}
Next we consider $L^p$  interpolation type inequalities, which are critical in order 
to prove our sharp, scale invariant  energy estimates.

For clarity and later use we provide a more general  interpolation result. 
Our main  result, which applies in any Lipschitz domain $\Omega$ with a nondegenerate defining function $r$, is as follows:
\begin{proposition}\label{p:interpolation-g}
Let $\sigma_0, \sigma_m \in \mathbb{R}$ and $1\leq p_0,p_m\leq \infty$. Define 
\begin{equation}
\label{interp}
\theta_j =   \frac{j}{m}, \qquad \frac{1}{p_j} = \frac{1-\theta_j}{p_0}+\frac{\theta_j}{p_m}, \qquad \sigma_j = 
\sigma_0(1-\theta_j) + \sigma _m\theta_j,
\end{equation}
and assume that
\begin{equation}
m - \sigma_m - d\left( \frac{1}{p_m}-\frac{1}{p_0} \right) > -\sigma_0, \qquad \sigma_j >-\frac{1}{p_j}.
\end{equation}
Then  for $0 < j < m$ we have
\begin{equation}
\label{interp-est}
\| r^{\sigma_j } \partial^j f \|_{L^{p_j}} \lesssim \|  r^{\sigma_0}  f \|_{L^{p_0}}^{1-\theta_j} \|r^{\sigma_m}\partial^m f \|_{L^{p_m}}^{\theta_j} . 
\end{equation}
\end{proposition}

\begin{remark}
One particular case of the above proposition which will be used later is when  $p_0=p_1=p_2=2$, with the corresponding relation in between the exponents of the  $r^{\sigma_j}$ weights.
\end{remark} 

As the objective here is to interpolate
between the $L^2$ type $\H^{m,\sigma}$ norm and  $L^\infty$ bounds, we will need the following straightforward consequence of Proposition~\ref{p:interpolation-g}:

\begin{proposition}\label{p:interpolation}
Let  $\sigma_m > -\frac12$ and 
\begin{equation}
m - \sigma_m - \frac{d}2 > 0.
\end{equation}
Define
\begin{equation}
\label{interp-2-infty}
\theta_j =   \frac{j}{m}, \qquad \frac{1}{p_j} = \frac{\theta_j}2, \qquad \sigma_j =  \sigma_m \theta_j.
\end{equation}
Then  for $0 < j < m$ we have
\begin{equation}
\| r^{\sigma_j } \partial^j f \|_{L^{p_j}} \lesssim \|   f \|_{L^\infty}^{1-\theta_j} \|r^{\sigma_m}\partial^{m}f \|_{L^{2}}^{\theta_j}  .
\end{equation}
\end{proposition}

We will also need the following two variations of Proposition~\ref{p:interpolation}:

\begin{proposition}
\label{p:interpolation-c}
Let  $\sigma_m >-\frac{1}{2}$ and 
\[
m-\frac{1}{2}-\sigma_m -\frac{d}{2}>0.
\]
Define 
\[
\sigma_j=\sigma_m\theta_j,\quad \theta_j=\frac{2j-1}{2m-1}, \quad \frac{1}{p_j}=\frac{\theta_j}{2}.
\]
Then for $0<j<m$ we have
\[
\Vert r^{\sigma_j}\partial ^j f\Vert_{L^{p_j}}\lesssim \Vert f\Vert^{1-\theta_j}_{\dot{C}^{\frac{1}{2}}}  \Vert r^{\sigma_m}\partial^m f \Vert^{\theta_j}_{L^{2}}
\]
\end{proposition}

respectively

\begin{proposition}
\label{p:interpolation-d}
Let  $\sigma_m >\frac{m-2}{2}$ and 
\[
m-\frac{1}{2}-\sigma_m -\frac{d}{2}>0.
\]
Define 
\[
\sigma_j=\sigma_m\theta_j-\frac12(1-\theta_j),\quad \theta_j=\frac{j}{m}, \quad \frac{1}{p_j}=\frac{\theta_j}{2}.
\]
Then for $0<j<m$ we have
\[
\Vert r^{\sigma_j}\partial ^j f\Vert_{L^{p_j}}\lesssim \Vert f\Vert^{1-\theta_j}_{\tC}  \Vert r^{\sigma_m}\partial^m f \Vert^{\theta_j}_{L^{2}}
\]
\end{proposition}
Here the role of the lower bound on $\sigma_m$ is to ensure that $\sigma_j > - \frac{1}{p_j}$ for all intermediate $j$,
where the $j = 1$ constraint is the strongest.

We will use the last two propositions  for $(r,v)$, where the pointwise bound comes from the control norms $A$ and  $B$. 

\begin{proof} [Proof of Proposition \ref{p:interpolation-g}] We begin with several simplifications. First we note that it suffices to prove the case $m=2$ and $j=1$. 
Then the general case follows by reiteration. Indeed the case $m=2$ allows us to compare any three consecutive norms
\[
\Vert r^{\sigma_{j+1}}\partial^{j+1}f\Vert_{L^{p_{j+1}}}\leq \Vert r^{\sigma_{j}}\partial^{j}f\Vert^{\frac{1}{2}}_{L^{p_{j}}}\Vert r^{\sigma_{j+2}}\partial^{j+2}f\Vert^\frac12_{L^{p_{j+2}}}.
\]
and then the main estimates \eqref{interp-est} follows from combining the above bounds. 

A second simplification is to observe that we can also reduce the problem to the 
one dimensional case, which we state in the following lemma:
\begin{lemma}
\label{l:interp}
 Let $p_j \in [1, \infty]$, and $\sigma_j \in \R$ with $j=\overline{0,2}$,
 so that
\[
\frac{1}{p_2} +\frac{1}{p_0}=\frac{2}{p_1}, \quad  \mbox{ and } \quad  \sigma_0+\sigma_2=2\sigma_1,
\]
and with 
\[
2 - d\left( \frac{1}{p_2}-\frac{1}{p_0} \right) > \sigma_2 -\sigma_0, \qquad \sigma_1 >-\frac{1}{p_1}.
\]
Then the following inequality holds
 \begin{equation}
\label{interp-1d-sister}
\| x^{\sigma_1 } \partial f \|_{L^{p_1}} \lesssim \|  x^{\sigma_0}  f \|_{L^{p_0}}^{\frac{1}{2}} \|x^{\sigma_2}\partial^2 f \|_{L^{p_2}}^{\frac{1}{2}},
\end{equation}
\end{lemma}
To see that the $n$-dimensional case reduces to the one dimensional case,
we consider a constant vector field $X$ which is transversal to the boundary,
apply \eqref{interp-1d-sister}, with $x$ replaced by $r$,
on every $X$ line $\Omega_y$ in $\Omega$, where $y$ denotes the transversal direction.
We raise it to the power $p$ and integrate in $y$. This yields
\[
\begin{split}
\| r^{\sigma_1 } X f \|_{L^{p_1}(\Omega)}^{p_1} \lesssim & \ 
\int \|  r^{\sigma_0}  f \|_{L^{p_0}(\Omega_y)}^{\frac{p_1}{2}} \|r^{\sigma_2}X^2 f \|_{L^{p_2}(\Omega_y )}^{\frac{p_1}{2}} dy
\\
\lesssim & \ 
 \|  r^{\sigma_0}  f \|_{L^{p_0}(\Omega)}^{\frac{p_1}{2}} \|r^{\sigma_2}X^2 f \|_{L^{p_2}(\Omega )}^{\frac{p_1}{2}} ,
\end{split}
\]
where at the second step we have used H\"older's inequality.
The full $n$-dimensional bound is obtained by applying the above estimate 
for a finite number of vector fields $X$ which (i) are transversal to the boundary
and (ii) span $\R^n$. It remains to prove the last Lemma~\ref{l:interp}:

\begin{proof}[Proof of Lemma~\ref{l:interp}]
This interpolation inequality is a weighted Gagliardo–Nirenberg-Sobolev inequality, see\cite{nirenberg}. 
One main ingredient in the original proof  given in \cite{nirenberg} 
for the unweighted case, is the following inequality due to P. Ungar: 
\begin{proposition} \label{p:ungar} 
On an interval $I$, whose length  is denoted by $\lambda$, one has
\[
\Vert  u_x \Vert_{L^{p_1}(I)}^{p_1}\lesssim \lambda^{1+p_1-\frac{p_1}{p_2}}\Vert u_{xx}\Vert ^{p_1}_{L^{p_2}(I)}  +\lambda^{-(1+p_1-\frac{p_1}{p_2})}\Vert u\Vert ^{p_1}_{L^{p_0}(I)},
\]
where  $p_j \in [1, \infty]$, $j=\overline{0,2}$
\end{proposition} 
The heuristic interpretation of Proposition \eqref{p:ungar} is that  the 
average of the  first derivative of a function is controlled by its pointwise values,
and its variation is controlled by its second derivative.
 This observation yields the balance between the parameters $m$, $\sigma_0$, $\sigma_1$ and  $\sigma_2$ in Lemma \ref{interp-1d-sister}. We will use
 the same result here to prove \eqref{interp-1d-sister}.


The first step is to use a dyadic spatial  decomposition of $\mathbb{R}^+$, such that the interval $I$ in Proposition \ref{p:ungar} is fully contained in a generic interval $[r,2r]$, where $r=2^k$, and $k\in \mathbb{Z}$. Using Proposition \eqref{p:ungar}, we have
\[
\begin{split}
 r^{\sigma_1p_1}\Vert\partial f\Vert_{L^{p_1}(I)}^{p_1}&=\| x^{\sigma_1}  \partial f \|^{p_1}_{L^{p_1}(I)} \\
 &\lesssim  r^{p_1(\sigma_1-\sigma_2)}\lambda ^{1+p_1-\frac{p_1}{p_2}}\|x^{\sigma_2}\partial^2 f \|^{p_1}_{L^{p_2}(I)}  
 + r^{(\sigma_1-\sigma_0)p_1}\lambda^{-(1+p_1-\frac{p_1}{p_0})}\|   x^{\sigma_0}f \|_{L^{p_0}(I)}^{p_1}.
 \end{split}
\]
To get from this inequality to \eqref{interp-1d-sister} it would be convenient to know that the last two terms in the above inequality are comparable in size. One can try to achieve this by increasing the size of the interval $I$ until this is true.  The difficulty is when this it cannot be done without going past the dyadic interval size. So the natural strategy is to consider the dyadic decomposition of interval $[0, \infty]$ and compare the $L^{p_2} $ and $L^{p_0}$ norms in each of these dyadic intervals. 

If on any such dyadic interval we get
\begin{equation}\label{comp-k}
r^{p_1(\sigma_1-\sigma_2) +1+p_1-\frac{p_1}{p_2}}\|x^{\sigma_2}\partial^2 f \|^{p_1}_{L^{p_2}([r,2r])}  
 \geq  r^{(\sigma_1-\sigma_0)p_1 -(1+p_1-\frac{p_1}{p_2})}\|   x^{\sigma_0}f \|_{L^{p_0}([r,2r])}^{p_1}
\end{equation}
then we subdivide this interval into pieces where these two terms are comparable, and complete the proof of \eqref{interp-1d-sister} within this interval. 

Unfortunately this may not be the case in all dyadic subintervals. To rectify this we introduce a slowly varying frequency envelopes $\left\{c^2_{k} \right\}$ for $\|x^{\sigma_2}\partial^2 f \|_{L^{p_2}}$ and $\left\{c^0_{k} \right\}$ for $\|x^{\sigma_0} f \|_{L^{p_0}}$, so that the following properties hold:
\begin{itemize}
\item Control norm  
\[
\|x^{\sigma_2}\partial^2 f \|_{L^{p_2}([I_k])}\leq c^2_{k} \mbox{ and } \|x^{\sigma_0} f \|_{L^{p_0}([I_k])}\leq c^0_{k}
\]
\item $l^{p_2}$ and $l^{p_0}$ summability 
\[
\sum_{k} (c^2_{k})^{p_2} \approx \|x^{\sigma_2}\partial^2 f \|^{p_2}_{L^{p_2}} \mbox{ and }  \sum_{k} (c^0_{k})^{p_0} \approx \|x^{\sigma_0} f \|^{p_0}_{L^{p_0}}
\]
\item Slowly varying 
\[
\frac{c^0_k}{c^0_j}\lesssim 2^{\delta \vert j-k\vert },\mbox{ and }
\frac{c^2_k}{c^2_j}\lesssim 2^{\delta \vert j-k\vert }
\]
for $\delta$ small and positive.
\end{itemize}
Now, we compare again as in \eqref{comp-k}
\begin{equation}\label{threshold-k}
2^{k \left\{ p_1(\sigma_1-\sigma_2) +1+p_1-\frac{p_1}{p_2}\right\}}(c^2_k)^{p_1}  
 \geq  2^{k \left\{(\sigma_1-\sigma_0)p_1 -(1+p_1-\frac{p_1}{p_2})\right\}}(c_k^0)^{p_1}
\end{equation}
\begin{equation}\label{threshold-k-s}
2^{k \left\{1+ (\sigma_1-\sigma_2) +\frac{1}{p_1}-\frac{1}{p_2}\right\}}c^2_k  
 \geq  2^{k \left\{(\sigma_1-\sigma_0)- 1-(\frac{1}{p_1}-\frac{1}{p_2} )\right\}}c_k^0
\end{equation}
which holds iff 
\[
c_k^2\geq 2^{k\left\{(\sigma_2 -\sigma_0 -2) +\frac{1}{p_2}-\frac{1}{p_0}\right\}}c_k^0.
\]
In the dyadic regions where this holds we finish the proof as discussed above, by subdividing the dyadic intervals and applying Proposition \ref{p:ungar}.  To see where the switch happens we observe that $c^2_k$ is slowly varying whereas the RHS of the inequality above decreases exponentially, as $k$ grows. Then we can find a unique $k_0$ where the two are comparable,
\begin{equation}\label{threshold-k0}
c_{k_0}^2\approx 2^{k_0\left\{(\sigma_2 -\sigma_0 -2) +\frac{1}{p_2}-\frac{1}{p_0}\right\}}c_{k_0}^0.
\end{equation}
Then \eqref{threshold-k-s} holds for $k\geq k_0$, which implies that 
  \begin{equation}
\label{interp-1d-s+}
\| x^{\sigma_1 } \partial f \|^{p_1}_{L^{p_1}(I_k)} \lesssim  (c_k^0 c_k^2)^{\frac{p_1}{2}} .
\end{equation}

It remains to consider the case when $k< k_0$, where we are simply going to obtain a pointwise bound for $\partial f$. Selecting a favorable point $x_0 \in I_{k_0}$, i.e. where
\begin{equation}
\label{calc2}
\partial f (x_0) \lesssim 2^{-k_0}\int_{I_{k_0}}\vert \partial f\vert \, dx \lesssim 2^{-(\frac{1}{p_1}+\sigma_1)k_0} \| x^{\sigma_1}  \partial f \|_{L^{p_1}(I_{k_0})}
\end{equation}
we estimate for $x\in I_{k_1}$ with $k_1<k_0$:
\[
\begin{split}
\vert \partial f (x)\vert \lesssim & \ \vert \partial f(x_0)\vert  +\int_x^{x_0}\vert \partial^2 f \vert \, dx 
\\
\lesssim & \ \vert \partial f(x_0)\vert  +
\sum_{k = k_1}^{k_0}\int_{I_k} \vert \partial^2 f \vert \, dx 
\\
\lesssim & \ \vert \partial f(x_0)\vert  +\sum_{k = k_1}^{k_0}
2^{k\left(\frac{p_2-1}{p_2}-\sigma_2\right)}   \Vert x^{\sigma_2}\partial^2f\Vert_{L^{p_2}(I_k)} 
\\
\lesssim & \ \vert \partial f(x_0)\vert  +\sum_{k = k_1}^{k_0}
2^{k\left(\frac{p_2-1}{p_2}-\sigma_2\right)}   (c_k^2) 
\\
\lesssim & \ \vert \partial f(x_0)\vert  +
 2^{(k_0-k_1) \left(-\frac{p_2 -1}{p_2} +\sigma_2 +2\delta\right)_+} \cdot 2^{k_0\left(\frac{p_2 -1}{p_2} -\sigma_2\right)}   (c_{k}^2) .
 \\
\lesssim & \ \vert \partial f(x_0)\vert  +
 (x_0/x)^{\left(-\frac{p_2 -1}{p_2} +\sigma_2 +2\delta\right)_+} \cdot 2^{k_0\left(\frac{p_2 -1}{p_2} -\sigma_2\right)}   (c_{k}^2) .
 \end{split}
\]
Now we estimate using the bound above
\begin{equation}
    \label{calc}
\| x^{\sigma_1 } \partial f \|^{p_1}_{L^{p_1} ([0, x_0])}= \int_{0}^{x_0} x^{p_1\sigma_1 } (\partial f)^{p_1}\, dx \lesssim x_0^{p_1\sigma_1+1} \vert \partial f(x_0)\vert^{p_1} +2^{k_0p_1\left(\frac{p_2 -1}{p_2} +\sigma_1-\sigma_2\right)}  x_0 (c_{k}^2),
\end{equation}
where the integral converges since as the exponents obey the restriction dictated by the scaling in \eqref{scaling}, and $\delta$ is sufficiently small. 
To finish the proof we observe that by \eqref{threshold-k0} and \eqref{calc2}, the RHS of \eqref{calc} is comparable to the right hand side of \eqref{interp-1d-s+} when $k=k_0$.

This concludes the proof of Lemma\ref{l:interp}.
\end{proof}
The proof of the Proposition \ref{p:interpolation} follows as a straightforward consequence. 
\end{proof}

\begin{proof} [Proof of Proposition \ref{p:interpolation-c}]
This is largely similar to the proof of Proposition~\ref{p:interpolation-g},
so we omit the details and only describe the key differences. The reduction to the case $m=2$ is similar, using also the $m= 2$ case of Proposition~\ref{p:interpolation-g}, at least if we allow $p_2$ to be arbitrary rather than $2$. The one dimensional reduction is also similar. Thus we are left
with having to prove the following analogue of Lemma~\ref{l:interp}

\begin{lemma}
\label{l:interp-bis}
 Let $p_j \in [1, \infty]$, and $\sigma_j \in \R$ with $j=\overline{1,2}$,
 so that
\[
\frac{1}{p_2}=\frac{3}{p_1}, \quad  \mbox{ and } \quad  \sigma_2=3\sigma_1,
\]
and with 
\[
\frac32 - \frac{1}{p_2} > \sigma_2, \qquad \sigma_1 >-\frac{1}{p_1}.
\]
Then the following inequality holds
 \begin{equation}
\label{interp-1d-sister-bis}
\| x^{\sigma_1 } \partial f \|_{L^{p_1}} \lesssim \|   f \|_{\dot C^\frac12}^{\frac{2}{3}} \|x^{\sigma_2}\partial^2 f \|_{L^{p_2}}^{\frac{1}{3}}.
\end{equation}
\end{lemma}

This Lemma is proved using the following analogue of Proposition~\ref{p:ungar}, which is a straightforward exercise.

\begin{proposition} \label{p:ungar-g} 
On an interval $I$, whose length  is denoted by $\lambda$, one has
\[
\Vert  u_x \Vert_{L^{p_1}(I)}^{p_1}\lesssim \lambda^{1+p_1-\frac{p_1}{p_2}}\Vert u_{xx}\Vert ^{p_1}_{L^{p_2}(I)}  +\lambda^{-\frac12(1+p_1-\frac{p_1}{p_2})}\Vert u\Vert ^{p_1}_{\dot{C}^{\frac12}(I)},
\]
where  $p_j \in [1, \infty]$, $j=\overline{0,2}$.
\end{proposition}
\end{proof}

\begin{proof} [Proof of Proposition \ref{p:interpolation-d}]
This is also similar to the proof of Proposition~\ref{p:interpolation-g},
so we omit the details and only describe the key differences. The reduction to the case $m=2$ uses again the $m= 2$ case of Proposition~\ref{p:interpolation-g}, and the one dimensional reduction is also similar. Thus we are left
with having to prove the following analogue of Lemma~\ref{l:interp}:

\begin{lemma}
\label{l:interp-tc}
 Let $p_j \in [1, \infty]$, and $\sigma_j \in \R$ with $j=\overline{1,2}$,
 so that
\[
\frac{1}{p_2} =\frac{2}{p_1}, \quad  \mbox{ and } \quad  \sigma_2 - \frac12 =2\sigma_1,
\]
and with 
\[
2 -  \frac{d}{p_2} > \sigma_2 +\frac12, \qquad \sigma_1 >-\frac{1}{p_1}.
\]
Then the following inequality holds
 \begin{equation}
\label{interp-1d-sister+}
\| x^{\sigma_1 } \partial f \|_{L^{p_1}} \lesssim \|  f \|_{\tC}^{\frac{1}{2}} \|x^{\sigma_2}\partial^2 f \|_{L^{p_2}}^{\frac{1}{2}},
\end{equation}
\end{lemma}
This Lemma is proved in the same fashion as Lemma~\ref{l:interp} using directly Proposition~\ref{p:ungar}
for $f+c$ with well chosen constants $c$.

\end{proof}

\section{The linearized equations}

\label{s:linearize}

This section is devoted to the study of the linearized equations,
which have the form
\begin{equation}\label{lin-original}
\left\{
\begin{aligned}
&\partial_t s + v \cdot \nabla s   + w \cdot \nabla r + \kappa ( s \nabla \cdot v + r \nabla \cdot w) = 0 \\
& \partial_t  w + (v \cdot \nabla) w +  (w \cdot \nabla) v  + \nabla s  = 0.
\end{aligned}
\right.
\end{equation}
Using the material derivative, these equations are written in the form
\begin{equation}\label{lin}
\left\{
\begin{aligned}
&D_t s   + w \cdot \nabla r + \kappa ( s \nabla \cdot  v + r \nabla \cdot w) = 0 \\
& D_t  w  +  (w \cdot \nabla) v  + \nabla s  = 0.
\end{aligned}
\right.
\end{equation}
Here $(s,w)$ are functions defined within the time dependent  gas domain $\Omega$. Notably, no boundary conditions on $(s,w)$ are imposed or required on the free boundary $\Gamma$.

\subsection{Energy estimates and well-posedness} 
We first consider the question of proving well-posedness and energy estimates for the  linearized equations:

\begin{proposition}
Let $(r,v)$ be a solution to the compressible Euler equations \eqref{free-bd-euler-b} in the moving domain $\Omega_t$. Assume that 
both $r$ and $v$ are Lipschitz continuous, and that $r$ vanishes simply on the free
boundary.  Then the linearized equation \eqref{lin} is well-posed in $\H$, and the following energy estimate holds for all solutions $(s,w)$:
\begin{equation}\label{lin-ee}
\left | \frac{d}{dt} \| (s,w)\|_{\H}^2 \right | \lesssim \|\nabla v\|_{L^\infty}
 \| (s,w)\|_{\H}^2
\end{equation}
\end{proposition}
Here we estimate the absolute value of the time derivative of the linearized energy, in order to guarantee both forward and backward energy estimates; these are both needed in order to prove well-posedness.

\begin{proof}
We recall the time dependent weighted $\H$ norm,
\[
\| (s,w)\|_{\H}^2 = \int r^{\frac{1-\kappa}{\kappa}}(\vert s\vert ^2 + \kappa r \vert w\vert ^2)\, dx .
\]
To compute its time derivative, we use the material derivative in a standard fashion. For later reference we state the result in the following Lemma:

\begin{lemma}
\label{p:calculus}
Assume that the time dependent domain $\Omega_t$ flows with Lipschitz velocity 
$v$. Then the  time derivative of the time-dependent volume integral is given by
\begin{equation}
    \frac{d}{dt}\int_{\Omega (t)} f(t, x)\,dx= \int_{\Omega_t} D_tf + f \nabla \cdot v (t)\, dx.
\end{equation}
\end{lemma}

Using the above Lemma, we compute 
\[
\begin{aligned}
\frac{d}{dt} \| (s,w)\|_{\H}^2 = -\kappa \! \int_{\Omega_t}\!\! r^{\frac{1-k}{k}}\nabla v \left( \vert s\vert^2 +2 r \vert w\vert^2 \right) dx - 2 \! \int_{\Omega_t} \!\! r^{\frac{1-\kappa}{\kappa}}
(s  (w \cdot \nabla r + \kappa r \nabla\cdot w)   + \kappa r w  \nabla s)  dx .
\end{aligned}
\]
We observe that the last integral is zero. The computations is straightforward
and follows from integration by parts:
\[
\begin{split}
 \  - 2 \int r^{\frac{1-\kappa}{\kappa}}
(s  w \cdot \nabla r + \kappa r \nabla ( s  w) )\, dx = \ 0,
\end{split}
\]
as the boundary terms vanish on $\Gamma$.

The first integral includes the bounded term $\nabla \cdot v$. It follows right away that the energy norm will indeed control it, and the desired energy estimate  \eqref{lin-ee} follows.

The well-posedness result will follow in a standard fashion from a similar estimate for the adjoint equation,  interpreted as a backward evolution in the dual space $\H^*$. We identify $\H^* = \H$ by Riesz's theorem, 
with respect to the associated inner product in $\H$:
\begin{equation}
\label{inner product}
\langle (s,w), (\tilde{s},\tw)\rangle_{\H}=\int_{\Omega_t} r^{\frac{1-\kappa}{\kappa}} (s\ts +\kappa r w\tw)\, dx,
\end{equation}

Then the adjoint system associated to \eqref{lin-original}, with respect to this duality, is easily computed to be the following:
\begin{equation}
    \label{adjoint-lin-system}
    \left\{
    \begin{aligned}
    &D_t\ts +\kappa r\nabla \tw+\tw\nabla r=0\\
    &D_t\tw-\tw\nabla v+\nabla \ts=0.
    \end{aligned}
    \right.
\end{equation}
Modulo bounded, perturbative terms, this is identical to the direct system 
\eqref{lin}, therefore the backward energy estimate for the adjoint problem
\eqref{adjoint-lin-system} follows directly from \eqref{lin-ee}.
\end{proof}

In particular we note that,  due to translations in time and space symmetries, the linearized estimate applies  to the functions $(s,w) = (\nabla r,\nabla v)$, as well as $(s,w) = (\partial_t r,\partial_t v)$.

\subsection{Second order transition operators}
We remark that discarding the $\nabla v$ terms from the equations we obtain a reduced linearized equation,
\begin{equation}\label{lin-reduced}
\left\{
\begin{aligned}
&D_t s   + w \cdot \nabla r + \kappa  r \nabla \cdot w = 0 \\
& D_t  w   + \nabla s  = 0,
\end{aligned}
\right.
\end{equation}
which is also well-posed in $\H$.  For many purposes it is useful to
also rewrite the linearized equation as a second order evolution.  We
will only seek to capture the leading part, up to terms of order $1$.
Starting from the above reduced linearized equation, 
we compute second order equations where we discard the $\nabla v$
terms arising from commuting $D_t$ and $\nabla$.

Then for $s$ we obtain the reduced second order equation, (which would be exact if $v$ were  constant)
\begin{equation}\label{L1-def}
D_t^2 s  \approx L_1 s, \qquad L_1 s  =   \kappa r \Delta s  + \nabla r \cdot \nabla s , 
\end{equation}
which for $\kappa = 1$ yields 
\[
L_1   =  \nabla r \nabla . \
\]

On the other hand for $w$ we  similarly obtain
\begin{equation}\label{L2-def}
D_t^2 w \approx L_2 w, \qquad L_2 w = \kappa \nabla (r \nabla \cdot w) 
+ \nabla (\nabla r \cdot w )  .
\end{equation}
The operators $L_1$ and $L_2$ will play an important role in the analysis of the 
energy functionals in the next section.  An important observation is that
they are symmetric operators in the  $L^2$ spaces  which occur in our energy functional $E_{lin}$ and in the norm $\H$. For a more in depth discussion we 
separate them:

\begin{lemma}\label{l:L1-symmetric}
Assume that $r$ is Lipschitz continuous in the domain $\Omega$, and nondegenerate on the  boundary $\Gamma$. Then the operator $L_1$, defined as an unbounded 
operator in the Hilbert space $H^{0,\frac{1-\kappa}{\kappa}} = L^2(r^{\frac{1-\kappa}{\kappa}})$, with 
\begin{equation*}
 \mathcal{D}(L_1) :=\left\{ f\in L^2(r^{\frac{1-\kappa}{\kappa}})\, |\, L_1 f \in L^2(r^{\frac{1-\kappa}{\kappa}}) \mbox{ in the distributional sense}  \right\}.
 \end{equation*}
is a nonnegative, self-adjoint operator.
\end{lemma}

The proof is relatively standard and is left for the reader. Later in the paper,
see Lemma~\ref{l:coercive}, we prove that $L_1$ is coercive, and that it satisfies
good elliptic bounds, which in particular will allow us to identify the domain of $L_2+L_3$ as
\[
D(L_1) = H^{2,\frac{1+\kappa}{2\kappa}},
\]
which is the first component of the $\H^2$ space.

Next we turn our attention to the operator $L_2$. This is also a symmetric operator, this time in the space $ L^2(r^{\frac{1}{\kappa}})$, which is the 
second component of $\H$. However, it lacks full coercivity as $L_2 w$ only controls
the divergence of $w$. For this reason, we will match it with a second 
operator which controls the curl of $w$, namely 
\[
L_3 = \kappa r^{-\frac1{\kappa}} \div r^{1+\frac{1}{k}} \curl = \kappa \div r \curl + \nabla r \curl 
\]
so that $L_2L_3 = L_3 L_2 = 0$. Then the operator $L_2+L_3$ has the right properties:

\begin{lemma}\label{l:L2-symmetric}
Assume that $r$ is Lipschitz continuous in $\Omega$, and nondegenerate on the  boundary $\Gamma$. Then the operator $L_2+L_3$, defined as an unbounded 
operator in the Hilbert space $L^2(r^{\frac{1}{\kappa}})$, with 
\begin{equation*}
 \mathcal{D}(L_2+L_3) :=\left\{ f\in L^2(r^{\frac{1}{\kappa}})\, |\, (L_2+L_3) f \in L^2(r^{\frac{1}{\kappa}}) \mbox{ in the distributional sense}  \right\}.
 \end{equation*}
is a nonnegative, self-adjoint operator.
\end{lemma}

Later in the paper, as a consequence of Lemma~\ref{l:coercive}, it follows  that $L_2+L_3$ is coercive, and that it satisfies good elliptic bounds, which in particular will allow us to identify the domain of $L_2+L_3$ as
\[
D(L_2+L_3) = H^{2,\frac{1+3\kappa}{2\kappa}},
\]
which is the second component of the $\H^2$ space.

\begin{remark}
 For the most part, we will think of $L_1$ and $L_2$ in a paradifferential fashion, i.e. with the 
$r$ dependent coefficients localized at a lower frequency than the argument. The exact interpretation of this
will be made clear later on.
\end{remark}

\section{Difference bounds and the uniqueness result}

\label{s:uniqueness}
Our aim here is to prove $L^2$ difference bounds for solutions,
which could heuristically  be seen as integrated\footnote{Along a one parameter $C^1$ family of solutions.}  versions of the estimates for the linearized equation in the previous section. As a corollary, this will yield the uniqueness result in Theorem~\ref{t:unique}.

For this we consider two solutions $(r_1,v_1)$ and $(r_2,v_2)$ for our main system \eqref{free-bd-euler-b}, and seek to compare them.   Inspired by the linearized energy estimate, we seek to produce a  similar weighted $L^2$ bound for the difference 
\[
(s,w) = (r_1-r_2,v_1-v_2).
\]

The first difficulty we encounter is that the two solutions may not have the same domain. The obvious solution is to consider the differences within the common domain,
\[
\Omega = \Omega_1 \cap \Omega_2.
\]
The domain $\Omega$ no longer has a $C^1$ boundary. However, 
if we assume that the two boundaries $\Gamma_1$ and $\Gamma_2$ 
are close in the Lipschitz topology, then $\Omega$ still has a 
Lipschitz boundary $\Gamma$ which is close to $C^1$. To measure 
the difference between the two solutions on the common domain,
we introduce the following distance functional\footnote{We do not prove 
or claim that this defines a metric.}
\begin{equation}\label{Diff-nondeg}
\begin{split}
D_{\H}((r_1,v_1),(r_2,v_2)) =\int_{\Omega_t}  (r_1+r_2)^{\sigma-1}\left( (r_1-r_2)^2 +\kappa (r_1+r_2) (v_1-v_2)^2\right) dx,
\end{split}
\end{equation}
where $\sigma = \frac{1}{\kappa}$ throughout the section.
We remark that the weight $r_1+r_2$ vanishes on $\Gamma$ only at points where
$\Gamma_1$ and $\Gamma_2$ intersect.  Away from such points, both 
$r_1+r_2$ and $r_1-r_2$ are nondegenerate; precisely, we have 
\[
|r_1(x_0)-r_2(x_0)| = r_1(x_0)+r_2(x_0), \qquad x_0 \in \Gamma_t.
\]
Since both $r_1$ and $r_2$ are assumed to be uniformly Lipschitz and nondegenerate, it follows that this relation extends to a neighbourhood of $x_0$,
\[
|r_1(x)-r_2(x)| \approx r_1(x_0)+r_2(x_0), \qquad |x-x_0| \ll r_1(x_0)+r_2(x_0).
\]
Then the first term in $D_{\H}$ yields a nontrivial contribution in this boundary region:

\begin{lemma}
Assume that $r_1$ and $r_2$ are  uniformly Lipschitz and nondegenerate, and close
in the Lipschitz topology. Then we have 
\begin{equation}\label{Diff-bdr}
 \int_{\Gamma_t} |r_1+r_2|^{\sigma+2} d\sigma \lesssim    D_\H((r_1,v_1),r_2,v_2)) .
\end{equation}
\end{lemma}
One can view the integral in \eqref{Diff-bdr} as a measure of the distance between 
the two boundaries, with the same scaling as $D_\H$. 

Now we can state our main estimate for differences of solutions:

\begin{theorem}\label{t:Diff}
Let $(r_1,v_1)$ and $(r_2,v_2)$ be two solutions for the system \eqref{free-bd-euler-b} in $[0,T]$, with regularity $\nabla r_j \in \tC$, $v_j \in C^1$, so that $r_j$ are uniformly nondegenerate near the boundary and close in the Lipschitz topology, $j=1,2$. Then
we have the uniform difference bound
\begin{equation}\label{Diff-est}
\sup_{t \in [0,T]} D_\H((r_1,v_1)(t),(r_2,v_2)(t))    \lesssim D_\H((r_1,v_1)(0),(r_2,v_2)(0)). 
\end{equation}
\end{theorem}
We remark that 
\[
D_\H((r_1,v_1),(r_2,v_2)) = 0 \quad \text{iff} \quad (r_1,v_1)= (r_2,v_2).
\]
Thus, our uniqueness result in Theorem~\ref{t:unique} can be viewed as a consequence of the above theorem.

The remainder  of this section is devoted to the proof of the theorem.

\subsection{A degenerate difference functional}

The distance functional $D_\H$ introduced above is effective in measuring the distance between the two boundaries because it is nondegenerate at the boundary. 
This, however, turns into a disadvantage when we seek to estimate its time derivative. For this reason, in the energy estimates for the difference it is convenient to replace it by a seemingly weaker functional, where the weights
do vanish on the boundary. Our solution is to replace the $r_1+r_2$  weights  in $D_\H$ with symmetric expressions in $r_1$ and $r_2$, which agree to second order  with $r_1+r_2$ where $r_1 = r_2$, and also which vanish on $\Gamma_t = \partial \Omega_t$. 

Precisely, we will consider the modified difference functional
\begin{equation}\label{Diff}
 \tilde D_\H((r_1,v_1),(r_2,v_2)): = \int_{\Omega_t} (r_1+r_2)^{\sigma-1}\left(a(r_1,r_2) (r_1-r_2)^2 + \kappa b(r_1,r_2)(v_1-v_2)^2\right)\, dx, 
\end{equation}
where for now the weights $a$ and $b$  are chosen as follows as functions
of $\mu= r_1+r_2$ and $\nu = r_1 - r_2$:

\begin{enumerate}
\item They are smooth, homogeneous, nonnegative  functions of degree $0$ respectively $1$, even in $\nu$, in the region $\{ 0 \leq |\nu| < \mu\}$.

\item They are connected by the relation $\mu a = 2 b$.

\item They are supported in $\{ |\nu| < \frac12 \mu$, with $a = 1$ in $|\nu| < \frac14 \mu$.
\end{enumerate}

For almost all the analysis these conditions will suffice. Later, almost of the end of the section, we will add one additional condition, see \eqref{choose-a}, and show that 
such a condition can be enforced.

Our objective now is to compare the two difference functionals. Clearly $\tilde D_\H \lesssim D_\H$. The next lemma proves the converse.

\begin{lemma}\label{l:D-equiv}
Assume that $A = A_1+A_2$ is small. Then
\begin{equation}\label{D-equiv}
  D_\H((r_1,v_1),(r_2,v_2)) \approx_A  \tilde D_\H((r_1,v_1),(r_2,v_2)) .
\end{equation}
\end{lemma}
\begin{proof}
We need to prove the "$\lesssim$" inequality. To do that, we observe that by foliating $\Omega (t)$ with lines transversal 
to $\Gamma$, the  the bound \eqref{D-equiv} reduces to the one dimensional case.  Denoting the distance to the boundary by $r$ and the value of $r_1+r_2$ on the boundary by $r_0$,
we have the relations
\[
r_1+r_2 \approx r+r_0, \qquad a \approx \frac{r}{r+r_0}, \qquad b 
\approx r.
\]
Then 
\[
\tilde D_\H \approx \int_0^\infty r (r+r_0)^{\sigma-2}\left( 
(r_1-r_2)^2 +\kappa (r_0+r)(v_1-v_2)^2\right)\, dr. 
\]
On the other hand, 
\[
 D_\H \approx \int_0^\infty  (r+r_0)^{\sigma-1}\left( 
(r_1-r_2)^2 +\kappa (r_0+r)(v_1-v_2)^2\right)  \,dr
+ r_0^{\sigma+2}.
\]
In the region where $r \ll r_0$ we have $|r_1-r_2| \approx r_0$.
Then we can evaluate the first part in the $\tilde{D}_\H$ integral by
\[
\int_0^{cr_0} r (r+r_0)^{\sigma-2}
(r_1-r_2)^2 dr \approx r_0^{\sigma+2},
\]
thereby obtaining the integral in \eqref{Diff-bdr}.
Conversely, we have
\[
\int_0^{cr_0} (r+r_0)^{\sigma-1}
(r_1-r_2)^2 dr \approx r_0^{\sigma+2},
\]
which gives the desired bound for the missing part of the first term of $ D_\H$.

It remains to compare the $v_1-v_2$ terms, where we also need 
to focus on the region $r \approx r_0$. Denote by 
\[
\bar v := \fint_{r \approx r_0} v_1 - v_2 \, dr
\]
for which we can estimate 
\[
|\bar v|^2 \lesssim r_0^{-\sigma-1} \tilde{D}_\H .
\]
Then for smaller $r$ we can use the H\"older $C^\frac12$ norm
to estimate
\[
|v_1-v_2|^2 \lesssim |\bar v|^2 + A r_0 .
\]
Hence
\[
\int_0^{r_0} (r+r_0)^{\sigma}
(v_1-v_2)^2 dr \lesssim r_0^{\sigma+1} ( |\bar v|^2 + A r_0 ) \lesssim \tilde D_\H,
\]
as needed.
\end{proof}

\subsection{The energy estimate} 

The second step in the proof of Theorem~\ref{t:Diff} is to track 
the time evolution of the degenerate energy $\tilde D_\H$:

\begin{proposition}
We have
\begin{equation}\label{Diff-est0}
\frac{d}{dt} \tilde D_\H  ((r_1,v_1),(r_2,v_2)) \lesssim (B_1+B_2) D_\H  ((r_1,v_1),(r_2,v_2)).
\end{equation}
\end{proposition}
In view of Lemma~\ref{l:D-equiv}, the conclusion of the theorem
then follows if we apply Gronwall's inequality.

\begin{proof}
To compute the time derivative of $\tilde D_{\H}(t)$ we use material
derivatives. But we have two of those, $D_t^1$ and $D_t^2$, and it is essential to 
do the computations in a symmetric fashion so we will use the averaged material
derivative 
\[
D_t = \frac12(D_t^1+D_t^2).
\]
Using the equations \eqref{free-bd-euler-b-dt}, we compute difference equations 
\begin{align}
\label{diff-r}
D_t(r_1-r_2) = &   - \frac{\kappa}2 (r_1 -r_2)\nabla (v_1+ v_2)\! -\! \frac{\kappa}2 (r_1 + r_2)\nabla (v_1- v_2)\! -\! \frac12(v_1-v_2) \nabla(r_1+r_2),
\\
D_t(v_1-v_2) = &  - \nabla(r_1-r_2) - \frac12(v_1-v_2) \nabla(v_1+v_2).
\label{diff-v}
\end{align}
We will also need a symmetrized sum equation
\begin{equation}\label{sum-r}
D_t(r_1+r_2) =  - \frac{\kappa}2 (r_1 +r_2)\nabla (v_1+ v_2) - \frac{\kappa}2 (r_1 - r_2)\nabla (v_1- v_2) - \frac12(v_1-v_2) \nabla(r_1-r_2).
\end{equation}
We use these relations to compute the time derivative of the energy, using Lemma~\ref{p:calculus} with $v := \frac12(v_1+v_2)$. We have 
\[
|\nabla v_1|+ |\nabla v_2| \lesssim B := B_1+B_2,
\]
so the contribution of the $\nabla \cdot v$ term is directly estimated 
by $B D_\H(t)$, and so are the contributions of the first term in \eqref{diff-r},
 the first two terms in \eqref{sum-r},  as well as the second term in \eqref{diff-v}.  Hence we obtain
\[
\frac{d}{dt} \tilde D_\H (t) = I_1+I_2 + I_3 + O(B) D_\H (t) ,
\]
where the contributions $I_j$ are as follows:

\medskip

i) $I_1$ represents the contributions of the averaged material derivative 
applied to the first factor $(r_1+r_2)^{\sigma-1}$ via the third term \eqref{sum-r}, namely 
\[
I_1 = - \frac{\sigma-1}{2} \int  (r_1+r_2)^{\sigma-2}\left(a(r_1,r_2) (r_1-r_2)^2 + \kappa b(r_1,r_2)(v_1-v_2)^2\right)   (v_1-v_2) \nabla(r_1-r_2)     \, dx.
\]
We separate the two terms, 
\[
I_1 = J_1^a + O(J_2),
\]
where 
\[
J_1^a = - \frac{\sigma-1}{2} \int  (r_1+r_2)^{\sigma-2} a(r_1,r_2) (r_1-r_2)^2    (v_1-v_2) \nabla(r_1-r_2)\, dx
\]
and 
\[
J_2 =  \int (r_1+r_2)^{\sigma-1}  |v_1-v_2|^3 \, dx.
\]

\medskip
ii) $I_2$ represents the contributions of the averaged material derivative 
applied to the $a$ and $b$ factors via the third\footnote{The contributions of the first and second terms  terms in \eqref{diff-r} and \eqref{sum-r} are directly bounded by $O(B) D_\H (t)$.} terms in \eqref{diff-r} and \eqref{sum-r}, namely
\[
\begin{aligned}
I_2 = & \ - \frac12 \int  (r_1+r_2)^{\sigma-1}\left(a_\mu(r_1,r_2) (r_1-r_2)^2 + \kappa b_\mu(r_1,r_2)(v_1-v_2)^2\right) (v_1-v_2) \nabla(r_1-r_2) \, dx
\\
& \ - \frac12 \int  (r_1+r_2)^{\sigma-1}\left(a_\nu(r_1,r_2) (r_1-r_2)^2 + \kappa b_\nu(r_1,r_2)(v_1-v_2)^2\right) (v_1-v_2) \nabla(r_1+r_2)\, dx.
\end{aligned}
\]
We also split this into
\[
I_2 = J_1^b + J_1^c + O(J_2),
\]
where
\[
J_1^b = - \frac12 \int  (r_1+r_2)^{\sigma-1} a_\mu(r_1,r_2) (r_1-r_2)^2  (v_1-v_2) \nabla(r_1-r_2) \, dx
\]
\[
J_1^c = - \frac12 \int  (r_1+r_2)^{\sigma-1} a_\nu(r_1,r_2) (r_1-r_2)^2 (v_1-v_2) \nabla(r_1+r_2)\, dx.
\]

\medskip

iii) $I_3$ represents the contribution of the averaged material derivative 
applied to the quadratic factors $(r_1-r_2)^2$ and $(v_1-v_2)^2$ via 
the second and third term in \eqref{diff-r} and the first term in \eqref{diff-v}.
\[
\begin{aligned}
I_3 = & - \kappa \!  \int \! (r_1+r_2)^{\sigma-1}\! \left(a(r_1,r_2) (r_1-r_2) (r_1+r_2)
\nabla(v_1-v_2) 
\! + \!  2 b(r_1,r_2)(v_1-v_2)\nabla(r_1-r_2)\right) \! dx
\\
&  - \int (r_1+r_2)^{\sigma-1} a(r_1,r_2) (r_1-r_2) (v_1-v_2)\nabla(r_1-r_2) \, dx .
\end{aligned}
\]
This is the main term, where we expect to see the same cancellation as in the 
case of the linearized equation. At this place we need the matching condition between $a$ and $b$, namely $2b = (r_1+r_2) a$. Substituting this in and integrating by parts, we obtain an almost full cancellation unless the derivative falls on $a$,
namely
\[
I_3 =  \kappa  \int (r_1+r_2)^{\sigma} (r_1-r_2) (v_1-v_2) \nabla a(r_1,r_2) 
 \, dx = J_1^d + J_1^e,
\]
where 
\[
J_1^d = \kappa  \int (r_1+r_2)^{\sigma} a_{\mu}(r_1,r_2) (r_1-r_2) (v_1-v_2)  
\nabla(r_1+r_2)
 \, dx,
\]
\[
J_1^e = \kappa  \int (r_1+r_2)^{\sigma} a_{\nu}(r_1,r_2) (r_1-r_2) (v_1-v_2)  
\nabla(r_1-r_2)
 \, dx.
\]

The above analysis shows that
\[
\frac{d}{dt} \tilde D_\H(t) \leq J_1^a + J_1^b+ J_1^c + J_1^d + J_1^e + O(J_2) + O(B_1+B_2)  D_\H(t).
\]
Hence, in order to prove \eqref{Diff},
it remains to estimate the error terms,
\begin{equation}\label{Diff-est1}
J_1^a + J_1^b + J_1^c + J_1^d + J_2  \lesssim_{A} (B_1+B_2) D_\H(t).
\end{equation}

\bigskip

{\bf A. The bound for $J_2$.}
We begin with the bound for $J_2$, which is simpler and will also be needed 
later.
As in Lemma~\ref{l:D-equiv}, we can reduce the problem to the 
one dimensional case by foliating $\Omega$ with parallel lines
nearly perpendicular to its boundary $\Gamma$.  Denoting  again the distance to the boundary by $r$ and the value of $r_1+r_2$ on the boundary by $r_0$, we have 
\[
D_\H(t) = \int_0^\infty  (r+r_0)^{\sigma-1}\left( 
(r_1-r_2)^2 +(r+r_0)(v_1-v_2)^2\right) \, dr + r_0^{\sigma+2}.
\]
Then in order to estimate $J_2$, it suffices to  prove the $L^3$ bound in the following interpolation lemma 
\begin{lemma}\label{l:delta-v}
Let $\sigma > 0$ and $r_0 > 0$. Then we have the following interpolation bound in $[r_0,\infty)$:
\begin{equation}
\| r^{\frac{\sigma-1}3} w \|^3_{L^3} \lesssim \| r^{\frac{\sigma}2} w \|^2_{L^2} \| w'\|_{L^\infty} .
\end{equation}
\end{lemma}
The Lemma is applied with $w = v_1-v_2$.  Note that by direct 
integration the same bound holds in all dimensions. Thus we obtain
\begin{corollary}
In the context of our problem we have
\begin{equation}
\| r^{\frac{\sigma-1}3} w \|^3_{L^3} \lesssim B D_\H(t). 
\end{equation}
The same bound also holds if all norms are restricted to any horizontal cylinder
(i.e. transversal to $\Gamma$).
\end{corollary}

\begin{proof}
We think of this as some version of a  Hardy type inequality. The proof is based on similar argument as seen before in Section 2. We interpret $r$ as being pointwise equivalent with $x$ and get
\[
\Vert r^{\frac{\sigma -1}{3}}w\Vert_{L^3 (0, \infty)} \sim \Vert x^{\frac{\sigma -1}{3}}w\Vert_{L^3 (0, \infty)} .\]
To get the result we integrate by part and use H\"older's inequality as follows
\[
\int_0^{\infty} x^{\sigma -1}w^3\, dx =  -\frac{3}{\sigma}\int_0^{\infty}x^{\sigma}w^2 w'\, dx. 
\]
Since, we assumed that $w' \in L^{\infty} (0, \infty)$, we indeed get:
\[
\Vert r^{\frac{\sigma -1}{3}}w\Vert_{L^3 (0, \infty)} \leq \frac{3}{\sigma} \Vert w'\Vert_{L^{\infty}} \Vert r^{\frac{\sigma}{2}}w\Vert^2_{L^2}.
\]

\end{proof}

\bigskip{\bf B. The bound for $J_1^a$, $J_1^b$, $J_1^c$, $J_1^d$ and $J_1^e$.}
We group the like terms and set
\[
J_1^a + J_1^b + J_1^e := J_1^-, \qquad J_1^c+J_1^d : = J_1^+
\]
where we can express $J_1^{-}$ and $J_1^+$ in the form 
\[
J_1^{-} =  \int  (r_1+r_2)^{\sigma-2} a^{\pm}(r_1,r_2) (r_1-r_2)^2    (v_1-v_2) \nabla(r_1-r_2)\, dx.
\]
with $a^-$ smooth and $0$-homogeneous,
\[
a^{-}(r_1,r_2) =  - \frac{\sigma-1}{2} a(r_1,r_2) - \frac12 (r_1+r_2)
a_\mu(r_1,r_2) + \kappa (r_1+r_2)^2 (r_1-r_2)^{-1} a_{\nu}(r_1,r_2),
\]
respectively 
\[
J_1^+ =  \int  (r_1+r_2)^{\sigma} a^+(r_1,r_2) (r_1-r_2)    (v_1-v_2) \nabla(r_1+r_2)\, dx.
\]
with $a^+$ smooth and $-1$-homogeneous,
\[
a^{+}(r_1,r_2) = (\kappa-\frac12) a_\mu(r_1,r_2). 
\]
Here we have used the fact that $a$ is $0$-homogeneous, which yields
$\mu a_\mu+\nu a_\nu = 0$. Also we remark that $a_\mu$ vanishes in a conical 
neighbourhood of $\nu = 0$, therefore we can also think of the 
$J_1^+$ integrand as being at least cubic in $r_1-r_2$.

Heuristically, one might think that after another round of integration by parts
one might place the derivative in $J_1^-$ either on $v_1-v_2$, in which case we get good Gronwall terms, or on $r_1+r_2$, where we just discard it and reduce the problem to estimating an integral of the form
\[
J_1 = \int_{\Omega} (r_0+r)^{\sigma-3} |v_1-v_2||r_1-r_2|^3 \, dx,
\]

Unfortunately such a strategy works only if $\kappa \in (0,1]$; for larger $\kappa$ a problem arises, having to do with potentially large contributions 
within a thin boundary layer. 

Instead, to address the full range of $\kappa$,
we will develop the idea of separating a carefully selected  boundary layer,
where we provide a direct argument, whereas outside this boundary layer 
we can use the simpler integration by parts idea above.

To understand our choice of the boundary layer, we consider first the much simpler case when $r_1-r_2 = 0$ and $\nabla(r_1-r_2) = 0$ on the boundary, where $r_0 = 0$ 
and 
\begin{equation}\label{r3-threshold}
|\nabla(r_1-r_2)| \lesssim Br^\frac12, \qquad |r_1 - r_2| \lesssim B r^\frac32. \end{equation}
Then the estimate for $J_1$ above reduces to the one dimensional case, where we can simply argue by Holder's inequality:
\begin{equation}\label{interp-out}
\begin{aligned}
J_1 \lesssim & \ 
\| r^{\frac{\sigma-1}3} (v_1 -v_2)\|_{L^3} \| r^{\frac{2}{9} \sigma - \frac{8}{9}}   (r_1-r_2) \|_{L^\frac{9}2}^3
\\
\lesssim & \ 
\| r^{\frac{\sigma-1}3} (v_1 -v_2)\|_{L^3} 
\| r^{\frac{\sigma -1}2} (r_1-r_2)\|_{L^2}^\frac43 
\|r^{-\frac32} (r_1-r_2) \|_{L^\infty}^\frac23 \| (r_1-r_2)/r\|_{L^\infty} \lesssim BD_{\H}.
\end{aligned}
\end{equation}
Unfortunately, in general the bound \eqref{r3-threshold} will not hold,
and we will separate the region where it holds and the region where 
it does not hold.

Our boundary layer will depend on $B$, and will roughly be defined as the complement of the region where \eqref{r3-threshold} holds, with the additional proviso that it must have thickness at least $r_0$. For a rigorous definition, we start with the function 
$r_3$ defined on the boundary $\Gamma$ of $\Omega$ as follows:
\begin{equation}\label{def-r3}
r_3 = C r_0 + (B^{-1} r_0)^\frac23 + (B^{-1} |\nabla (r_1-r_2)|)^2,    
\end{equation}
where $C$ is a fixed large universal constant.
Then we define our boundary layer as
\begin{equation}
\Omega^{in} = \Omega \cap \bigcup_{x \in \Gamma} B(x, c r_3(x)),    \end{equation}
as well as its enlargement 
\begin{equation}
\tOmega^{in} = \Omega \cap \bigcup_{x \in \Gamma} B(x, 4c r_3(x)).
\end{equation}
Here $c$ is a small universal constant.

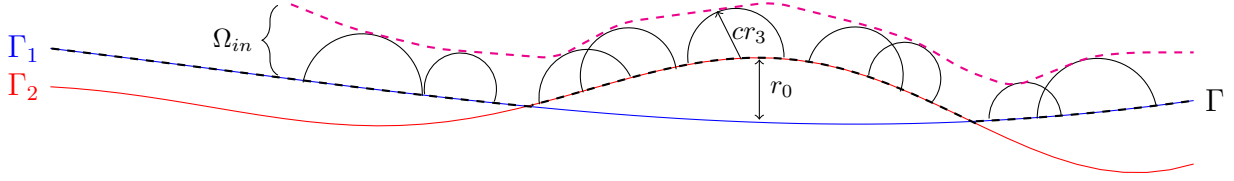
\begin{figure}
\begin{tikzpicture}[x=1.6cm,y=1.6cm]
 \draw[red] [samples=50,domain=2.5:12] plot  (-\x,{cos(\x r)*exp(-\x/6)}) node[left, align=left] (script){$\Gamma_2$};
\draw[blue] [samples=50,domain=2.5:12] plot (-\x,{-sin(\x r/4)*exp(-\x/16)+0.5}) node[left, align=left] (script){$\Gamma_1$};
\draw[] (-9.9,0.15) arc(169:-2:.5);
\draw[] (-8.89,0.053) arc(188:-20:.3);
\draw[] (-7.18,.2) arc(26:188:.4);
\draw[] (-6.8,.26) arc(8:200:.4);
\draw[] (-5.7,.315) arc(165:-20:.4);
\draw[] (-6.7,.31) arc(188:-1:.4);
\draw[] (-5.2,.21) arc(175:-40:.3);
\draw[] (-4.2,-.15) arc(180:5:.3);
\draw[] (-3.8,-.15) arc(180:12:.5);

\draw[rounded corners=2mm, magenta, dashed, thick] (-10, .8) -- (-9.5,.6) -- (-8.8,0.43)  -- (-7.8, .35) -- (-7.25,.65) -- (-6, .82) -- (-5.2,.62)-- (-4.8,.49) -- (-4.4,.25) -- (-4,.11) -- (-3.35, .37) -- (-3,.4) -- (-2.5,.4);

\draw[black, dashed,thick] [samples=50,domain=4.31:8] plot  (-\x,{cos(\x r)*exp(-\x/6)}) node[left, align=left] (script){};

\draw[black,dashed,thick][samples=50,domain=4.31:2.5] plot (-\x,{-sin(\x r/4)*exp(-\x/16)+0.5}) node[right, align=left] (script){$\Gamma$};

\draw[black,dashed,thick][samples=50,domain=8:12] plot (-\x,{-sin(\x r/4)*exp(-\x/16)+0.5}) node[right, align=left] (script){};

\draw [<->={at=(-4,2)}, decorate, decoration={amplitude=6pt,mirror},xshift=-96pt,yshift=-96pt](-4,1.95) -- (-4,2.45) node [black,midway,xshift=.3cm] {\footnotesize $r_0$} ;

\draw [<-={at=(-4,2)}, decorate, decoration={amplitude=6pt,mirror},xshift=-96pt,yshift=-96pt](-4.35,2.85) -- (-4.15,2.45) node [black,midway,xshift=.3cm] {\footnotesize \!\!\! $cr_3$} ;

\draw [decorate,decoration={brace,amplitude=6pt,mirror},xshift=-96pt,
yshift=-96pt]
(-8,2.9) -- (-8,2.32) node [black,midway, xshift=-.8cm] {\footnotesize $\, \quad \Omega_{in} $} ;

\end{tikzpicture}
\caption{ The boundary layer of variable thickness $cr_3$.}
\end{figure}

We want this boundary layer to have a locally uniform geometry. This is insured
by a slowly varying type property of the function $r_3$.
\begin{lemma}\label{l:slow}
We have 
\begin{equation}
|r_3(x) - r_3(y)|   \lesssim r_3^\frac12(x) |x-y|^{\frac12} + |x-y| + r_0^\frac12 r_3^\frac12
\end{equation}
\end{lemma}
\begin{proof}
We consider each of the three components of $r_3$ in \eqref{def-r3}. For the first one we simply use 
the Lipschitz bound for $r_0$. For the second one, we use the $\tC$ bound
on $\nabla r_1$ and $\nabla r_2$ to estimate 
\[
|r_0 (x) - r_0(y)| \lesssim |x-y| |\nabla r_0| + B (|x-y|^\frac32 + r_0(x)^\frac12 |x-y|)
\lesssim B(|x-y| r_3^\frac12 +  |x-y|^\frac32),
\]
which suffices.
Finally for the last term we have
\[
|\nabla (r_1-r_2)(x) - \nabla (r_1-r_2)(y)| \lesssim B (r_0^\frac12(x) + |x-y|^\frac12),
\]
which is again enough.
\end{proof}

This property insures that $\Omega^{in}$ and $\tOmega^{in}$ are separate:

\begin{lemma}\label{l:nice-chi}
There exists a smooth cutoff function $0 \leq \chi \leq 1$ in $\Omega$ with the following properties:

a) Support: $\chi = 1$ in $\Omega^{in}$ and $\chi = 0$ in $\Omega \setminus \tOmega^{in}$,

b) Regularity: $|\partial^{\alpha} \chi(x)| \lesssim (r_1+r_2)^{-|\alpha|}$.

\end{lemma}

\begin{proof}
For $y \in \Omega$ we define the function
\[
G(y) = \min_{x \in \Gamma} |x-y| r_3(x)^{-1}
\]
so that $\Omega^{in}$, $\tOmega^{in}$ are described by 
\[
\Omega^{in} = \{ G(y) \leq c \}, \qquad \tOmega^{in} = \{ G(y) \leq 4c \}.
\]
Then we can use the function $G$ to describe the separation between 
$\Omega^{in}$ and $\Omega\setminus \tOmega^{in}$. Precisely, it suffices to show 
that we can control the Lipschitz constant for $G$ in the transition region,
\[
c \leq G(y) \leq 4c \quad \Rightarrow \quad |\nabla G(y)| \lesssim r^{-1}.
\]
Since $G$ is an infimum, it suffices to show the same for each of its defining functions.
Equivalently, it suffices to show that if $y$ is in the transition region then
\[
c \leq |x-y| r_3(x)^{-1} \leq 4c \quad \Rightarrow \quad r(y) \lesssim r_3(x).
\]
Let $z$ be the closest point to $y$ on the boundary, so that $r(y) \approx |y-z|$.
Then the first relation implies that 
\[
|x-z| \leq 8cr_3(x).
\]
Since $c$ is small, Lemma~\ref{l:slow} shows that $r_3(z) \approx r_3(x)$.
Since we are in the transition region, we must also have
\[
|x-z| \geq c r_3(z),
\]
as needed.
\end{proof}

Finally we verify that we have good control over $r_1-r_2$ on the outer region:
\begin{lemma} \label{l:good-diff-out}
The good bound \eqref{r3-threshold}
holds outside $\Omega^{in}$.
\end{lemma}

\begin{proof}
Let $y \not \in \Omega^{in}$, and $x$ the closest point to $y$ on the boundary.
Then 
\[
r(y) \approx |x-y| \geq c r_3(x).
\]
Using the $\tC$ bound for $\nabla (r_1-r_2)$ along the $[x,y]$ line, we have
\[
|\nabla (r_1-r_2)(z) - \nabla (r_1-r_2)(x)| \lesssim B(r_0(x)^\frac12 + |z-x|^\frac12).
\]
If we use this directly we obtain
\[
|\nabla (r_1-r_2)(y)| \lesssim |\nabla (r_1-r_2)(x)| + B(r_0(x)^\frac12 + |z-x|^\frac12)
\lesssim B(r_3(x)^\frac12 + |y-x|) \lesssim B r(y)^\frac12.
\]
If instead we integrate it between $x$ and $y$ then we obtain
\[
\begin{aligned}
|(r_1-r_2)(y)| \lesssim & \ r_0(x) + |x-y| |\nabla (r_1-r_2)(x)| +
B(r_0(x)^\frac12 |x-y| + |x-y|^\frac32)
\\
\lesssim & \ B r_3(x)^{\frac32}  + B r_3(x)^{\frac12}|x-y| +
B(r_0(x)^\frac12 |x-y| + |x-y|^\frac32) \lesssim B r_3^\frac32.
\end{aligned}
\]

\end{proof}

Now we use the cutoff $\chi$ to split each of the above integrals
in two, and estimate each of them in turn. 

\bigskip

{\bf B.1. The estimate in the outer region.}
Here we insert the cutoff $(1-\chi)$ in each of the two integrals $J_1^{\pm}$, and integrate by parts in $J_1^-$. Precisely, the outer part of $J_1^-$ is
\[
J_1^{-,out} =  \int (1-\chi) (r_1+r_2)^{\sigma-2} a^-(r_1,r_2) (r_1-r_2)^2  (v_1-v_2) \nabla (r_2-r_1) \, dx.
\]
The $\nu$ dependent part of the integrand is 
\[
\nu^2 a^-(\mu,\nu) \nabla \nu .
\]
In order to be able to integrate by parts, we define a function $c(\mu,\nu)$
in the region of interest $|\nu| < \mu$ by 
\[
\partial_\nu c(\mu,\nu) = \nu^2 a^-(\mu,\nu), \qquad c(\mu,0) = 0 .
\]
By definition, $c$  is smooth, homogeneous of order three, and satisfies
\[
|c(\mu,\nu)| \lesssim \nu^3, \qquad |\partial_\mu c (\mu,\nu)| \lesssim \mu^{-1} \nu^3.
\]
Then we can write
\[
\nu^2 a(\mu,\nu) \nabla \nu = \nabla c(\mu,\nu) - \partial_\mu c (\mu,\nu) \nabla \mu .
\]
We substitute this in $J_1^{-,out}$ to obtain
\[
J_1^{-,out} = \frac{\sigma -1}{2} \int (1-\chi) (r_1+r_2)^{\sigma-2} (v_1-v_2) \nabla c  \,dx
+ \int (1-\chi) (r_1+r_2)^{\sigma-2} c_{\mu} (v_1-v_2) \nabla \mu  
\, dx.
\]
In the first integral we integrate by parts. If the derivative falls on $v_1-v_2$
we get a Gronwall term. Else, it falls on $\mu$, which we discard, or on $\chi$, where we use Lemma~\ref{l:nice-chi}. Hence we obtain
\[
J_1^{-,out} \lesssim \int_{\Omega\setminus\Omega^{in}} (r_0+r)^{\sigma-3} |v_1-v_2||r_1-r_2|^3 \, dx
+ O(B_1+B_2)D_\H(t).
\]
In view of Lemma~\ref{l:good-diff-out}, we can estimate the integral as in \eqref{interp-out}
and conclude.

The argument for $J_1^{b,out}$ is similar but simpler, as no integration by parts is needed.

\bigskip

{\bf B.2. The estimate in the boundary layer  region.}
To fix scales, we use the slowly varying property of $r_3$ in Lemma~\ref{l:slow} to 
partition $\tOmega^{in}$ into cylinders $C_{x_0}$ centered at some point 
$x_0 \in \Gamma$, with radius $4c r_3(x_0)$ and similar height, and correspondingly, 
we partition our integrals using an appropriate locally finite partition of unity,
\[
\chi = \sum \chi_{x_0},
\]
where each $\chi_{x_0}$ is smooth on the $r_3(x_0)$ scale.
Within this cylinder we will think of $r_3$ as a constant, $r_3=r_3(x_0)$.

Denoting
\[
J_1^{-,x_0} =  \int_{C_{x_0}} \chi_{x_0} (r_1+r_2)^{\sigma-2} a^-(r_1,r_2) (r_1-r_2)^2  (v_1-v_2) \nabla (r_1-r_2) \, dx,
\]
and similarly for $J_1^{+}$, 
our objective will be to show that in each such component we have 
\begin{equation}\label{J1a-in}
J_1^{\pm,x_0} \lesssim B D_{\H}^{x_0} ,   
\end{equation}
where $D_{\H}^{x_0}$ denotes the integral in $D_\H$ but with the added cutoff $\chi_{x_0}$.
After summation over $x_0$ this will give the desired estimate. We will 
consider separately the cases when $B$ is small or large.

As a prerequisite to the proof of \eqref{J1a-in}, we consider 
pointwise difference bounds within $C_{x_0}$.
We begin with $r_1-r_2$. By construction, within $C_{x_0}$ we have 
\begin{equation}
|\nabla(r_1-r_2)| \lesssim Br_3^\frac12, \qquad     |r_1-r_2| \lesssim Br_3^\frac32.
\end{equation}
In particular this yields
\begin{equation}
r_0 \lesssim    Br_3^\frac32, 
\end{equation}
and the improved pointwise bound
\begin{equation}\label{dr-layer}
|r_1-r_2| \lesssim r_0 + B\, r\, r_3^\frac12,    \end{equation}
where we observe that $r_0$ needs not be constant on the boundary within $C_{x_0}$.
\medskip

Depending on the relative size of $B$ and $r_3$ we will distinguish two scenarios:

\begin{lemma}
One of the following two scenarios applies in $C_{x_0}$:

a) Either $r_0(x_0) \ll r_3$, in which case we must have $B \sqrt{r_3} \lesssim  1$.

b) Or $r_0 \approx r_3$, in which case we must have $B \sqrt{r_0} \gtrsim 1$.
\end{lemma}
We will refer to the first case as the \emph{ small $B$ case } and the second 
as the \emph{large $B$ case}.
\begin{proof}
We start by comparing $r_0(x_0)$ with $r_3$. If $r_3 \approx r_0$, 
then we must have 
\[
r_0 \gtrsim (B^{-1} r_0)^\frac23, 
\]
and further $B \gtrsim (r_0)^{-\frac12}$,
which places us in case (b).

If $r_3 \gg r_0(x_0)$, then we have two nonexclusive possibilities.
Either we have 
\[
r_3 \approx (B^{-1} r_0)^\frac23 \gg r_0,
\]
which yields $B^2 \approx r_0^2 r_3^{-3} \ll r_3^{-1}$, placing us in case 
(a). Or, we have 
\[
r_3 \approx B^{-2} | \nabla (r_1-r_2) |^2 \lesssim B^{-2} ,
\]
which places us again in case (b).
\end{proof}

In addition to bounds for $r_1-r_2$, we also need bounds for $v_1-v_2$.
We will show that within the same cylinder we have a good uniform 
bound for $v_1-v_2$:

\begin{lemma}\label{l:dv-layer}
Within $C_{x_0}$ we have 
\begin{equation}
|v_1-v_2| \lesssim B r_3 + (D_{\H}^{x_0})^\frac12  r_3^{-\frac{\sigma +1}2} r_3^{-\frac{d-1}2}.      
\end{equation}
\end{lemma}
\begin{proof}
Denote by $(v_1-v_2)_{avg}$ the average of $v_1-v_2$ in the region
\[
\tilde C_{x_0} = C_{x_0} \cap \{r \gtrsim \frac12 r_3(x)\},
\]
which represents an interior portion of $C_{x_0}$ away from the boundary.
We estimate this using the distance $D_\H^{x_0}$, where we observe that within 
$\tilde C_{x_0}$ we have $b \approx r_3$. Then we obtain 
\[
r_3^{d} r_3^{\sigma}   (v_1-v_2)_{avg} \lesssim  D_\H^{x_0}.
\]
To obtain the full bound for $v_1-v_2$ we combine this with the $B$ Lipschitz bound,
which yields 
\[
|v_1-v_2| \lesssim B r_3 + |(v_1-v_2)_{avg}| 
\]
within the full cylinder $C_{x_0}$.

\end{proof}

\bigskip

{\bf B.2.a. The case of large $B$.} We recall that in this case we have  $r_3 = r_0$ 
and $B \sqrt{r_0} \gtrsim 1$. Consider $J_1^{-,x_0}$ first.
We discard the gradient terms, bound $r_1-r_2$ by $r_0$ and 
use  Lemma~\ref{l:dv-layer} for $v_1-v_2$. This yields
\[
J_1^{-,x_0} \lesssim r_0^d  r_0^{\sigma} (B r_0 + (D_\H^{x_0})^\frac12  r_0^{-\frac{\sigma +1}2} r_0^{-\frac{d-1}2} ).
\]
On the other hand, a localized version of \eqref{Diff-bdr} yields
\[
r_0^{\sigma+2} \lesssim r_0^{-(d-1)} D_{\H}^{x_0}.
\]
 Combining the last two bounds gives
\[
J_1^{-,x_0} \lesssim D_{\H}^{x_0} (B + r_0^{-\frac12}) \lesssim B  D_{\H}^{x_0},
\]
as needed. The arguments for $J_1^{+,x_0}$ is identical.

\bigskip

{\bf B.2.b. The case of  small $B$.} We recall that this corresponds to $r_0 \ll r_3$
and $B \sqrt{r_3} \lesssim 1$. This is the more difficult case.

The first observation concerning the cylinder $C_{x_0}$ is that $r_1-r_2$ is large there on average, of size $Br_3^\frac32$. This is reflected in a bound from below for $D_\H^{x_0}$:

\begin{lemma}\label{l:B-by-D}
Assume we are in the small $B$ case. Then we have 
\begin{equation}
B^2 r_3^{\sigma+3} r_3^{d-1} \lesssim D_\H^{x_0}.     
\end{equation}
\end{lemma}

\begin{proof}
We approximate $r_1-r_2$ near $x_0$ with its linear expansion, 
\[
(r_1-r_2)(y) = r_0 + \nabla(r_1-r_2)(x_0) (y-x_0) + O(B (r_0^\frac12 + |x_0-y|^\frac12)|x_0-y|))  . 
\]
Within $C_{x_0}$ this can be simplified to 
\[
(r_1-r_2)(y) = r_0 + \nabla(r_1-r_2)(x_0) (y-x_0) + O(B r_3^\frac12 |x_0-y|)).
 \]
Now we consider a small interior  ball
\[
B = B(x_0 + 2 r N,  r), \qquad r_0 < r < cr_3,
\]
where we have $a \approx 1$ and $r_1+r_2 \approx r$,
and use $D_\H^{x_0}$ to estimate
\[
r^{\sigma-1} \int_{B} |r_0 + \nabla(r_1-r_2)(x_0) (y-x_0)|^2 dy \lesssim 
r^{\sigma-1} r^{d} (B r^\frac32)^2 + D_\H^{x_0} .
\]
The integral on the left is easily evaluated, to get
\[
r^{\sigma-1} r^{d} (r_0^2  + r^2 |\nabla(r_1-r_2)(x_0)|^2) \lesssim 
r^{\sigma-1} r^{d} (B r^\frac32)^2 + D_\H^{x_0}.
\]
We can compare the constants on the left and the first term on the right.
We know that 
\[
r_3 \approx \max \{ (B^{-1} r_0)^{\frac23}, (B^{-1} |\nabla(r_1-r_2)(x_0)|)^2 \}.
\]

If the first quantity on the right is larger, then 
\[
r_0 = B r_3^\frac32
\]
and we obtain 
\[
r^{\sigma-1} r^{d}  (B r_3^\frac32)^2 \lesssim 
r^{\sigma-1} r^{d} (B r^\frac32)^2 + D_\H^{x_0}.
\]
Choosing $r = cr_3$ with a small constant $c$, the first term on the right is absorbed on the left and we arrive at the desired conclusion. 

If the second quantity on the right is larger, then \[
|\nabla(r_1-r_2)(x_0)| = B r_3^\frac12,
\]
and we obtain 
\[
r^{\sigma-1} r^{d}  r^2 (B r_3^\frac12)^2 \lesssim 
r^{\sigma-1} r^{d} (B r^\frac32)^2 + D_\H^{x_0}.
\]
Hence we can conclude exactly as before.
\end{proof}

The above Lemma allows us to slightly improve Lemma~\ref{l:dv-layer} to
\begin{lemma}\label{l:dv-layer-smallB}
Assume that $B$ is small. Then within $C_{x_0}$ we have 
\begin{equation}
|v_1-v_2| \lesssim  (D_\H^{x_0})^\frac12  r_3^{-\frac{\sigma +1}2} r_3^{-\frac{d-1}2} .     
\end{equation}
\end{lemma}

We are now ready to estimate the first integral,
\[
\begin{aligned} 
J_1^{-,x_0} \lesssim & \ \int_{C_{x_0}}  (r_0+r)^{\sigma-2}  (r_1-r_2)^2  |v_1-v_2| |\nabla (r_2-r_1)| \, dx
\\
\lesssim & \ Br_3^\frac12  r_3^{-\frac{\sigma +1}{2}} r_3^{-\frac{d-1}{2}} ( D_{\H}^{x_0})^\frac12 
\int_{C_{x_0}} (r_0+r)^{\sigma-2}  (r_0 + B \, r\, r_3^\frac12)^2\, dr \, dx_0
\\
\lesssim & \ Br_3^\frac12  r_3^{-\frac{\sigma +1}{2}} r_3^{-\frac{d-1}{2}} ( D_{\H}^{x_0})^\frac12 
\left( \int r_0^{\sigma+1} dx_0 + 
r_3^{d-1} B^2 r_3^{\sigma +2} \right)
\\
\lesssim & \ Br_3^\frac12  r_3^{-\frac{\sigma +1}{2}} r_3^{-\frac{d-1}{2}} ( D_{\H}^{x_0})^\frac12 
r_3^{d-1} \left( (Br_3^\frac32)^{\sigma+1}+
 B^2 r_3^{\sigma +2} \right)
 \\
 \lesssim & \ ( D_{\H}^{x_0})^\frac12 B^2 r_3^{\frac{d-1}2} r_3^{\frac{\sigma+3}2}
((B \sqrt{r_3})^\sigma + B \sqrt{r_3})
\\
\lesssim & \  BD_{\H}^{x_0}.
\end{aligned}
\]

It remains to estimate $J_1^{+,x_0}$, which we recall here:
\[
J_1^{+,x_0} = C \int  \chi_{x_0} \nu  \mu^{\sigma} a_{\mu} (v_1-v_2) \nabla \mu \, dx, \qquad C = \kappa-\frac12.
\]
Aside from the obvious cancellation when $\kappa = \frac12$,  we would like to integrate 
by parts in order to move the derivative away from $\mu$.  To implement this integration by parts, we need an auxiliary function $c(\mu,\nu)$ so that 
\[
\partial_\mu c(\mu,\nu) = \mu^\sigma a_\mu.
\]
Suppose we have such a function $c$ which is smooth, homogeneous of order $\sigma$
and supported in $ |\mu| \lesssim |\nu| < \mu$. Then integration by parts yields 
\[
\begin{aligned}
J_1^{+,x_0} = & \  C \int  \chi_{x_0} \nu   c_\mu(\mu,\nu) (v_1-v_2) \nabla \mu  \,dx
\\
= & \  - C \int  \chi_{x_0} \nu  c(\mu,\nu)   \nabla \cdot (v_1-v_2) \, dx
\\
& \ - C \int  \chi_{x_0}    (c(\mu,\nu) + \nu c_{\nu}(\mu,\nu))  (v_1-v_2) \nabla \nu \,dx
\\
 & \  - C \int  \nu  c(\mu,\nu)  (v_1-v_2)  \nabla \chi_{x_0} \,dx.
\end{aligned}
\]
In the first integral we bound $\nabla \cdot(v_1-v_2)$ by $B$, and then bound the rest by $D_\H^{x_0}$ since $\mu \approx \nu$ in the support of the integrand. The second 
integral is similar to $J_1^{a,x_0}$. Finally in the third integral the gradient of $\chi_{x_0}$ yields an $r_3^{-1}$ factor, and we can estimate it using Lemma~\ref{l:dv-layer-smallB}
and the bound \eqref{dr-layer} for $r_1-r_2$ by
\[
\begin{aligned}
\lesssim & \ r_3^{-1} \int_{C_{x_0}}  (r_1-r_2)^{\sigma+2} |v_1-v_2|  \, dx 
\\
\lesssim & \ r_3^{d-1} (r_3 r_0^{\sigma+2} + (B \sqrt{r_3})^{\sigma + 2} r_3^{\sigma+3} )
 (D_{\H}^{x_0})^\frac12 r_3^{-\frac{\sigma+1}{2}} r_3^{-\frac{d-1}2}
\\
\lesssim & \, B \, D_{H}^{x_0},
\end{aligned}
\]
where at the last step we bound $r_0 \lesssim Br_3^\frac32$ twice, $r_0 \leq r_3$ for the rest of $r_0,$ and use Lemma~\ref{l:B-by-D}; the powers of $r_3$ will all cancel, as predicted by scaling considerations.

It remains to show that we can find such a function $c$.
This is where a convenient choice of $a$ helps. Precisely, we want $a$ to be nonnegative,
even in $\nu$, supported in $|\nu| < \mu$ and equal to $1$ when $|\nu| \ll \mu$.
In order to avoid boundary terms in the integration by parts, we will choose $c$ with similar 
support. But we also want $c$ to be smooth and homogeneous, and then we will have an issue at $\mu = 0$, unless we can arrange for $c$ to also be supported away from $\mu = 0$.
But this will happen only if 
\begin{equation}\label{choose-a}
\int \mu^\sigma a_\mu \,d\mu = 0.
\end{equation}

\begin{lemma}\label{l:choose-a}
There exists a good choice for  $a$ which satisfies \eqref{choose-a}.
\end{lemma}
\begin{proof}
We will take advantage of the fact that the function $\mu^\sigma$ is increasing, as follows.
We start with a choice $a_0$ for $a$ which is nonincreasing. That  would make the integral in \eqref{choose-a} positive. To correct this we use a nonnegative, compactly supported bump function $a_1$. Its contribution will be negative, as it can be seen integrating by parts:
\[
\int \mu^\sigma a_{1,\mu}\, d\mu  = -\frac{1}{\sigma+1} \int \mu^{\sigma+1} a_{1}\, d\mu.
\]
Then we choose $a = a_0 + C a_1$, with $C > 0$ chosen so that the two contributions 
to the integral in \eqref{choose-a} cancel.
\end{proof}
\bigskip 
\end{proof}

\section{Energy estimates for solutions}
\label{s:ee}

Our objective here is to prove Theorem~\ref{t:energy}. More precisely, we aim to establish uniform control over the $\bfH^{2k}$ norm of the solutions $(v,r)$ in terms 
of the similar norm of the initial data, with growth estimated in terms of the control parameters $A,B$. The key to this is to characterize these norms using energy functionals 
constructed with suitable vector fields naturally associated to the evolution.

\subsection{ The div-curl decomposition}

A first step in our analysis is to understand the structure of our system of equations. In the nondegenerate case, it is known 
that at leading order the compressible Euler  equations decouple into a wave equation for $(r, \nabla\! \cdot\! v)$ and a transport equation for $\omega = \curl v$. We will show that the same happens here. Of course, algebraically the computations are identical. However, interpreting the coupling terms as perturbative is far more delicate in the present context.

We begin with a direct computation, which 
yields the following second order wave equation for $r$, 
\begin{equation}\label{wave-r}
D_t^2 r = \kappa r \Delta r +   \kappa^2 r |\nabla \cdot v|^2 +  \kappa \nabla v (\nabla v)^T
\end{equation}
with speed of propagation (sound speed)
\[
c_s = \kappa r,
\]
where $\nabla\! \cdot\! v$ corresponds to the (material) velocity 
\[
-\kappa \nabla\! \cdot\! v = r^{-1} D_t r.
\]

On the other hand for the vorticity we obtain the transport equation 
\begin{equation}\label{vorticity-transport}
D_t\omega = - \omega  \nabla v - (\nabla v)^T \omega.
\end{equation}

These two equations are coupled, so it is natural to consider them at matched regularity levels, but we will use different energy functionals to capture their contributions to the energy.

\subsection{Vector fields}
Our energy estimates will be obtained by applying a number of well chosen vector fields to the equation in a suitable fashion, so that the differentiated fields obtained as the outcome solve the linearized equation with perturbative source terms.
We do this  separately for the wave component and for the transport part.

\medskip

\emph{a) Vector fields for the wave equation.} Here we use all the vector fields which commute with the wave equation at the leading order. There are two such vector fields, which generate an associated algebra:

\begin{enumerate}[label=a\arabic*)]
\item The material derivative $D_t$; this has order $\frac12$.

\item  The tangential derivatives, $\Omega_{ij} = r_i \partial_j - r_j \partial_i$; these have order $1$.
\end{enumerate}

\noindent 
We will only use $D_t$ in this article, but note that a similar analysis 
works for the tangential derivatives.

\medskip

\emph{b) Vector fields for the transport equation. }
Here we have more flexibility in our choices, again 
generating an algebra. 

\begin{enumerate}[label=b\arabic*)]

\item  The material derivative $D_t$; this has order $\frac12$.

\item  All regular derivatives $\partial$, of order $1$. 

\item  The multiplication by $r$, which has order $-1$.
\end{enumerate}
In order to avoid negative orders here, one may replace $r$
by $r \partial^2$, which has has order $1$.

\subsection{The energy functional}

Here we define energy functionals $E^{2k}(r,v)$ of order $k$, i.e. which involve combinations of vector fields of  orders up to $k$. We will set this up as the sum of a wave and a transport component,
\begin{equation}
E^{2k}(r,v) = E^{2k}_w(r,v) + E^{2k}_{t}(r,v).
\end{equation}

\bigskip

a) The wave energy. Here we want to use operators of the form 
\[
D_t^j , \qquad j \leq 2k
\]
applied to the solution $(r,v)$. However, we would like to have these defined in terms of the data at each fixed time, rather than dynamically. Algebraically this is easily achieved by reiterating the equation. We define
\[
(r_{j},v_{j}) = (D_t^j r,D_t^j v),
\]
which should be viewed as discussed above, as nonlinear\footnote{Strictly speaking, at leading order these are linear expressions, so the better terminology would be quasilinear.} functions of $(r,v)$ at fixed time. 

One might hope that these functions should be good approximate solutions for the linearized equations. Unfortunately, this is not exactly the case even for $(r_{1},v_{1})$. This is because, unlike $\partial$, $D_t$ does not exactly generate an exact symmetry of the equation.  The solution to this difficulty is to work with
associated \emph{good variables}, obtained by adding suitable  corrections to them. We denote these good variables by
$(s_{j}, w_{j})$, and define them as follows:

\begin{enumerate}[label=\roman*)]
\item  $j = 0$. 
\[
(s_{0}, w_{0}) = (r,v).
\]

\item $j = 1$. 
\[
(s_{1}, w_{1}) = \partial_t (r,v).
\]

\item  $j = 2$.
\[
(s_{2}, w_{2}) =  (r_{2} + \frac12 |\nabla r|^2 ,v_{2}).
\]

\item $j \geq 3$:
\[
(s_{j}, w_{j}) = (r_{j} - \nabla r \cdot w_{j-1},v_{j}).
\]
\end{enumerate}

We now define the wave component of the energy as
\begin{equation}
E^{2k}_w(r,v) = \sum_{j \leq k} \| (s_{2j}, w_{2j}) \|_{\H}^2,
\end{equation}
where we recall that $\H$ defined in \eqref{def-H} represents the natural energy functional for the linearized equation. In the sequel we will use these good variables only for 
even $j$, but for the sake of completeness  we have listed them for all $j$.

\bigskip

b) The transport equation. Here we consider a simpler energy, namely
\begin{equation}
E^{2k}_t(r,v) =  \| \omega  \|^2_{H^{2k-1, k+\frac1\kappa}}
\end{equation}
which at leading order scales in the same way as the wave energy above. One can think of this energy as the outcome of applying 
vector fields up to and including order $k$ to the vorticity $\omega$.

\subsection{Energy coercivity}

Our goal here is to prove the equivalence of the energy $E^{2k}$ with the $\bfH^{2k}$ size
of $(r,v)$.

\begin{theorem}\label{t:coercive}
Let $(r,v)$ be smooth functions in $\overline\Omega$ so that $r$ is positive in $\Omega$ and uniformly nondegenerate
on $\Gamma = \partial \Omega$. Then we have
\begin{equation}
E^{2k}(r,v) \approx_A \| (r,v)\|_{\H^{2k}}^2.
\end{equation}
\end{theorem}

\begin{proof}
a)  We begin with the easier part ``$\lesssim$''.  This is obvious for the vorticity component so it remains to discuss the wave component. 

We consider the 
expressions for $(s_{2k}, w_{2k})$. These are both linear combinations of 
multilinear expressions in $r$ and $\nabla v$ with the following properties:

\begin{itemize}
\item They have order $k-1$, respectively $k-\frac12$.

\item They have exactly $2k$ derivatives.

\item 
They contain at most $k+1$, respectively $k$ factors of $r$
or its derivatives.
\end{itemize}

These properties suffice in order to be able to distribute the powers of $r$ and use the interpolation inequalities in Proposition~\ref{p:interpolation}.
We will demonstrate this in the case of $s_{2k}$; the case 
of $w_{2k}$ is similar. A multilinear expression in $s_{2k}$
has the form
\[
M = r^{a} \prod_{j=1}^J \partial^{n_j} r \prod_{l=1}^L \partial^{m_l} v,
\]
where $n_j \geq 1$, $ m_l \geq 1 $,
\[
\sum n_j + \sum m_l = 2k,
\]
and\footnote{Here we allow for $J=0$ or $K=0$, in which case the corresponding products are omitted.}
\[
a + J + L/2 = k+1 .
\]
We seek to split
\[
a = \sum b_j + \sum c_l,
\]
and correspondingly 
\[
M = \prod_{j=1}^J r^{b_j} \partial^{n_j} r \prod_{l=1}^L r^{c_l} \partial^{m_l} v,
\]
so that we can apply our interpolation inequalities from Proposition ~\ref{p:interpolation}, Proposition~\ref{p:interpolation-c}.
These will give  bounds of the form
\[
\| r^{b_j} \partial^{n_j} r \|_{L^{p_j}(r^\frac{1-\kappa}{\kappa})} 
\lesssim A^{1-\frac{2}{p_j}} \|(r,v)\|_{\H^{2k}}^{\frac{2}{p_j}},
\qquad \frac{1}{p_j} = \frac{n_j - 1 - b_j}{2(k-1)},
\]
respectively 
\[
\| r^{c_l} \partial^{m_l} r \|_{L^{q_l}(r^\frac{1-\kappa}{\kappa})} 
\lesssim A^{1-\frac{2}{q_l}} \|(r,v)\|_{\H^{2k}}^{\frac{2}{q_l}},
\qquad \frac{1}{q_l} = \frac{m_l - 1/2 - c_l}{2(k-1)},
\]
where the denominators represent the orders of the expressions
being measured, so they add up to $k-1$ as needed.

It only remains to verify that the $b_j$'s and the $c_l$'s
can be chosen in the range where our interpolation estimates 
apply, which is 
\[
0 \leq b_j \leq  (n_j-1) \frac{k}{2k-1}, 
\]
respectively 
\[
0   \leq c_l \leq (m_l - 1/2) \frac{k+1/2}{2k-\frac12}.
\]
To verify that we can satisfy these conditions we need
\[
\sum (n_j-1) \frac{k}{2k-1} + \sum (m_l - 1/2) \frac{k+1/2}{2k-\frac12}
\leq a.
\]
But the sum on the left is evaluated by
\[
\leq (\sum n_j + \sum m_l - J -L) \frac{k}{2k-1}  
= (2k -J - L)\frac{k}{2k-1} \leq (a+ k-1)\frac{k}{2k-1}
\leq a 
\]
using $a \leq k$. Here equality holds only if $a = k$, $J = 1$ and $L=0$ i.e. for the leading linear case.

\bigskip

b) We continue with the ``$\gtrsim$'' part. To do this we will argue inductively, relating $(s_{2j}, w_{2j})$ with $(s_{2j-2}, w_{2j-2})$. This is done using the transition operators $L_1$
and $L_2$ introduced earlier.

\begin{lemma} \label{l:recurrence}
For $j\geq 2$ we  have  a pair of homogeneous recurrence type relations 
\begin{equation}\label{sw-recurrence}
\begin{aligned}
&s_{2j} = L_1 s_{2j-2} + f_{2j}, 
\qquad w_{2j} = L_2 w_{2j-2} + g_{2j},
\end{aligned}
\end{equation}
where $f_{2j}$ and $g_{2j}$ are also multilinear expressions as above, of order $j-1$, respectively $j-\frac12$, but with the additional property that they are  non-endpoint, i.e. they contain at 
least two factors of the form $\partial^{2+} r$ or $\partial^{1+} v$.
\end{lemma}

\begin{proof}
 We begin with the first relation, for which we first discuss the generic case $j\geq 3$. We begin expanding the expression of $s_{2j}$, and then continue calculating the LHS of \eqref{sw-recurrence}. We have
\begin{equation}
\label{s2jL1L2}
s_{2j}= (\kappa r\Delta +\nabla r\cdot \nabla) (r_{2j-2} -\nabla r \cdot w_{2j-3}) +f_{2j}.\end{equation}
The LHS expands as follows
\begin{equation}
\label{s2j exp}
\begin{aligned}
s_{2j}&=r_{2j}-\nabla r\cdot w_{2j-1}=D^{2j}_tr -\nabla r\cdot D^{2j-1}_t v.
\end{aligned}
\end{equation}
Each of the two terms appearing in the expression above can be further analyzed. For the first term on the RHS in \eqref{s2j exp} we have
\begin{equation}
    \label{process}
D^{2j}_tr =D^{2j-2}_t (D^2_t r)=D_t^{2j-2} \left(\kappa r\Delta r +\kappa^2 r \vert \nabla \cdot v\vert ^2+\kappa \nabla v (\nabla v)^T\right).
\end{equation}
The last two terms already satisfy the non-endpoint property, so we are left to process the first term on the RHS of \eqref{process} further:
\[
D_t^{2j-2} \left(\kappa r\, \Delta r\right)=\kappa \sum_{m=0}^{2j-2-m} \begin{pmatrix}
  2j-2 \\
 m \\
\end{pmatrix} D_t^{2j-2-m}r\, D_t^m \Delta r.
\]
We note that $D_t^m \Delta r$ gives at least  $\partial ^{2+}r$ derivatives, and, for any  $m\neq 2j-2$ the claim is obvious, as we have that one material derivative on $r$ will produce $\partial^{1+}v$ derivatives. Hence, the more difficult case is when $m=2j-2$; we  discuss it further:
\begin{equation}
\label{DL-rec}
\begin{aligned}
\kappa r\, D_t^{2j-2} \Delta r &= \kappa r\, D_t^{2j-3} \left( D_t\Delta r\right) = \kappa r\, D_t^{2j-3} \left( D_t\Delta r\right).
\end{aligned}
\end{equation}
We commute the material derivative with the Laplacian using the formula 
\begin{equation}
\label{DL commutator}
\left[ D_t, \Delta \right]  =-\Delta v \cdot \nabla -\nabla v \, \nabla^2 ,
\end{equation}
and \ref{DL-rec} gives
\begin{equation}
\label{mad}
\begin{aligned}
\kappa r\, D_t^{2j-2} \Delta r & = \kappa r\, D_t^{2j-3} \left( D_t\Delta r\right)\\
& = \kappa r\, D_t^{2j-3} \left( \Delta \, D_t r - \Delta v\cdot \nabla r -\nabla v \nabla^2r
\right)\\
& = \kappa r\, D_t^{2j-3} \left( \Delta \, D_t r \right) - \kappa r D_t^{2j-3} \left( \Delta v\cdot \nabla r\right) - \kappa rD_t^{2j-3} \left(\nabla v \nabla^2r
\right).\\
\end{aligned}
\end{equation}
The  last term in the expression above gets absorbed in $f_{2j}$. For the next to last term we have
\[
-\kappa r D_{t}^{2j-3}(\Delta v \nabla r)=-\kappa r \sum_{k=0}^{2j-3}\begin{pmatrix}
  2j-3 \\
 \kappa\\
\end{pmatrix} D_t^{2j-3-k} (\Delta v)\,   D_t^k (\nabla r). 
\]
We distribute and commute all the material derivatives to observe that all but one term are readily in $f_{2j}$ (commuting $D_t$ with $\nabla$, or even better with $\Delta$ gives rise to $\nabla v\cdot \nabla$, respectively \eqref{DL commutator} terms, which ensures the non-endpoint property), namely
\[
\kappa D_t^{2j-3} (\Delta v)\,   \nabla r.
\]
For this we need commute the material derivatives with $\Delta$:
\begin{equation}\label{mad+}
\begin{aligned}
 D_t^{2j-3} (\Delta v)\,   \nabla r &= \Delta (D_t^{2j-3}) v\,   \nabla r\\
 &= [D_t^{2j-3}, \Delta ]v\,  \nabla r+\Delta  (D^{2j-3}_t v) \, \nabla r\\
 &=[D_t^{2j-3}, \Delta ]v\, \nabla r+\Delta v_{2j-3} \, \nabla r.
 \end{aligned}
\end{equation}
The first term above is in $f_{2j}$ and the last term is part of the expression in \eqref{s2jL1L2}.

For the first term in \eqref{mad}, we commute $D_t^{2j-3}$ with the Laplacian 
\[
\kappa r\, D_t^{2j-3} \left( \Delta \, D_t r \right) =\kappa r\,  \left\{\Delta D_t^{2j-3} \,( D_tr ) + \left[ D_t^{2j-3} , \Delta  \right] \, D_t r   \right\} .
\]
We observe that the first term on the RHS above is $\kappa r \Delta D_t^{2j-3} \,( D_tr )= \kappa r\Delta r_{2j-2} $ which is one of the terms on the RHS of the expansion in \eqref{s2jL1L2}. The last terms is included in $f_{2j}$, as the commutator $\left[ D_t^{2j-3} , \Delta  \right]$, for $j\geq 2$,  will produce at least one of each terms in $ \left\{\nabla v, \, \nabla r\right\}$.

We now deal with the last term in \eqref{s2j exp} 
\begin{equation}
    \label{lasts2j}
\begin{aligned}
-\nabla  r\cdot D_t^{2j-1}v &= -\nabla  r\cdot D_t^{2j-2}(-\nabla r) \\
&=\nabla  r\cdot D_t^{2j-3}( D_t\nabla r)\\
&=\nabla  r\cdot D_t^{2j-3}( [D_t, \, \nabla]  r +\nabla\,  D_tr)\\
&=\nabla  r\cdot D_t^{2j-3}( -\nabla v\cdot \nabla r  +\nabla\,  D_tr).
\end{aligned}
\end{equation}
For the first term on the RHS of \eqref{lasts2j}we get
\[
\begin{aligned}
-\nabla  r\cdot D_t^{2j-3}( \nabla v\cdot \nabla r)&=- \nabla r \cdot \sum_{k=0}^{2j-3} \begin{pmatrix}
  2j-3 \\
 k \\
\end{pmatrix}  D_t^{2j-3-k}(\nabla v)\,  D_t^{k}\nabla r \\
&=- \nabla r \cdot \sum_{k=0}^{2j-3} \begin{pmatrix}
  2j-3 \\
 k \\
\end{pmatrix}  D_t^{2j-3-k}(\nabla v)\,  D_t^{k-1}(D_t\nabla r),
\end{aligned}
\]
where we can, by inspection, see that almost all the terms are in $f_{2j}$, except for the case
$k=0$, i.e. the term $D_t^{2j-3} (\nabla v) \nabla r$. As before, we have
\[
D_t^{2j-3} (\nabla v) \nabla r=[ D_t^{2j-3} , \, \nabla ] v \nabla r +\nabla r \nabla D_t^{2j-3}v= [ D_t^{2j-3} , \, \nabla ] v \nabla r + \nabla r \nabla v_{2j-3} ,
\]
where the first terms is in $f_{2j}$ and the last one (together with $\nabla r$ from \eqref{lasts2j}) gives another term in \eqref{s2jL1L2}, namely 
\begin{equation}\label{mad++}
\nabla r \nabla v_{2j-3} \nabla r.
\end{equation}

Lastly, we return to the last term in \eqref{lasts2j},
\[
\nabla r \cdot D_t^{2j-3} (\nabla D_t r),
\]
which we rewrite as
\[
\nabla r \cdot D_t^{2j-3} (\nabla D_t r)= \nabla r \cdot ( [ D_t^{2j-3}, \nabla] D_t r + \nabla (D_t^{2j-2}r))= \nabla r \cdot  [ D_t^{2j-3}, \nabla] (D_t r) +\nabla r \cdot \nabla r_{2j-2}.
\]
This finishes the proof of  the \eqref{sw-recurrence} for the $s_{2j}$ formula in the case $j\geq 3$: the first term is part of the $f_{2j}$ and the last one appears in \eqref{s2jL1L2}.

The argument for the case $j=2$ is similar. The only difference occurs at the very end, where 
we collect the contribution of  last term in \eqref{mad+} 
(with the corresponding $\kappa r$ factor)
and the expression in \eqref{mad++} and rewrite them as follows:
\[
\kappa r \nabla r \Delta \nabla r + \nabla r \nabla^2 r \nabla r = L_1(\frac12 |\nabla r|^2) + 
\kappa r |\nabla^2 r|^2 ,
\]
where the last term goes into $f_4$.

\bigskip

For the $w_{2j}$ there is no difference in the case $j=2$. The formula we are asked to show is 
\begin{equation}
\label{w2j expand}
w_{2j}=\kappa \nabla (r\nabla \cdot w_{2j-2})+ \nabla (\nabla r\cdot w_{2j-2})+g_{2j}.
\end{equation}
As before, we expand the LHS of \eqref{w2j expand} and peel off the terms that belong to $g_{2j}$, and then inspect that the remaining terms match its RHS
\[
\begin{aligned}
w_{2j}=D_t^{2j-1} (D_t v)&= -D_t^{2j-1} (\nabla r) = -D_t^{2j-2} (D_t \nabla r)=-D_t^{2j-2} ( \nabla D_tr +\nabla v\cdot \nabla r),
\end{aligned}
\]
which gives
\[
w_{2j}=\kappa \left\{ \nabla D_t^{2j-2} (r\nabla \cdot v) + [D_t^{2j-2}, \nabla] (r\nabla \cdot v)\right\} -D_t^{2j-2}(\nabla v\cdot \nabla r):=I +II+III.
\]
The commutator terms $II$ gets absorbed in $g_{2j}$. For $I$ we note that all but one of the terms have the non-endpoint property, namely $\kappa \nabla (r \nabla D_t^{2j-2}v) = \kappa \nabla (r \nabla w_{2j-2}) $, which is part of the RHS of \eqref{w2j expand}. Lastly, for the $III$ we have
\[
D_t^{2j-2}(\nabla v\cdot \nabla r)=\sum_{m=0}^{2j-2-m} \begin{pmatrix}
  2j-2 \\
 m \\
\end{pmatrix}  D_t^{2j-2-m}(\nabla v) \cdot D_t^m (\nabla r),
\]
The case $m=0$ gives 
\[
D_t^{2j-2}(\nabla v)\cdot \nabla r=(\nabla D_t^{2j-2}v +[D_t^{2j-2},\nabla ]v)\cdot \nabla r.
\]
the commutator term belongs to $g_{2j}$, and hence we are left with
\[
\nabla v_{2j-2} \cdot \nabla r,
\]
which is again part of the RHS of \eqref{w2j expand}. 

\end{proof}

To take advantage of the above recurrence lemma, we will need
a pair of elliptic estimates for the operators $L_1$, $L_2$. 
There is one small matter to address, which is that we would like these
bounds to depend only on our control parameter $A$, whereas $L_2$ contains
second derivatives of $r$ in the coefficients. This can be readily rectified
by replacing $L_2$ by 
\begin{equation}
\tilde L_2 = \kappa \nabla r \nabla + \nabla r  \nabla
\end{equation}
or in coordinates, to avoid ambiguity in notations,
\begin{equation}
(\tilde L_2)_{ij} = \kappa \partial_i  r \partial_j + \partial_j r  \partial_i
\end{equation}
We note that the difference between $L_2 w$ and $\tilde L_2 w$ is the expression $\nabla^2 r w$,  whose contribution can be 
harmlessly placed in $g_{2j}$ in \eqref{sw-recurrence}.

Set
\[
\sigma := \frac{1}{2\kappa}.
\]
Then we have

\begin{lemma}\label{l:coercive}
Assume that $A$ is small. Then the following elliptic estimates hold:
\begin{equation}\label{coercive1}
\|s\|_{H^{2,\sigma+\frac12}} \lesssim \|L_1 s\|_{H^{0,\sigma-\frac{1}{2}}} + \|s\|_{H^{0,\sigma+\frac12}},
\end{equation}
respectively
\begin{equation}\label{coercive2}
\|w\|_{ H^{2,\sigma+1}} \lesssim \|\tilde L_2 w\|_{H^{0,\sigma}} + \|\curl w \|_{H^{1,\sigma+1}}+ \|w\|_{ H^{0,\sigma+1}}
\end{equation}
and 
\begin{equation}\label{coercive3}
\|w\|_{ H^{2,\sigma+1}} \lesssim \|(\tilde L_2 + L_3) w\|_{H^{0,\sigma}}+ \|w\|_{ H^{0,\sigma+1}}
\end{equation}
\end{lemma}

\begin{remark}\label{r:support}
We note that in essence this estimate has a scale invariant nature.
The lower order term added on the right plays no role in the proof, and can be dropped 
if either $(s,w)$ are assumed to have small support (by the Poincare inequality), or if 
we use the corresponding homogeneous norms on the left. 
\end{remark}

We will in fact need a more general result, where the $L_1$ and $\tilde{L}_2$ operators are replaced by $L^b_1$ and $\tilde{L}_2^b$, respectively, where $b>0$:

\begin{corollary} \label{c:coercive} The results in Lemma~\ref{l:coercive} also hold when $L_1$ and $\tilde L_2$ are replaced by $L^b_1$ and $\tilde{L}^b_2$, for $b>0$, where
\[
L_1^b = (\kappa r \nabla + (1+b\kappa) \nabla r) \cdot \nabla, \quad \tilde{L}_2^b:=  \kappa  \nabla r \nabla + (1+\kappa b) \nabla r \nabla.
\]
\end{corollary}
This is a direct consequence of the proof of Lemma~\ref{l:coercive}, rather than of the Lemma.

\begin{proof}[Proof of Lemma~\ref{l:coercive}]

We first observe that the bound \eqref{coercive2} is a direct consequence of \eqref{coercive3} since $L_3 w$ is a function of $\curl w$. Hence it suffices to prove \eqref{coercive1} and \eqref{coercive3}.

Before we dwelve fully into the proof, we note that we have
the relatively standard   weaker elliptic bounds
\[
\|s\|_{H^{2,\sigma+\frac12}} \lesssim_A \|L_1 s\|_{H^{0,\sigma-\frac{1}{2}}} + \| s\|_{H^{1,\sigma-\frac12}},
\]
respectively 
\[
\|w\|_{H^{2,\sigma+1}} \lesssim_A \|(\tilde L_2 + L_3) w\|_{H^{0,\sigma}} + \| w\|_{H^{1,\sigma}}.
\] 
For these bounds we only need integration by parts, treating the first order term in both $L_1$ and $\tilde L_2 + L_3$ perturbatively, and using only the pointwise bound for $\nabla r$.
We leave this straightforward computation to the reader. 

Taking the above bounds into account, our bounds \eqref{coercive1}
and \eqref{coercive3} reduce to the scale invariant estimates
\begin{equation}\label{coercive1a}
\|\nabla s\|_{H^{0,\sigma+\frac12}} \lesssim \|L_1 s\|_{H^{0,\sigma-\frac{1}{2}}},
\end{equation}
respectively
\begin{equation}\label{coercive2a}
\|\nabla w\|_{H^{0,\sigma+1}} \lesssim \|(\tilde L_2 + L_3) w\|_{H^{0,\sigma}}.
\end{equation}

We consider first \eqref{coercive1}, where we proceed using 
a simple integration by parts. To avoid differentiating $r$ twice, we  assume that at some point $\nabla r (x_0)= e_n$. Then in our domain we have 
\[
|\nabla r - e_n| \lesssim A \ll 1.
\]
We compute
\[
\begin{split}
\int r^{\frac{1-\kappa}{\kappa}} (\kappa r \nabla + \nabla r) \nabla s \cdot \partial_n s \, dx = & \ 
\int \kappa r^{\frac{1}{\kappa}}  \Delta s \partial_n s +  r^{\frac{1-\kappa}{\kappa}} (|\partial_n s|^2 +O(A) |\nabla s|^2) 
\, dx
\\ = & \ \frac12 \int  r^{\frac{1-\kappa}{\kappa}} |\nabla s|^2 +O(A) |\nabla s|^2 \, dx,
\end{split}
\]
which suffices by the Cauchy-Schwarz inequality.

\bigskip

 Next we consider the bound \eqref{coercive2} for the $v$ component, where 
\[
\begin{aligned}
r^{\frac{1}{\kappa}} ((\tilde L_2 + L_3) w)_i = & \ \kappa[ \partial_i  r \partial_j w_j + \partial_j r (\partial_j w_i- 
\partial_i w_j) ]
+  \partial_j  r \partial_i w_j + \partial_j r ( \partial_j w_i- 
\partial_i w_j) 
\\
= & \ \kappa[ \partial_j (r^{\frac{1}{\kappa}+1} \partial_j w_i)     
+ r^\frac{1}{\kappa}(\partial_i r \partial_j w_j - \partial_j r \partial_i w_j) ].
\end{aligned}
\]
We use  a computation similar to the one before, integrating by parts and using the fact  that all the tangential derivatives of $r$ are $O(A)$ and its normal derivative is $1+O(A)$,
\[
\begin{split}
\int\! r^{\frac{1}{\kappa}}  (\tilde L_2 + L_3) w \! \cdot \! \partial_n w \,dx = &  
\kappa \! \int \!\! - r^{\frac{1}{\kappa}+1} \partial_j w_i  \partial_n \partial_j   w_i   
\!+\!  r^\frac{1}{\kappa}\left((\partial_i r \partial_j w_j\! -\! \partial_j r \partial_i w_j) \partial_n w_i
\! +\! O(A) |\nabla w|^2 \right) dx
\\
 = & 
\kappa \int  r^{\frac{1}{\kappa}} \left[\frac12(\frac{1}{\kappa}+1) |\partial_j  w_i|^2  
+ \partial_j w_j  \partial_n w_n - \partial_n w_j \partial_j w_n + O(A) |\nabla w|^2\right] \,dx.
\end{split}
\]
We claim that the above expression can be bounded from below by 
\[
\geq  \ (1-O(A)) \int  r^{\frac{1}{\kappa}} |\nabla w|^2 \, dx.
\]
To see that, we cancel the two $|\partial_n w_n|^2$ terms, and restricting indices below to $k,m \neq n$, we have to show that 
\begin{equation} \label{cs}
 -\int r^{\frac{1}{\kappa}} (\partial_k w_k \partial_n w_n - \partial_n w_k \partial_k w_n) \, dx
\lesssim \frac12 \int r^{\frac{1}{\kappa}} \left[|\partial_j  w_i|^2 + O(A) |\nabla w|^2\right] \, dx
\end{equation}
Indeed, we can bound the expression on the left by Cauchy-Schwarz as
\[
 -\int r^{\frac{1}{\kappa}} (\partial_k w_k \partial_n w_n - \partial_n w_k \partial_k w_n) \, dx
\lesssim \frac12 \int r^{\frac{1}{\kappa}}( |\sum_{k=1}^{n-1}\partial_k w_k|^2 + |\partial_n w_n|^2 + \sum_{k=1}^{n-1} (|\partial_n w_k|^2 + \partial_k w_n|^2)) \, dx.
\]
If we can establish that the first term on the right admits the equivalent representation
\[
\int r^{\frac{1}{\kappa}}( |\sum_{k=1}^{n-1}\partial_k w_k|^2 \, dx = 
\int r^{\frac{1}{\kappa}}( \sum_{k,m=1}^{n-1}\partial_k w_m \partial_m w_k  + O(A) |\nabla w|^2)\, dx,
\]
then \eqref{cs} follows by one more application of Cauchy-Schwarz. This last bound, in turn, 
reduces to the relation
\begin{equation}\label{tang-ip}
   I_{km}:= \int r^{\frac{1}{\kappa}} ( \partial_k w_m \partial_m w_k - \partial_m w_m \partial_k w_k) \, dx 
= O(A) \int r^{\frac{1}{\kappa}} |\nabla w|^2 \, dx .
\end{equation}
In the model case $r=x_n$, the left hand side is exactly zero, integrating by parts. In the general 
case, we arrive at almost the same result after a more careful integration by parts:
\[
\begin{aligned}
I_{km} = & \int r^{\frac{1}{\kappa}+1} \partial_n( \partial_k w_m \partial_m w_k - \partial_m w_m \partial_k w_k) \, dx+ O(A) \int r^{\frac{1}{\kappa}} |\nabla w|^2 \, dx 
\\
 = & \int r^{\frac{1}{\kappa}+1} \partial_k( \partial_n w_m \partial_m w_k - \partial_m w_m \partial_n w_k)
 + \partial_m ( \partial_n w_k \partial_k w_m - \partial_k w_k \partial_n w_m) \, dx
 \\ & + O(A) \int r^{\frac{1}{\kappa}} |\nabla w|^2 \, dx 
 \\
= &  O(A) \int r^{\frac{1}{\kappa}} |\nabla w|^2 \, dx
\end{aligned}   
\]
This concludes the proof of \eqref{tang-ip}, and thus the proof of the lemma.

\end{proof}

The above set-up suffices in order to prove our coercivity bounds.
We will successively establish the estimates
\begin{equation}\label{coercive-iterate}
\| (s_{2j-2},w_{2j-2})\|_{\H^{2k-2j+2}} \lesssim 
\| (s_{2j},w_{2j})\|_{\H^{2k-2j}} + O(A)\|(r,v)\|_{\H^{2k}}
, \qquad 1 \leq j \leq k .
\end{equation}
Concatenating these bounds we get the desired estimates in the theorem, where the errors are absorbed using the smallness 
condition $A \ll 1$.

The case $j = k$ follows directly from Lemma~\ref{l:coercive} above,
using the interpolation estimates to 
get smallness for $(f_{2k},g_{2k})$, in the sense that
\begin{equation}
\| (f_{2k},g_{2k})\|_{\H} \lesssim_A  A \|(r,v)\|_{\H^{2k}}.
\end{equation}

The case $ 2 \leq  j < k$ requires an additional argument. 
Precisely, we will apply Lemma~\ref{l:coercive}
to functions $(s,w)$ of the form
\[
s = L s_{2j-2}, \qquad w = L w_{2j-2},
\]
where $L$ is any operator in the right class,
\[
L = r^a \partial^b, \qquad 2a \leq b \leq 2(k-j).
\]
In order to do that we need to have a good relation between
$L(s_{2j},w_{2j})$ and $L(s_{2j-2},w_{2j-2})$. To achieve this,
we apply $L$ in \eqref{sw-recurrence}. For $s_{2j}$ this yields
\[
L_1 L s_{2j-2} = L s_{2j} -[L,L_1] s_{2j-2} - L f_{2j},
\]
where we need to  examine more closely the commutator term. 
To keep the analysis simple it suffices to argue by induction on $a$, beginning with $a=0$. All terms in the commutator, where at least one $r$ factor gets differentiated twice, are non-endpoint terms, and can be estimated by interpolation. All terms in the commutator where two $r$ factors get differentiated are taken care
of by the induction in $a$. Finally, all terms where only one $r$
term is differentiated are also taken care of by the induction in $a$, unless $a=0$. Thus if $a > 0$ then all commutator terms are estimated either as error terms or  via the induction hypothesis.

So the only nontrivial case is when $a =0$. In this case
it is convenient to consider a frame $(x',x_n)$ adapted to the free surface, so that 
\[
|\partial' r| \lesssim A, \qquad |\partial_n r - 1| \lesssim A .
\]
Then all commutators with tangential derivatives are error terms,
and the only nontrivial commutator terms are those with $\partial_n$. For these, we write modulo good $O(A)$ error terms
\[
[\partial_n^b, L_1] \approx  b \Delta \partial_n^{b-1}
\approx  b \nabla r \cdot \nabla  \partial_n^{b}
+ b  \partial_n^{b-1} (\partial')^2.
\]
The contribution of the first term on the right can be included 
in $L_1$, akin to a conjugation. The contribution of the 
second term on the right can be viewed as an induction term
if we phrase the argument as an induction in the number $b$
of normal derivatives. Then we can write
\[
\partial_n^b L_1 \approx L_1^b \partial_n^b,
\]
where 
\[
L_1^b = (\kappa r \nabla + (1+b\kappa) \nabla r) \cdot \nabla
\]
for which we can still apply the analysis in Lemma~\ref{l:coercive}.

Finally, we consider the case $j=1$, where the relation in Lemma~\ref{l:recurrence} is not exactly true, but it is essentially true once we differentiate at least twice. Precisely, we compute
\[
s_2 = \kappa r \Delta r + \frac12 |\nabla r|^2 + r O(|\nabla v|^2).
\]

Instead of comparing $s_2$ with $L_1 s_0$,
we compare $L s_2$ with $L_1 L s_0$ where 
as before $L=r^a \partial^b$. Here we must have 
$b \geq 2$, so we begin with the case $a = 0$
and $b=2$. For tangential derivatives we get modulo $O(A)$ error terms
\[
\partial^b s_2 \approx  L_1 \partial^b s_0,
\]
while for normal derivatives 
\[
\partial_n^b s_2 \approx L_1^b \partial_n^b s_0.
\]
From here on the argument is similar to the $j > 2$ case.

The analysis is similar in the case of $L_2$, which, we recall,
has the form
\[
L_2  = \nabla ( \kappa r \nabla + \nabla r) .
\]
For this we can write a similar conjugation relation, again modulo $O(A)$ perturbative and induction terms,
\[
\partial_n^b L_2 \approx L_2^b \partial_n^b,
\]
where 
\[
L_2^b = \nabla ( \kappa r \nabla + (1+\kappa b) \nabla r).
\]
Substituting $L_2^b$ with $\tilde L_2^b$, we can then apply the elliptic bounds in Corollary~\ref{c:coercive}.
\end{proof}

\subsection{Energy estimates}
Here we prove energy estimates in $\H^{2k}$ for solutions $(r,v)$. We recall the equations.
\begin{equation} \label{original-eq}
\left\{
\begin{aligned}
& r_t  + v \nabla r + \kappa r \nabla v = 0
\\
& v_t + (v \cdot \nabla) v  + \nabla r = 0,
\end{aligned}
\right.
\end{equation}
or, with $D_t$:

\begin{equation} \label{Dt-eq}
\left\{
\begin{aligned}
& D_t r + \kappa r \nabla v = 0
\\
& D_t v  + \nabla r = 0.
\end{aligned}
\right.
\end{equation}

We will also use the transport equation for $\omega = \curl v$,
\begin{equation}
D_t \omega = - \ \omega\cdot  \nabla v -(\nabla v)^T \omega.
\end{equation}

Now we consider the higher Sobolev norms $\H^{2k}$. For these we will prove the following:

\begin{theorem}
The energy functional $E^{2k}$ in $\H^{2k}$  has the following two properties:

a) Norm equivalence:
\begin{equation}\label{e:equiv}
E^{2k} (r, v) \approx_A \|(r,v)\|_{\H^{2k}}^2.
\end{equation}

b) Energy estimate:

\begin{equation}\label{e-est}
\frac{d}{dt} E^{2k} (r, v)  \lesssim_A B \|( r, v)\|_{\H^{2k}}^2.
\end{equation}
\end{theorem}

The first part of the theorem, i.e. the coercivity, was proved in the previous subsection.
To prove the second part of the theorem we will separately estimate the time derivative of each component in 
$E^{2k}$. The first step in that is to derive the equations satisfied
by the functions used in the definition of the energy.
\bigskip

I) The wave component. Here we  will show that $(s_{2k},w_{2k})$ is a good approximate solution to the linearized equation:

\begin{lemma}\label{l:good-var-lin}
Let $k \geq 1$. The functions $(s_{2k},w_{2k})$ solve the equations
\begin{equation}\label{sw-eqn}
\left\{
\begin{aligned}
&D_t s_{2k}   + w_{2k} \cdot \nabla r + \kappa  r \nabla w_{2k} =  f_{2k} \\
& D_t  w_{2k}  + \nabla s_{2k}  = g_{2k} ,
\end{aligned}
\right.
\end{equation}
where $f_{2k}$ and $g_{2k}$ are non-endpoint\footnote{We recall that this means 
that there is  no single factor in $f_{2k}$,
respectively $g_{2k}$ which has order larger that $k-1$, respectively $k-\frac12$. Equivalently, each of them has at least two $\partial^{2+} r$
or $\partial v$ factors.} multilinear expressions in $r$, $\nabla v$ of order $k-\frac12$, respectively $k$, with exactly $2k+1$ derivatives.
\end{lemma}

\begin{proof}
The assertions about the order and the number of derivatives are obvious.
It remains to show that
no single factor in $f_{2k}$,
respectively $g_{2k}$  has order larger that $k-1$, respectively $k-\frac12$. 
In other words,  we want to see that each product in $f_{2k}$,
respectively $g_{2k}$, has at least two factors of the form $\partial^{2+} r$ or $\partial^{1+}v$.

We begin with $f_{2k}$:
\[
\begin{aligned}
f_{2k} = & \  D_t (D_t^{2k} r  - \nabla r \cdot D_t^{2k-1} v) + \nabla r \cdot D_t^{2k} v + \kappa r \nabla D_t^{2k} v
\\
= & \ - \kappa( D_t^{2k}( r \nabla v)  - r \nabla D_t^{2k} v) - D_t(\nabla r) D_t^{2k-1}v.
\end{aligned}
\]
The first term has a commutator structure involving $[D_t^{2k}, r \nabla]$ which yields at least a $\nabla v$ 
coefficient. The same happens with $D_t \nabla r$ in the second term.

We continue with $g_{2k}$:
\[
\begin{split}
g_{2k} = & \  D_t^{2k+1} v + \nabla ( D_t^{2k} r  - \nabla r \cdot D_t^{2k-1} v)
\\
= & \ - D_t^{2k} \nabla r + \nabla  D_t^{2k} r -  \nabla r \cdot \nabla  D_t^{2k-1} + \nabla^2 r \nabla D_t^{2k-1}v.
\end{split}
\]
Here we are commuting $D_t^{2k}$ with $\nabla$, which yield at lest a $\nabla v$ term. The only 
case when we do not get the desired structure is if the commutator occurs at the level of the last $D_t$,
\[
[D_t^{2k},\nabla] = [D_t^{2k-1}, \nabla] D_t + D_t^{k-1} [D_t,\nabla].
\]
The contribution of the first term is always balanced. However, for the second term we have
\[
 [D_t,\nabla] r = - \nabla v \cdot \nabla r.
 \]
Thus we get a possibly an unbalanced contribution if all of $D_t^{2k-1}$ applies to $v$.
We obtain,
\[
g_{2k} =  \partial_i r \partial_j D_t^{2k-1} v_i - \partial_i r \partial_j  D_t^{2k-1} v_i + \mbox{balanced} = \mbox{balanced}.
\]
The computation for $k = 1$ is similar but simpler, and it is omitted.

\end{proof}

\bigskip

II) The transport component. Here the functions 
whose weighted $L^2$ norms we are trying to propagate
are denoted by $\omega_{2k}$, and have the form
\begin{equation}\label{omega-2k}
\omega_{2k}= r^a \partial^b \omega, \qquad |b| \leq 2k-1, \quad b-a = k-1.
\end{equation}
For these functions we have

\begin{lemma}\label{l:omega-2k}
The functions $\omega_{2k}$ are approximate solutions for the transport equation
\begin{equation}\label{dt-omega-2k}
  D_t \omega_{2k} = h_{2k},
\end{equation}
where  $h_{2k}$ are non-endpoint multilinear expressions in $r$, $\nabla v$ of order $2k$ with exactly $k$ derivatives.
\end{lemma}

\begin{proof}
We compute the transport equation
\[
D_t \omega_{2k} = h_{2k},
\]
where we write schematically
\[
\begin{aligned}
h_{2k} = D_t ( r^{k} \partial^{2k-1} \omega) &= [D_t,r^k \partial^{2k-1}] \omega - r^k \partial^{2k-1} (\nabla v)^2.
\end{aligned}
\]
This proves that all terms in $h_{2k}$ are balanced, since all commutators include $\nabla v$ factors.
\end{proof}

To conclude the proof of the energy estimates it remains to bound 
the time derivative of the linearized energies
\[
\| (s_{2k},w_{2k})\|_{\H}^2, \qquad \|\omega_{2k}\|_{L^2_\sigma}^2
\]
by $\lesssim_A B \|(r,v)\|_{\H^{2k}}$. In view of our energy estimates for the linearized equation, respectively the transport equation, in order 
to obtain the desired estimate it suffices to bound the source terms $(f_{2k},g_{2k})$, respectively $h_{2k}$:

\begin{lemma}\label{l:balanced-interp}
The expressions $f$ and $g$ above satisfy the scale invariant bounds
\begin{equation}\label{fg-2k-est}
\|(f_{2k},g_{2k}) \|_{\H}+  \|h_{2k} \|_{H^{0,\sigma}} \lesssim_A B \| (r,v)\|_{\H^{2k}}.
\end{equation}
\end{lemma}

\begin{proof}
This follows using our interpolation inequalities in Propositions~\ref{p:interpolation}
\ref{p:interpolation-c} and \ref{p:interpolation-d},
following the same argument as in the proof of part (a) of Theorem~\ref{t:coercive}.

The control parameter $A$ gives $L^\infty$ control at degree $0$, i.e. for $\|\nabla r\|_{L^\infty}$ and $\|v\|_{\dot C^\frac12}$,
and $B$ gives $L^\infty$ control at degree
$\dfrac12$, i.e. for $\|\nabla v\|_{L^\infty}$ and $\|\nabla r\|_{\tC}$.

We consider the factors in each multilinear expression in $f_{2k}$, $g_{2k}$ and $h_{2k}$
as follows. The factors of order $-\frac12$ (i.e. the $r$ factors are interpreted as weights,
and distributed to the other factors. The factors of order $0$ in $f_{2k},g_{2k},h_{2k}$ (i.e. $\partial r$  factors) are directly estimated in $L^\infty$ by $A$ and discarded.
The factors of maximum order are estimated directly by $\|(r,v)\|_{\H^{2k}}$.
The intermediate factors can be estimated in $L^p$ norms in two ways, by interpolating the $\H^{2k}$ norm with $A$, or by interpolating with $B$.

Overall the product needs to be estimated in $L^2$, using exactly one
$\|(r,v)\|_{\H^{2k}}$ factor. Then a scaling analysis 
shows that we will have to use exactly one $B$ norm, i.e. for instance 
for monomials $f_{2k}^m$ of order $m$ in $f_{2k}$ we have
\[
\| f_{2k}^m \|_{H^{0,\sigma-\frac12}} \lesssim A^{m-2} B \|(r,v)\|_{\H^{2k}}.
\]
This is exactly as in the proof of Theorem~\ref{t:coercive}(a); the details are left for the reader.

\end{proof}

\bigskip

\section{Construction of regular solutions}
\label{s:reg-exist}
This section contains the first part of the proof of our well-posedness
result; precisely, here we give a constructive proof of existence of regular 
solutions.  The rough solutions will be obtained 
in the following section as unique limits of regular solutions.

Given an initial data $(r_0,v_0)$ with regularity 
\[
(r_0,v_0) \in \bfH^{2k},
\]
where $k$ is assumed to be sufficiently large, we will construct 
a local in time solution with a lifespan depending on the $\bfH^{2k}$ size of the data.
Unlike all prior works on this problem, which use parabolic regularization methods
in Lagrangian coordinates, here we propose a new approach, implemented fully 
within the setting of the Eulerian coordinates.

Our novel method  is loosely based on nonlinear semigroup methods,
where an approximate solution is constructed by discretizing the problem in time. 
Then the challenge  is to carry out a time step construction which, on one hand, 
is as simple as possible, but where, on the other hand, the uniform in time energy bounds survive. In a classical semigroup approach this would require solving an elliptic 
free boundary problem, with very precise estimates. At the other extreme,
in a pure ode setting one could simply use an Euler type method. The Euler method
cannot work here, because it would loose derivatives. A better alternative would be 
to combine an Euler method with a transport part; this would reduce, but not eliminate 
the loss of derivatives.

The idea of our approach is to retain the simplicity of the Euler + transport method,
while preventing the loss of derivatives by an initial regularization step.
Then the regularization step becomes the more delicate part of the argument, 
because it also needs to have good energy bounds. To achieve that, we carry out 
the regularization in a paradifferential fashion, but in a setting where 
we are avoiding the use of complicated classes of pseudodifferential operators. 
Thus, in a nutshell, our  solution is to divide and conquer,
splitting the time step into three:
\begin{itemize}
    \item Regularization
    \item Transport
    \item Euler's method,
\end{itemize}
where the role of the first two steps is to improve the error estimate in the third step.

To summarize, our approach provides a new, simpler method to construct solutions in the context of free
boundary problems. Further, we believe it will prove useful in a broader class of problems.

\subsection{ A few simplifications}
In order to keep our construction as simple as possible, we observe 
here that we can make a few simplifying assumptions:

\medskip

i) By finite speed of propagation and Galilean invariance, we can assume
that $v$ vanishes and $r$ is linear outside a small compact set.

\medskip

ii) Given the reduction in step (i), the coercivity bound \eqref{coercive3} 
proved  in Lemma~\ref{l:coercive} carries over to the operator $L_2+L_3$.
This yields a natural div-curl orthogonal decomposition for $v$ in $\H$,
\[
v = L_2(L_2+L_3)^{-1} + L_3(L_2+L_3)^{-1} v:=v_1+v_2
\]
where the first component is a gradient and the second depends only on $\curl v$.
In particular, it follows that we have 
\[
\begin{aligned}
\| \curl v \|_{H^{2k-1,\frac{1}{\kappa}}}^2 = & \ \| \curl v_2 \|_{H^{2k-1,\frac{1}{\kappa}} }^2 \approx  \sum_{j=0}^k \|(L_2+L_3)^j v_2\|_{H^{0,\frac{1}{\kappa}}}^2
\\
\approx & \ \| \curl v\|_{H^{0,\frac{1}{\kappa}}}^2+ \sum_{j=1}^k \|L_3^j v\|_{H^{0,\frac{1}{\kappa}}}^2
\end{aligned}
\]
where we refer the reader to Lemma~\ref{l:coercive-k} below for the second step.
This allows us to make the simplified choice
\begin{equation}\label{transport-en}
E^{2k}_t(r,v) = \| \curl v\|_{H^{0,\frac{1}{\kappa}}}^2+ \sum_{j=1}^k \|L_3^j v\|_{H^{0,\frac{1}{\kappa}}}^2.
\end{equation}
for the transport component of the energy.

\subsection{Construction of approximate solutions}

Given a small time-step $\epsilon > 0$ and an initial data $(r_0,v_0) \in \bfH^{2k}$ we will produce a discrete approximate 
solution $(r(j\epsilon), v(j\epsilon))$, with the following properties:

\begin{itemize}
\item (Norm bound)  We have
\begin{equation}
E^{2k}(r((j+1)\epsilon),v ((j+1)\epsilon)) 
\leq (1+ C \epsilon) E^{2k}(r((j \epsilon),v (j\epsilon)) .
\end{equation}

\item (Approximate solution) 
\begin{equation}
\label{good-iterate}
\left\{
\begin{aligned}
& r((j+1)\epsilon) - r(j \epsilon)  + \epsilon \left[v(j\epsilon) \nabla r(j\epsilon)+ \kappa r (j \epsilon)   \nabla \cdot v(j\epsilon)\right] = O(\epsilon^{1+}) 
\\
&v((j+1)\epsilon) - v(j \epsilon)    
+ \epsilon \left[(v(j\epsilon) \cdot \nabla) v(j\epsilon) + \nabla r(j\epsilon)\right] = O(\epsilon^{1+}). 
\end{aligned}
\right.
\end{equation}
\end{itemize}

The first property will ensure a uniform energy bound for our sequence.
The second property will guarantee that in the limit we obtain an exact solution.
There we can use a weaker topology, where the exact choice of norms is not so important.

Having such a sequence of approximate solutions, it will be a fairly simple matter to  produce, as the   limit on a subsequence, an exact solution $(r,v)$ on a short time interval which stays bounded in the above topology. 
The key point is the construction of the above sequence. It suffices to carry out a single step:

\begin{theorem}\label{t:onestep}
Let $k$ be a large enough integer. Let $(r_0,v_0)$ with regularity 
\begin{equation}\label{data-size}
E^{2k}(r_0,v_0) \leq M,
\end{equation}
and $\epsilon \ll 1$. Then there exist a one step iterate $(r_1,v_1)$ with the following properties:

\begin{enumerate}
\item (Norm bound)  We have
\begin{equation}\label{good-energy}
 E^{2k}(r_1,v_1) \leq (1+ C(M) \epsilon) E^{2k} (r_0,v_0) 
\end{equation}

\item (Approximate solution) 
\begin{equation}
\label{good-iterate1}
\left\{
\begin{aligned}
& r_1 - r_0  + \epsilon[v_0 \nabla r_0  + \kappa r_0 \nabla v_0] = O(\epsilon^{2}) 
\\
& v_1 - v_0  + \epsilon [(v_0 \cdot \nabla) v_0 +  \nabla r_0] = O(\epsilon^{2}) .
\end{aligned}
\right.
\end{equation}
\end{enumerate}
\end{theorem}

The remainder of this subsection is devoted to the proof of this
theorem. 

\bigskip

We begin with an obvious observation, namely that a direct iteration (Euler's method) loses derivatives. A better strategy would be to separate 
the transport part; this reduces (halves) the derivative loss, but does not fully eliminate it.
However, if  we precede this by an initial regularization step, then we can avoid the loss of derivatives altogether. In a nutshell, this will be our strategy.
We begin with the outcome of the regularization step.

\begin{proposition}\label{p:reg-step}
Given  $(r_0,v_0) \in \bfH^{2k}$ as in \eqref{data-size},   there exists a regularization $(r,v)$
with the following properties:
\begin{equation}\label{epsilon2}
r-r_0 = O(\epsilon^{2}), \qquad v - v_0 = O(\epsilon^{2}),
\end{equation}
respectively
\begin{equation}\label{e-growth}
E^{2k} (r,v)  \leq (1+ C \epsilon) E^{2k} (r_0,v_0) ,
\end{equation}
and
\begin{equation}\label{e-high}
\| (r,v)\|_{\H^{2k+2}}   \lesssim \epsilon^{-1}M.
\end{equation}
\end{proposition}

We postpone for the moment the proof of the proposition, and instead we show how to use it 
in order to prove the result in Theorem~\ref{t:onestep}.

\begin{proof}[Proof of Theorem~\ref{t:onestep}]
Here we construct $(r_1,v_1)$ starting from $(r,v)$ given by the last proposition. Naively the remaining steps are 
the  Euler iteration
\[
\left\{
\begin{aligned}
& r_1  = r - \epsilon  \kappa r \nabla v 
\\
& v_1 = v - \epsilon   \nabla r,
\end{aligned}
\right.
\]
and the  flow transport
\begin{equation}\label{flow-transport}
x_1 = x + \epsilon v(x).
\end{equation}
The important point is that these two steps cannot be carried out separately, as each of them  taken alone seems to be unbounded. Instead, taken together there is an extra cancellation to be taken advantage of, which is the direct analogue of a similar cancellation in the energy estimates.
Using the transport as above,  $(r_1,v_1)$ are defined as follows:
\begin{equation}\label{newton}
\left\{
\begin{aligned}
& r_1(x_1)  = r(x) - \epsilon  \kappa r(x) \nabla v(x) ,
\\
& v_1(x_1) = v(x) - \epsilon   \nabla r(x).
\end{aligned}
\right.
\end{equation}

It remain to show that these have the properties in the proposition. 
We begin by observing that 
\[
r_1(x_1) = r(x) (1+ O(\epsilon)),
\]
so these can be used interchangeably as weights. We also have
\[
d x_1 = dx(1+ O(\epsilon)),
\]
so the same can be said for the measures of integration.

We successively compute $D_t$ derivatives of $(r_1,v_1)$
in terms of similar derivatives of $(r,v)$. We will work with
operators of the form $ D_t^{2j}$. As before, when applied to a
data set $(r,v)$, these are interpreted as multilinear partial differential expressions, as if they were applied to a solution and then re-expressed,
using the equations, in terms of the initial data. In particular, we 
recall that the expressions $D_t^{2j} r$ and $D_t^{2j} v$ have orders $(j-2)/2$, respectively $(j-1)/2$.

Switching from derivatives in $x$ to derivatives in $x_1$ is done by repeated 
applications of the chain rule, which involves the Jacobian 
\[
J = (I + \epsilon Dv)^{-1}.
\]
Thus in this calculation we will not only produce multilinear expressions,
but also powers of $J$. To describe errors, we will enhance our standard notion of 
order by assigning the order $-\frac12$ to $\epsilon$;
this is natural because as a time step, $\epsilon$ can be thought of as the dual variable to $D_t$.
Such a choice will ensure that the expression $\epsilon \nabla v$ has order $0$, and that all our relations below  are homogeneous. Then we have
\begin{lemma}\label{l:approx-eqn}
a) The following algebraic relations hold: 
\begin{equation} \label{approx-sw-eqn}
\left\{
\begin{aligned}
D_t^{2j}  r_1(x_1) = & \  D_t^{2j}  r(x) + \epsilon D_t^{2j+1}  r(x)
+ \epsilon^2 R_{2j}(r,v,\epsilon \nabla v)(x)  \\
D_t^{2j}  v_1(x_1) = & \  D_t^{2j}  v(x) + \epsilon D_t^{2j+1}  v(x)
+ \epsilon^2 V_{2j}(r,v,\epsilon\nabla v)(x),
\end{aligned}
\right.
\end{equation}
where $R_{j}$ and $V_{j}$ are multilinear expressions in $(r,\nabla v,\epsilon \nabla v)$ and their derivatives, and also $J$, with the following properties:
\begin{itemize}
\item $v$ does not appear undifferentiated.
\item They have order $2$ respectively $j+1/2$.
\item In addition to powers of $J$, they contain exactly $2j+2$ derivatives
applied to factors of $r$, $v$ or $\epsilon \nabla v$.
\item They are balanced, i.e. they contain at least two $\partial^{2+} r$ or
$\partial^{1+} v$ factors.
\end{itemize}

b)  Similar relations hold for $\omega= \curl v$ and its weighted derivatives $\omega_{2j}$
\begin{equation}\label{approx-omega-eqn}
\omega_{2j,1}(x_1)=\omega_{2j}(x) -\epsilon h_{2j} -\epsilon^2 W_{2j}(\omega, v, \epsilon\nabla v)(x).
\end{equation}
where $h_{2j}$ is as in \eqref{dt-omega-2k} and $W_{2j}$ has the same properties as 
$R_{2j}$ and $V_{2j}$ above.
\end{lemma}

\begin{proof} We prove part (a), as part (b) is similar.
As discussed earlier, transcribing the expression $D_t^j r_1(x_1)$ in terms of $r$ and $v$  is based on repeated application of chain rule, which involves the Jacobian 
\[
J = (I + \epsilon Dv)^{-1},
\]
and yields contributions of order zero.
Thus one easily obtains 
\begin{equation}
\left\{
\begin{aligned}
D_t^j  r_1(x_1) = & \  D_t^j  r(x) 
+ \epsilon \tilde R_{j}(r,v,\epsilon \nabla v)(x)  \\
D_t^j  v_1(x_1) = & \  D_t^j v(x)  
+ \epsilon \tilde V_{j}(r,v,\epsilon \nabla v)(x),
\end{aligned}
\right.
\end{equation}
where $\tilde R_{j}$ and $\tilde V_{j}$ are multilinear
expressions in $(r,\nabla v,\epsilon \nabla v)$ and with added powers of $J$
and which have order $(j-1) /2$, respectively $j/ 2$, and exactly $j+1$ derivatives applied to factors of $r,v$ or $\epsilon \nabla v$.

It remains to identify the coefficients of the $\epsilon$ terms, which are
\[
( \tilde R_{j}(r,\nabla v,0), \tilde V_{j}(r,\nabla v,0)).
\]
Identifying $\epsilon$ with time $t$, and redenoting $(r_1,v_1) = (r(t), v(t))$,
we have
\[
( \tilde R_{j}(r,\nabla v,0), \tilde V_{j}(r,\nabla v,0)) = \frac{d}{dt}
(D_t^{j}  r(x), D_t^{j}  v(x)),_{t = 0}.
\]

But by construction the functions  $(r(t), v(t))$ solve the equation at $t=0$,
so the desired identification holds.
\end{proof}

Returning to the proof of the theorem, we note that the above lemma already 
gives the bound \eqref{good-iterate1} in the uniform topology. It remains to prove the bound \eqref{good-energy}, where we have to compare
$E^{2k}(r,v)$ with $E^{2k}(r_1,v_1)$. We recall that these energies 
have the wave component and the curl component. These are treated in a similar way, so we will focus on the wave component which is more interesting. For this we need to compare the $L^2$ type norms of the good variables
\[
\| (s_{2k},w_{2k})\|_{\H_r}^2, \qquad \| (s_{1,2k},w_{1,2k})\|_{\H_{r_1}}^2.
\]
The lower order norms also need to be compared, but that is a straightforward matter.
Note that these norms are represented as integrals over different domains. However, we identify these domains
via \eqref{flow-transport}, and we compare the corresponding densities accordingly.

For exact solutions, the good variables solve the linearized equations with source
terms \eqref{sw-eqn}. For our iteration, the above lemma yields a similar relation with additional source terms,
\begin{equation}\label{sw-eqn-app}
\left\{
\begin{aligned}
&s_{2k,1} = s_{2k}  -\epsilon( w_{2k} \cdot \nabla r + \kappa  r \nabla w_{2k}) - \epsilon f_{2k} + \epsilon^2 R_{2k} \\
& w_{2k,1} = w_{2k}  - \epsilon  \nabla s_{2k} - \epsilon g_{2k} + \epsilon^2 V_{2k} ,
\end{aligned}
\right.
\end{equation}
where $f_{2k},g_{2k}$ are perturbative source terms as in Lemma~\ref{l:good-var-lin}, and $(R_{2k},V_{2k})$ are as in the lemma above.
The terms $(f_{2k},g_{2k})$ satisfy the bound \eqref{fg-2k-est} in Lemma~\ref{l:balanced-interp}, which we recall here:
\[
\|(f_{2k},g_{2k})\|_{\H} \lesssim_A B \|(r,v)\|_{\H^{2k}} ,
\]
which is what allows us to treat them as perturbative.

In a similar fashion, Lemma~\ref{l:balanced-interp} shows that the expressions $(R_{2k},V_{2k})$ satisfy 
\[
\|(R_{2k},V_{2k})\|_{\H} \lesssim_A B \|(r,v)\|_{\H^{2k+1}} .
\]
Since these terms have an $\epsilon^2$ factor, the bound \eqref{e-high} 
also allows us to treat them as perturbative.

It remains to estimate the main expression, for which we compute
\[
\begin{split}
E_1 = & \ \| (s_{2k}  -\epsilon( w_{2k} \cdot \nabla r + \kappa  r \nabla w_{2k}),
 w_{2k}  - \epsilon  \nabla s_{2k})(x_1) \|_{\H_{r_1}}^2
\\
= & \ \| (s_{2k}  -\epsilon( w_{2k} \cdot \nabla r + \kappa  r \nabla w_{2k}),
 w_{2k}  - \epsilon  \nabla s_{2k})\|_{\H_{r}}^2 + C(M) \epsilon
\\
 = & \ \| (s_{2k}, w_{2k}) \|_{\H}^2 
 -2 \epsilon
 \langle (s_{2k},w_{2k}),  ( w_{2k} \cdot \nabla r + \kappa  r \nabla w_{2k},
  \nabla s_{2k}) \rangle_{\H}
  \\ &  + \epsilon^2
 \| ( w_{2k} \cdot \nabla r + \kappa  r \nabla w_{2k},
   \nabla s_{2k}) \|_{\H}^2  + C(M) \epsilon .
 \end{split}
\]
The second term can be seen to vanish after integrating by parts;
this is the same cancellation seen in the proof of the energy estimates for the 
linearized equation.
The third term, on the other hand, can be estimated as an error term via \eqref{e-high},
\[
\| ( w_{2k} \cdot \nabla r + \kappa  r \nabla w_{2k},
   \nabla s_{2k}) \|_{\H} \lesssim \| (s_{2k},w_{2k})\|_{\H}^\frac12
 \| (s_{2k},w_{2k})\|_{\H^2}^\frac12, \lesssim_M \epsilon^{-1}.
\]
This concludes the proof of the theorem.
\end{proof}

Now we return to the proof of our regularization result in Proposition~\ref{p:reg-step}.

\begin{proof}[Proof of Proposition~\ref{p:reg-step}]

We begin with a heuristic discussion, for which the starting point and the first candidate
is the regularization already constructed in Proposition~\ref{p:reg-state}, with the 
matched parabolic frequency scale $2^{-2h} = \epsilon$.  This will satisfy the properties \eqref{epsilon2} and \eqref{e-high}, but it is not accurate enough for \eqref{e-growth}.

To improve on this and construct a better regularization 
we need to understand its effect on the energies, and primarily on the leading energy term
which is $\| (s_{2k},w_{2k})\|_{\H}^2$. For this we need to better understand the expressions 
for  $(s_{2k},w_{2k})$. We have seen earlier that we have the approximate relations
\[
s_{2k} \approx L_1 s_{2k-2}, \qquad w_{2k} \approx L_2 w_{2k-2},
\]
so one might expect that we have 
\[
s_{2k}  \approx L_1^k r, \qquad w_{2k} \approx L_2^k v.
\]
However, this is not exactly accurate, as one can see by considering the first relation for $k = 1$.
There 
\[
s_2 = \kappa r \Delta r + \frac12 |\nabla r|^2,
\]
whereas 
\[
L_1 r =  \kappa r \Delta r +  |\nabla r|^2.
\]
To rectify this discrepancy, we will interpret the operators $L_1$ and $L_2$ in a paradifferential fashion,
i.e. decouple the $r$ appearing in the coefficients of $L_1$ and $L_2$ from the $r$ in the argument of $L_1^k$.
Instead, the $r$ in the coefficients will be harmlessly replaced with a regularized version of itself, call it $\rm$
and correspondingly $L_1$ and $L_2$ will be replaced by $L^-_1$, $L^-_2$. Then we will be able to write approximate relations of the form
\[
s_2 \approx L_1^- (r-\rm)  +  s^-_2,
\]
and further 
\[
s_{2k}  \approx (L^-_1)^k (r-\rm) + s^-_{2k},
\]
and similarly for $w_{2k}$.

Based on these considerations, we will construct our regularization as follows:

\begin{itemize}
\item Start with  the initial state $(r_0,v_0) \in \bfH^{2k}$.

\item Produce two initial regularizations $\rp$ and $\rm$ of $r_0$, on  scales
$h^+ > h > h^-$, with slightly larger domains, and then restrict them to $\Omega^{-}= 
\{ \rm > 0\}$.

\item Use the selfadjoint operators $L_1$ and $L_2+L_3$ associated to $\rm$ to 
regularize the high frequency part $(\rp-\rm,\vp-\vm)$ within $\Omega^-$ below frequency $2^h$.

\item Obtain the $h$ scale regularization $(\tr,\tv)$ of $(r_0,v_0)$ in $\Omega^-$, by adding the low frequency part $(\rm,\vm)$ to the regularized high frequency part.

\item Decrease $\tr$ by a small constant $c=O(\epsilon^4)$ and set $(r,v) = (\tr-c,\tv)$, in order to ensure that $\Omega:=\left\{ r >0\right\} \subset \Omega^-$. 
\end{itemize}

\bigskip

\textbf{1. A formal computation and the good variables.}
Both in order to motivate the definition of our regularization, and as a tool to prove we have the correct regularization, here we consider the question of  comparing  the good variables $(s^0_{2k},w^0_{2k})$ associated to $(r_0,v_0)$ with $(\ts_{2k},\tw_{2k})$ associated to 
$( \tr, \tv)$. The lemma below is purely algebraic, and makes no reference to the relation between $(r_0,v_0)$ and $(\tr,\tv)$.

Each term in $(s_{2k},w_{2k})$ is a multilinear expression of the same order
in $(r,v)$, so we will view the difference
\[
(s^0_{2k},w^0_{2k}) - (\ts_{2k},\tw_{2k})
\]
as a multilinear expression in $(r_0-\tr,v_0-\tv)$ and $(\tr,\tv)$.
Heuristically we will think of the first expression as the high frequency part
of $(r_0,v_0)$ and the second expression as the low frequency part.
Since we are working here in high regularity, the intuition is that 
high-high terms will be  better behaved and can be assigned to the error. Explicitly, we write
\begin{equation}\label{Dsw1}
\left\{
\begin{aligned}
s^{0}_{2k} = &\ \ts_{2k} + Ds_{2k}(\tr,\tv) (r_0-\tr,v_0-\tv) + F_{2k}
\\
w^{0}_{2k} = &\ \tw_{2k} + Dw_{2k}(\tr,\tv) (r_0-\tr,v_0-\tv) + G_{2k},    
\end{aligned}
\right.
\end{equation}
where $Ds_{2k}$ and $Dw_{2k}$ stand for the differentials of 
$s_{2k}$ and $w_{2k}$ as functions of $(r,v)$. This is akin to a paradifferential expansion of 
$(s_{2k}^0,w_{2k}^0)$.
In this expansion all terms on each line have the same order, which is
$k-1$, respectively $k-\frac12$, and $(F_{2k},G_{2k})$ are at least bilinear
in the difference $(r_0-\tr,v_0-\tv)$.

The high-high terms $(F_{2k},G_{2k})$ will play a perturbative role in our analysis. This leaves us with the terms  which are linear in the difference, i.e. the low-high terms involving the two differentials $Ds_{2k}$ and $Dw_{2k}$. We will further simplify this, by observing that 
the low-high terms where the low frequency factor is differentiated (i.e. has order $> 0$)
are also favourable. This leaves us only with low-high terms with top order 
in the high frequency factor in the leading part. These terms 
are identified in the following lemma:

\begin{lemma}\label{l:Dsw}
We have the algebraic relations
\begin{equation} \label{Dsw2}
\left\{
\begin{aligned}
Ds_{2k}(\tr,\tv) (r_0-\tr,v_0-\tv) = &\ \ (L_{1}(\tr))^k (r_0-\tr) + \tilde F_{2k}
\\
Dw_{2k}(\tr,\tv) (r_0-\tr,v_0-\tv) = & \  (L_{2}(\tr))^k (v_0-\tv) + 
\tilde G_{2k},    
\end{aligned}
\right.
\end{equation}
where the error terms $(\tilde F_{2k},\tilde G_{2k})$ are linear in 
$(r_0-\tr,v_0-\tv)$,
\[
\tilde F_{2k} = D^1_{2k}(\tr,\tv)(r_0-\tr,v_0-\tv),
\qquad \tilde G_{2k} = D^2_{2k}(\tr,\tv)(r_0-\tr,v_0-\tv),
\]
whose coefficients are multilinear differential expressions in 
$(\tr,\tv)$ which contain at least one factor with order $> 0$,
i.e. $\partial^{2+} \tr$ or $\partial^{1+}\tv $.
\end{lemma}
We remark that combining \eqref{Dsw1} and \eqref{Dsw2} we obtain the expansion
\begin{equation} \label{Dsw3}
\left\{
\begin{aligned}
s^{0}_{2k} = &\ \ts_{2k} + (L_{1}(\tr))^k (r_0-\tr) + F_{2k} + \tilde F_{2k}
\\
w^{0}_{2k} = &\ \tw_{2k} + (L_{2}(\tr))^k (v_0-\tv) + G_{2k} + \tilde G_{2k},    
\end{aligned}
\right.
\end{equation}
where all terms on each line are multilinear expressions  in $(r_0-\tr,v_0-\tv)$ and $(\tr,\tv)$ of order $k-1$, respectively  $k-\frac12$, and whose multilinear error terms have either:
\begin{enumerate}[label=\alph*)]
\item   (high-high) two difference factors, i.e. $(F_{2k},G_{2k})$ or
\item   (low-high balanced) exactly one difference factor,
and at least one nondifference factor with order $> 0$, i.e. $(\tilde F_{2k},\tilde G_{2k})$.
\end{enumerate}
One should think of the above expansions as paradifferential linearizations,
but implemented without using the paraproduct formalism.

\begin{proof}
Our starting point is provided by the relations \eqref{sw-recurrence}, 
differentiated with respect to $(r,v)$. 
This yields
\[
Ds_{2j} = L_1(r) D s_{2j-2}- DL_1(r) s_{2j-2} + Df_{2j}, \qquad j \geq 2.
\]
Since the expression $f_{2j}$ is balanced, its differential can be included
in $D^1_{2k}$. Similarly, the second expression on the right also has 
terms of order $> 0$ in $(r,v)$. Thus we get 
\begin{equation}\label{s2k-para}
D s_{2j} = L_1(r) Ds_{2j-2}+\tilde F_{2j}, \qquad j \geq 2.
\end{equation}
Next we turn our attention to the case $j=1$, where we have
\[
s_{2} = \kappa r \Delta r - \frac12 |\nabla r|^2 + f_2,
\]
therefore
\[
Ds_2 = \kappa r \Delta  + \kappa  \Delta r   - \nabla r \nabla
+Df_2,
\]
where the second and forth terms are admissible errors, so we also get 
\eqref{s2k-para}. Then the conclusion of the lemma follows by reiterated 
use of \eqref{s2k-para}. The argument for $w_{2k}$ is similar.
\end{proof}

\textbf{2. Regularizations for $(r_0,v_0)$.}
We begin with the dyadic frequency scale  $h$ matching the time step $\epsilon$, in a parabolic fashion, namely $2^{-2h} = \epsilon$. As mentioned earlier, the direct regularization $(r^h,v^h)$ of $(r_0,v_0)$
given by Proposition~\ref{p:reg-state} is not a sufficiently accurate regularization, in that 
it satisfies the properties \eqref{epsilon2} and \eqref{e-high}, but not necessarily \eqref{e-growth}.

Nevertheless, we will still use Proposition~\ref{p:reg-state} to bracket our desired regularization as follows. Starting with the frequency scale $h$ we define a lower 
and a higher frequency scale
\[
1 \ll h^- < h < h^+,
\]
where  $h^-$ and $h^+$ will be chosen later to satisfy a specific set of constraints.
We remark for now that this is a soft choice, in that there is a large range of parameters 
that will work.

Correspondingly we consider the regularizations given by Proposition~\ref{p:reg-state},
denoted by
\[
(\rp,\vp)= (r^{h^+},v^{h^+}), \qquad (\rm,\vm)= (r^{h^-},v^{h^-}).
\]
These regularizations are defined on the enlarged domains
$\tOmega^{[h^+]}$, respectively $\tOmega^{[h^-]}$.
We will use them on the domain $\Omega^{-} = \{\rm > 0\}$. By Proposition~\ref{p:reg-state},
this domain's boundary  is at distance at most $ 2^{-2 h^{-}(k-k_0+1)}$ from the original boundary 
$\Gamma_0$. In order to ensure that  $(\rp,\vp)$ are defined on this domain, we will impose the 
constraint
\begin{equation}\label{constraint1}
h^+ < h^-(k-k_0+1)   . 
\end{equation}
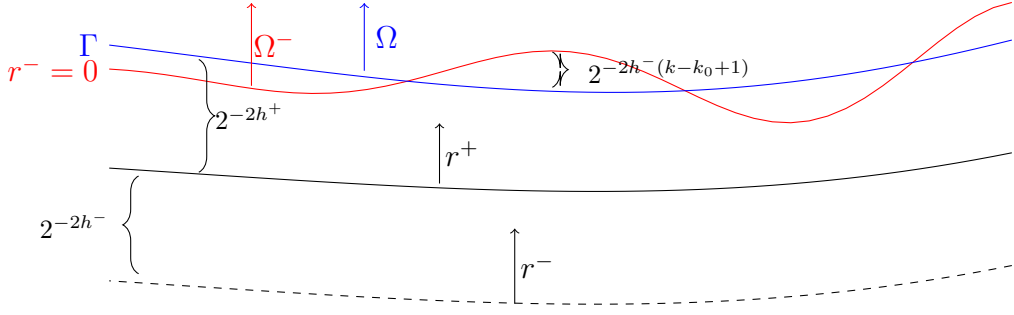
\begin{figure}
\begin{tikzpicture}
 \draw[red] [samples=50,domain=0:12] plot  (-\x,{cos(\x r)*exp(-\x/6)}) node[left, align=left] (script){$r^{-}=0$};
\draw[blue] [samples=50,domain=0:12] plot (-\x,{-sin(\x r/4)*exp(-\x/16)+0.5}) node[left, align=left] (script){$\Gamma$};
\draw[black] [samples=50,domain=0:12] plot (-\x,{-sin(\x r/5)*exp(-\x/10)-1 }) node[left, align=left] (script){};
\draw[black, dashed] [samples=50,domain=0:12] plot (-\x,{-sin(\x r/5)*exp(-\x/10)-2.5 }) node[left, align=left] (script){};
\draw [decorate,decoration={brace,amplitude=6pt,mirror},xshift=-60pt,yshift=-60pt]
(-9.5,.8) -- (-9.5,-.5) node [black,midway,xshift=-0.5cm] {\footnotesize  $2^{-2h^-} \qquad $};
\draw [decorate,decoration={brace,amplitude=6pt,mirror},xshift=-60pt,yshift=-60pt]
(-8.7,.85) -- (-8.7,2.35) node [black,midway,xshift=0.5cm] {\footnotesize \quad $2^{-2h^+}$};

\draw [->={at=(-9,2)}, red, decorate, decoration={amplitude=6pt,mirror},xshift=-60pt,yshift=-60pt](-8,2) -- (-8,3.1) node [black,midway,xshift=.3cm] {${\color{red}\Omega^{-}}$} ;

\draw [->={at=(-6.5,2.2)}, blue, decorate, decoration={amplitude=6pt,mirror},xshift=-60pt,yshift=-60pt](-6.5,2.2) -- (-6.5,3.1) node [black,midway,xshift=.3cm] {${\color{blue}\Omega}$} ;

\draw [->={at=(-5.5,1.5)}, decorate, decoration={amplitude=6pt,mirror},xshift=-60pt,yshift=-60pt](-5.5,.7) -- (-5.5,1.5) node [black,midway,xshift=.3cm] {$r^{+}$} ;

\draw [<-={at=(-4.5,-.5)}, decorate, decoration={amplitude=6pt,mirror},xshift=-60pt,yshift=-60pt](-4.5,.1) -- (-4.5,-.9) node [black,midway,xshift=.3cm] {$r^{-}$} ;

\draw [decorate,decoration={brace,amplitude=6pt,mirror},xshift=-60pt,yshift=-60pt]
(-4,1.98) -- (-4,2.45) node [black,midway,xshift=--0.5cm] {\footnotesize \quad  $\qquad \quad \qquad 2^{-2h^-(k-k_0+1)}  $};
\end{tikzpicture}
\label{reg.pic}
\caption{Domains associated with the regularization scheme.}
\end{figure}

\bigskip

 We will think of $ (\rm,\vm)$ as a ``sub''-regularization, which has to be 
a part of $(\tr,\tv)$, and of $(\rp,\vp)$ as a ``super"-regularization, in that 
$(\tr,\tv)$ will be a regularization of it. We arrive at $(r,v)$ in two steps:

\begin{enumerate}[label=\roman*)]
\item  We define our first regularization $(\tr,\tv)$ as smooth functions in 
$\Omega^-$ as follows:
\begin{equation}\label{choose-tr}
\tr := \rm + \chi_\epsilon(L_1(r_-)) (\rp-\rm)    ,
\end{equation}
\begin{equation}\label{choose-tv}
\tv := \vm + \chi_\epsilon((L_2+L_3)(r_-)) (\vp-\vm)  ,  
\end{equation}
where $\chi_\epsilon(\lambda) := \chi(\lambda \epsilon)$, with $\chi$
a smooth, positive bump function with values in $(0,1)$ 
and the following asymptotics: 
\begin{equation}\label{choose-chi}
\begin{aligned}
\chi(\lambda) & \ \approx 1-\lambda \qquad \text{near } \lambda = 0  
\\
\chi(\lambda) & \ \approx 1/\lambda \qquad \text{near } \lambda = \infty 
\end{aligned}
\end{equation}

\item The functions $(\tr,\tv)$ in $\Omega^-$ are not yet the desired regularizations as $\tr$ does not vanish on the boundary $\Omega^-$. If it were negative there, we would simply restrict them to $\Omega = \{ r > 0\}$. Unfortunately, all we know is that for some large $C$ we have
\[
|\tr| \ll 2^{-2Ch} \qquad \text{ on }\ \  \Gamma^-.
\]
Then we define 
\begin{equation}
(r,v):= (\tr - 2^{-2Ch},\tv)
\end{equation}
restricted to $\Omega= \{r > 0\}$ as our final regularization.
\end{enumerate}

\bigskip
\textbf{3. Bounds for the regularization $(\tr,\tv)$.}
To start with, we have the bounds for $(r^{\pm},v^{\pm})$ from Proposition~\ref{p:reg-state}.
So here we consider the bounds for $(\tr,\tv)$.

\begin{lemma}  \label{l:trv-reg}
Assume that $\|(r_0,v_0)\|_{\bfH^{2k}} \leq M$. Then the following estimates hold for 
$(\tr,\tv)$ in $\Omega^-$:
\begin{equation}\label{rv1-high}
\|(\tr,\tv)\|_{\H^{2k+2j}_{\rm}} \lesssim_M 2^{2hj}, \qquad j = 0,1 , 
\end{equation}
respectively
\begin{equation}\label{rv1-low}
\|(\rp-\tr,\vp - \tv)\|_{\H^{2k-2}_{\rm}} \lesssim_M 2^{-2h}.  
\end{equation}
\end{lemma}

\begin{proof}
a) With $L_1 = L_1(\rm)$ and similarly for $L_2$ and $L_3$, we have the obvious bounds
\[
\| (L_1^{k+j} \tr,(L_2+L_3)^{k+j} \tv)\|_{\H} \lesssim 2^{2hj} (\| (L_1^k \rp,(L_2+L_3)^{k}\vp)\|_{\H} + \| (L_1^k \rm,(L_2+L_3)^{k}\vm)\|_{\H})
\lesssim_M 2^{2hj}.
\]
Then \eqref{rv1-high} follows from elliptic bounds for $L_1$, respectively $L_2+L_3$, which for convenience 
we collect in the next Lemma:

\begin{lemma}\label{l:coercive-k}
Assume that $r$ satisfies 
\[
\| (r,0) \|_{\H^{2k}} \leq M ,
\]
and 
\[
\| (r,0) \|_{\H^{2k+2j}} \leq M 2^{2hj}, \qquad 0 < j \leq N.
\]
Then we have the estimates
\begin{equation}\label{ell-lo}
\|(s,w)\|_{\H^{2k}} \lesssim_M \sum_{l=0}^k \| (L_1^l s, (L_2+L_3)^{l} w)\|_{\H} ,   
\end{equation}
respectively
\begin{equation}\label{ell-hi}
\|(s,w)\|_{\H^{2k+2j}} \lesssim_M \epsilon^{-2j} \sum_{l=0}^{k+j} \| (L_1^l s, (L_2+L_3)^{l} w)\|_{\H}   \qquad 0 < j \leq N. 
\end{equation}
\end{lemma}
\begin{proof}
The estimates
in \eqref{ell-lo}, respectively \eqref{ell-hi}  will follow from the bounds
\begin{equation}\label{ell-lo-a}
\| (s,w) \|_{\H^{2m}}     \lesssim_M \| (L_1 s, (L_2+L_3) w)\|_{\H^{2m-2}}  \qquad 1 \leq m \leq k-1,
\end{equation}
respectively
\begin{equation}\label{ell-hi-a}
\| (s,w) \|_{\H^{2k+2j}}     \lesssim_M \| (L_1 s, (L_2+L_3) w)\|_{\H^{2k+2j-2}} +
\sum_{l=0}^{j-1} 2^{-2h(j-l)} \| ( s, w)\|_{\H^{2m+2l}}  \qquad  j \geq 1.
\end{equation}
The bounds for $s$ and the bounds for $w$ are independent of each other. As the arguments are similar,
we will prove the bounds for $s$ and leave the bounds for $w$ for the reader. 
We begin with  \eqref{ell-lo-a}, where we have to estimate
\[
\| s\|_{H^{2m,m+\sigma}}, \qquad \sigma = \frac{\kappa-1}{2\kappa}.
\]
To achieve this we will inductively bound the norms 
\[
\| s\|_{H^{m+a,a+\sigma}}, \qquad a = \overline{0,m}.
\]
For the induction step, we need to bound
\[
\| L s \|_{H^{2,\sigma}}, 
\]
where $L= r^{a-1} \partial^{m-2+a}$ is an operator of order $m-1$. 
By Lemma~\ref{l:coercive} we have 
\[
\| L s \|_{H^{2,\sigma+1}} \lesssim \| L_1 L s\|_{H^{0,\sigma}}
\lesssim \| L_1  s\|_{H^{2m-2,\sigma+m}} + \| [L,L_1] s\|_{H^{0,\sigma}}.
\]
The commutator $[L,L_1]$ has order $m$, but at most $2m-1$ derivatives. Hence 
by H\"older's inequality and interpolation we can estimate
\begin{equation}\label{com-1}
\| [L,L_1] s\|_{H^{0,\sigma}} \lesssim \| s \|_{H^{m+a-1,a+\sigma -1}}.
\end{equation}
Thus we obtain
\[
\| s\|_{H^{m+a,a+\sigma}} \lesssim \| L_1 s\|_{H^{m-2,m-1+\sigma}} + \| s \|_{H^{m+a-1,a+\sigma -1}},
\]
which concludes the induction step.

It remains to consider the initial case $a=0$, 
where we simply take $L = \partial^{m-1}$. Here we argue as in the proof of Theorem~\ref{t:coercive}, more precisely the 
bound \eqref{coercive-iterate}; in an adapted frame we split the derivatives into normal and tangential, $L = \partial_n^b \partial_\tau^c$, and conjugate
\[
 L L_1 = L_1^b L + R,
\]
where the remainder $R$ has $O(A)$ contributions only,
\begin{equation}\label{com-2}
\| R s\|_{H^{0,\sigma}} \lesssim_M A \| s\|_{H^{m,\sigma}}.
\end{equation}
Applying Lemma~\ref{l:coercive} for $L_1^b$ we obtain
\[
\| s\|_{H^{m,\sigma}} \lesssim \| L_1 s\|_{H^{m-2,m-1+\sigma}} + A \| s \|_{H^{m,\sigma}},
\]
where the error term on the right can be absorbed on the left.

Turning now our attention to the $s$ component of \eqref{ell-hi-a}, the argument is entirely similar,
with a slight modification in the commutator bounds \eqref{com-1} and \eqref{com-2}.
These are in turn replaced by 
\begin{equation}\label{com-1-hi}
\| [L,L_1] s\|_{H^{0,\sigma}} \lesssim \| s \|_{H^{j+k+a-1,a+\sigma -1}} 
+ \sum_{l = 0}^{j-1} 2^{-2h(l-j)} \|s\|_{H^{2k+2l,k+l+\sigma}}, \qquad L = r^{a-1} \partial^{k+j-2+a},
\end{equation}
respectively
\begin{equation}\label{com-2-hi}
\| R s\|_{H^{0,\sigma}} \lesssim_M A \| s\|_{H^{2k+2l,\sigma}} +
 \sum_{l = 0}^{j-1} 2^{-2h(l-j)} \|s\|_{H^{2k+2l,k+l+\sigma}}, \qquad L = \partial^{k+j-1}.
\end{equation}
The $O(A)$ terms in the last bound arise exactly as before when exactly one $L$ derivative 
applies to the $r$ factor in $L_1$. All other contributions have fewer derivatives on $s$, and are estimated
by H\"older's inequality and Sobolev embeddings. The negative $2^{-k}$ powers only arise when more than
$2k$ derivatives apply to the $r$ factors in $L_1$, which means that fewer derivatives apply to $s$.
The details are somewhat tedious but routine, and are omitted.
\end{proof}

\bigskip

We now return to the proof of Lemma~\ref{l:trv-reg}, and turn our 
attention to the bound \eqref{rv1-low}. We have 
\[
(\rp-\tr,\vp - \tv) = ((I - \chi_\epsilon(L_1(\rm)))(\rp-\rm),(I - \chi_\epsilon(L_2+L_3)(\rm)))(\vp-\vm)).
\]
Hence, given the properties of $\chi_\epsilon$, and the above Lemma, we have the $\H$ bound
\[
\|(\rp-\tr,\vp - \tv)\|_{\H^{2k-2}_{\rm}} \lesssim 2^{-2h} \| (L_1(\rm)^k (\rp-\rm),(L_2+L_3)(\rm)^k(\vp-\vm)\|_{\H_{\rm}} \lesssim_M 2^{-2h}.
\]

\end{proof}

\textbf{3. Comparing the energies for $(r_0,v_0)$ and $(\tr,\tv)$.}
Here the first energy is taken in the domain $\Omega_0$, while the second is taken in $\Omega^-$.
Our objective is to prove the following result:

\begin{lemma}\label{l:first-reg}
Assume that $k$ is large enough, and that $h^+$ and $h^-$ are suitably chosen relative to $h$.
Then we have 
\begin{equation}\label{comp-en}
E^{2k}(\tr,\tv) \leq (1+C\epsilon)  E^{2k}(r_0,v_0).
\end{equation}
\end{lemma}
The proof below consists of several steps, each of which will require various constraints on $h^+$ and $h^-$.
These are then collected at the end of the proof in \eqref{choose-hpm}. For orientation, one could simply 
think of the case $h^- = h/2$ and $h^+ = Ch$ with $C \approx k-k_0$.

\begin{proof}
These energies have two components, the wave energy and the transport energy. We will focus on the wave component in the 
sequel, as the argument for the transport part is similar but considerably simpler. For the wave 
component we need to compare the good variables $(s^0_{2k},w^0_{2k})$, respectively $(\ts_{2k}, \tw_{2k})$, 
associated to $(r_0,v_0)$, respectively $(\tr,\tv)$, and their $\H$ norms,
\begin{equation}\label{two-norms}
\| (s^0_{2k}, v^0_{2k})\|_{\H_r}^2  \qquad \text{vs.} \qquad   \| (\ts_{2k}, \tv_{2k})\|_{\H_{\rm}}^2.
\end{equation}
We note that in the second expression we are using the $\H_{\rm}$ norm, as $\rm$ is the defining function for the domain $\Omega^-$ where $(\ts_{2k}, \tv_{2k})$ are defined. As we seek to compare functions on different domains, 
it is natural to restrict them to a common domain. To understand this choice, we recall that the two free boundaries
$\Gamma^0$ and $\Gamma^-$ are at distance  $\ll 2^{-2h^-(k-k_0+1)}$ of each other, and the two weights are at a similar distance within the common domain,
\[
| r - \rm| \ll 2^{-2h^-(k-k_0+1)}.
\]
In order for the difference of the two weights to only yield $O(\epsilon)$ errors, we will restrict our comparison 
to the region $\Omega_0^{[< h^- (k-k_0)+1] -h}$, where we have
\[
| r - \rm| \ll \epsilon r \qquad \text{ in } \Omega_0^{[<h^-(k-k_0+1)-h]}.
\]
Outside this region we will simply neglect the contribution to the first norm
in \eqref{two-norms}. On the other hand we will seek to make  the second norm small
in this region. For this to work, we first need to make sure that the neglected region 
is within the $(\tr,\tv)$ boundary layer, which has width $2^{-2h}$. Thus we require
that 
\[
2h < h^-(k-k_0+1).
\]
But in addition to that, we also want the second norm to be $\epsilon$ small
in this region. Within a fixed layer $\Omega^{-,[h_1]}$ with $h_1 < h$ this norm is
\begin{equation}\label{thin-layer}
\| (\ts_{2k}, \tv_{2k})\|_{\H_{\rm}(\Omega^{-,[h_1]})}^2 \lesssim_M 2^{\frac{2}{\kappa}(h-h_1)},
\end{equation}
which is a consequence of the fact that we are integrating a function which is smooth on the $2^{-2h}$ scale, over a thinner region.
This is $\epsilon^2= 2^{-4h}$ small if
\begin{equation}\label{thin-layer-h}
h_1 > h(1+2\kappa).
\end{equation}
Hence we obtain 
\begin{equation}\label{compare0}
\| (\ts_{2k}, \tv_{2k})\|_{\H_{\rm}(   
\Omega_0^{[>h^-(k-k_0+1)-h}])  } \lesssim \epsilon^2,
\end{equation}
provided that 
\begin{equation}\label{constraint2}
 h^-(k-k_0+1) >    2h(1+\kappa).
\end{equation}

Within $\Omega_0^{[< h^-(k-k_0+1)-h]}$ we use Lemma~\ref{l:Dsw}, more precisely its consequence \eqref{Dsw3},
in order to compare $(s^0_{2k},w^0_{2k})$ respectively $(\ts_{2k}, \tw_{2k})$.
There we seek to estimate the errors perturbatively. We begin with $(F_{2k},G_{2k})$:

\begin{lemma}\label{l:trv-source}
Assume that $(r_0,v_0) \in \bfH^{2k}$, with size $M$ and that 
$(\tr,\tv)$ are defined as above. Then we have the error bounds
\begin{equation}
\| (F_{2k},G_{2k}) \|_{\H_{r}( \Omega_0^{[< h^-(k-k_0+1)-h]})} \lesssim_M \epsilon^2.
\end{equation}
\end{lemma}
The proof of this lemma is similar to the proof of Lemma~\ref{l:balanced-interp}, using interpolation inequalities, and is omitted. Here the region where we evaluate
the norm is less important, and serves only to insure that $r_0$ and $\rm$ 
are both defined and comparable there. The gain comes from the fact that 
the difference $(r-\tr,v-\tv)$ is small at low frequency, which comes 
from \eqref{rv1-low} combined with the bounds for the differences $(r_0-\rp,v_0-\vp)$
in Proposition~\ref{p:reg-state}. The power $\epsilon^2$ requires $k > k_0+2$, but 
one can gain more if $k$ is assumed to be larger.

\medskip

Next we consider the expressions $(\tF_{2k},\tG_{2k})$:
\begin{lemma}\label{l:trv-source-t}
Assume that $(r_0,v_0) \in \bfH^{2k}$, with size $M$ and that 
$(\tr,\tv)$ are defined as above. Then we have the error bounds
\begin{equation}\label{tFG-est}
\| (\tF_{2k},\tG_{2k}) \|_{\H_{r}( \Omega_0^{[< h^-(k-k_0+1)-h]})} \lesssim  \| (\rp-\tr,\vp-\tv) \|_{\H^{2k-1}}+
\epsilon^2 C(M).
\end{equation}
\end{lemma}
\begin{proof}
We recall  that the expressions $(\tF_{2k},\tG_{2k})$ are balanced multilinear 
expressions in $(\tr,\tv)$, respectively $(r_0-\tr,v_0-\tv)$, linear in the second component, containing exactly $2k$ derivatives, and of order $k-1$, respectively $k-\frac12$. The fact that they are balanced allows us to estimate them
using H\"older's inequality and interpolation as in Lemma~\ref{l:balanced-interp},
by 
\[
\| (\tF_{2k},\tG_{2k}) \|_{\H}
\lesssim_{A_0,\tA}( B_0 + A_0 \tB) \|  (\tr,\tv)\|_{\H^{2k-1}} + \tB \|( (r_0-\tr,v_0-\tv)\|_{\H^{2k-1}})
\]
where $A_0,B_0$ respectively $\tA$, $\tB$ are control parameters associated 
to $(\tr,\tv)$, respectively $(r_0-\tr,v_0-\tv)$.

Here the first component $(\tr,\tv)$ is localized at frequencies below $2^h$, while the second  is localized at frequencies above $2^h$. In particular, it follows that 
$A_0,B_0$ are small,
\[
A_0+B_0 \lesssim \epsilon^2,
\]
so their contributions go into the second term on the right in \eqref{tFG-est}.

On the other hand, $\tA$ and $\tB$ are merely bounded $\lesssim_M 1$.
We split 
\[
(r_0-\tr,v_0-\tv) = (r_0-\rp,v_0-\vp)+ (\rp-\tr,\vp-\tv).
\]
The first term is localized at frequencies $\geq 2^{h^+}$ so using 
the bounds in Proposition~\ref{p:reg} we have 
\[
\| (r_0-\rp,v_0-\vp)\|_{\H^{2k-1}} \lesssim_M 2^{-h^+} 
\]
which can be made smaller than $\epsilon^2$ if $h^+ > 4h$. 
The proof of the Lemma is concluded.
\end{proof}

\medskip

Using the above two lemmas together with \eqref{compare0}, we obtain our first relation between the two energies,
\begin{equation}\label{compare-1}
\begin{split}
  \| (\ts_{2k}, \tv_{2k})\|_{\H_{\rm}}^2 & \ \leq     \| (s^0_{2k}, v^0_{2k})\|_{\H_r}^2 + M \|(\rp-\tr,\vp-\tv) \|_{\H^{2k-1}}
  + C(M) \epsilon 
  \\ & \ 
  - 2  
  \langle (L_1(\tr)^{k} (r_0-\tr_1),L_2(\tr)^{k}(v_0-\tv)),(\ts_{2k},\tw_{2k}) \rangle_{\H_{\rm}(\Omega_0^{[< 2h^-(k-k_0)]})} .
\end{split}  
\end{equation}
This is not yet satisfactory, but we can improve it further. We first observe that in the above inner product
we can harmlessly replace the operators $L_1(\tr)$ and  $L_2(\tr)$ by 
$L_1(\rm)$ and  $L_2(\rm)$ respectively. Precisely, we have the difference bound
\[
\| (L_1(\tr)^{k}-L_1(\rm)^{k})  (r_0-\tr),(L_2(\tr)^{k} -L_2(\rm)^{k}) (v_0-\tv))\|_{\H_{\rm}(\Omega_0^{[< 2h^-(k-k_0)]})} \lesssim_M \epsilon .
\]
This is a consequence of interpolation inequalities and H\"older's inequality due the fact that both differences $\tr-\rm$ and $  ((r_0-\tr),(v_0-\tv))$ are concentrated at high frequencies and have small $O(\epsilon^C)$ pointwise size.
The details are left for the reader. We  arrive at 
\begin{equation}\label{compare-2}
\begin{split}
  \| (\ts_{2k}, \tv_{2k})\|_{\H_{\rm}}^2 & \ \leq     \| (s^0_{2k}, v^0_{2k})\|_{\H_r}^2+ M \|(\rp-\tr,\vp-\tv) \|_{\H^{2k-1}}
  + C(M) \epsilon 
  \\ & \ 
  - 2  
  \langle (L_1(\rm)^{k} (r_0-\tr),L_2(\rm)^{k}(v_0-\tv)),(\ts_{2k},\tw_{2k}) \rangle_{\H_{\rm}(\Omega_0^{[< 2h^-(k-k_0)]})} .
\end{split}  
\end{equation}

A second simplification is that we can replace $(r_0,v_0)$ by $(\rp,\vp)$ in the inner product. 
For this we need to show that 
\[
\langle (L_1(\rm)^{k} (\rp- r_0),L_2(\rm)^{k}(\vp-v_0)),(\ts_{2k},\tw_{2k}) \rangle_{\H_{\rm}(\Omega_0^{[< 2h^-(k-k_0)]})} \lesssim_M \epsilon.
\]
We first insert a cutoff $\chi^{[2h^-(k-k_0)]}$  function in the differences on the left, associated to the same boundary layer, which equals $1$ further inside
and $0$ closer to the boundary. This is allowed because the second 
factor in the inner product is already $\epsilon$ small in the cutoff region,
while the first one is still bounded in $\H_{r^{-}}$ in the same region,
provided that the cutoff is lower frequency than the $\rp - r_0$ frequency,
\begin{equation}\label{constraint3}
 h^-(k-k_0) < h^+.
\end{equation}
One should compare this to \eqref{constraint1}; together these bounds give the allowed range 
for $h^+$. With this substitution, we are left with proving that 
\[
\langle (L_1(\rm)^{k} [\chi^{[2h^-(k-k_0)]}(\rp- r_0)],L_2(\rm)^{k}[\chi^{[2h^-(k-k_0)]}(\vp-v_0)]),
(\ts_{2k},\tw_{2k}) \rangle_{\H_{\rm}(\Omega^-)} \lesssim_M \epsilon.
\]
Since $L_1$ and $L_2$ are self-adjoint, we can move one of them to the right.
This becomes
\[
\langle (L_1(\rm)^{k-1} [\chi^{[2h^-(k-k_0)]}(\rp- r_0)],L_2(\rm)^{k-1}[\chi^{[2h^-(k-k_0)]}(\vp-v_0)]),(L_1(\rm)\ts^1_{2k},L_2(\rm)\tw^1_{2k}) \rangle_{\H_{\rm}(\Omega^-)} \lesssim_M \epsilon.
\]
Now the left factor has size $2^{-2h^+} $ and the right factor 
has size $2^{2h}$. This yields an $\epsilon^2$ gain provided that
\begin{equation}\label{constraint4}
h^+ > 4 h.
\end{equation}
Thus, we can replace $(r_0,v_0)$ by $(\rp,\vp)$ in \eqref{compare-2}, to obtain
\begin{equation}\label{compare-3}
\begin{split}
  \| (\ts_{2k}, \tv_{2k})\|_{\H_{\rm}}^2 & \ \leq     \| (s^0_{2k}, v^0_{2k})\|_{\H_r}^2+ C(M) \|(\rp-\tr,\vp-\tv) \|_{\H^{2k-1}_{\rm}}
  + C(M) \epsilon 
  \\ & \ 
  - 2  
  \langle (L_1(\rm)^{k} (\rp -\tr),L_2(\rm)^{k}(\vp-\tv)),(\ts_{2k},\tw_{2k}) \rangle_{\H_{\rm}(\Omega_0^{[< 2h^-(k-k_0)]})} .
\end{split}  
\end{equation}
Once this is done, the expression on the left in the inner product is defined on the entire domain  $\Omega^-$, and we can harmlessly extend the inner product 
to the full region as the expression on the right in the inner product
is already $\epsilon$ small there. We get
\begin{equation}\label{compare-4}
\begin{split}
  \| (\ts_{2k}, \tw_{2k})\|_{\H_{\rm}}^2 & \ \leq     \| (s^0_{2k}, v^0_{2k})\|_{\H_r}^2+ C(M) \|(\rp-\tr,\vp-\tv) \|_{\H^{2k-1}}
  + C(M) \epsilon 
  \\ & \ 
  - 2  
  \langle (L_1(\rm)^{k} (\rp -\tr),L_2(\rm)^{k}(\vp-\tv)),(\ts_{2k},\tw_{2k}) \rangle_{\H_{\rm}(\Omega^-)} .
\end{split}  
\end{equation}

The next step is to apply the expansion \eqref{Dsw3} for the expression on the right
in the inner product to write
\begin{equation} \label{Dsw3+}
\left\{
\begin{aligned}
\ts_{2k} = &\ s^-_{2k} + (L_{1}(\rm))^k (\tr-\rm) + F_{2k}^- + \tilde F_{2k}^-
\\
\tw_{2k} = &\ w^-_{2k} + (L_{2}(\rm))^k (\tv-\vm) + G_{2k}^- + \tilde G_{2k}^-.  
\end{aligned}
\right.
\end{equation}
By the counterpart of Lemma~\ref{l:trv-source} the error terms 
$(F_{2k}^-,G_{2k}^-)$ will be $\epsilon$
small, so their contribution to \eqref{compare-4} can be included in the expression $C(M) \epsilon$.

For the contribution of $(\tF_{2k}^-,\tG_{2k}^-)$ we integrate by parts 
one instance of $L_1$, respectively $L_2$, to bound it by 
\[
\begin{aligned}
\| (L_1(\rm)^{k-1} (\rp -\tr),L_2(\rm)^{k-1}(\vp-\tv)) \|_{\H_{\rm}}
\| (L_1(\rm)\tF_{2k}^-,L_2(\rm)\tG_{2k}^-) \|_{\H_{\rm}} \lesssim_M \\
\|  (\rp -\tr,\vp-\tv) \|_{\H^{2k-2}_{\rm}} \| (\tr-\rm,\tv-\vm)\|_{\H^{2k+1}_{\rm}}
\lesssim_M  2^{-2h} \| (\tr-\rm,\tv-\vm)\|_{\H^{2k+1}_{\rm}}
\end{aligned}
\]
Finally, for the contribution of $(s^-_{2k},w^-_{2k})$ we can integrate again by parts to obtain
\[
\langle (\rp-\tr,\vp-\tv),(L_1(\rm)^{k}s^-_{2k},L_2(\rm)^{k}w^-_{2k}) \rangle_{\H_{\rm}(\Omega^-)} \lesssim_M \epsilon,
\]
provided that 
\begin{equation}\label{constraint5}
(k-1) h > k h^-.
\end{equation}
Thus \eqref{compare-4}
becomes
\begin{equation*}\label{compare-5}
\begin{split}
  \| (\ts_{2k}, \tv_{2k})\|_{\H_{\rm}}^2 & \ \leq     \| (s^0_{2k}, v^0_{2k})\|_{\H_r}^2
  + C(M) \epsilon + 
  \\ & \hspace{-.8in} 
  + C(M)( \|(\rp-\tr,\vp-\tv) \|_{\H^{2k-1}_{\rm}} + 2^{-2h} \| (\tr-\rm,\tv-\vm)\|_{\H^{2k+1}_{\rm}}) 
  \\ & \hspace{-.8in} 
  - 2  
  \langle (L_1(\rm)^{k} (\rp -\tr_1),L_2(\rm)^{k}(\vp-\tv)),  (L_{1}(\rm))^k (\tr-\rm) ,(L_{2}(\rm))^k (\tv-\vm)\rangle_{\H_{\rm}}. \end{split}  
\end{equation*}
Now our choice of $(\tr,\tw)$ guarantees that the inner product is positive.
Combining the above bound with its counterpart for the transport energy 
(this is where our choice \eqref{transport-en} simplifies matters), 
we further obtain
\begin{equation}\label{compare-6}
\begin{split}
 E^{2k}(\tr, \tv) \leq & \  E^{2k}(r^0, v^0) + C(M) \epsilon
  \\ & + C(M)( \|(\rp-\tr,\vp-\tv) \|_{\H^{2k-1}_{\rm}} + 2^{-2h} \| (\tr-\rm,\tv-\vm)\|_{\H^{2k+1}_{\rm}}) - 2 I,
\end{split}  
\end{equation}
where 
\[
I =  \langle (L_1(\rm)^{k} (\rp -\tr_1),(L_2+L_3)(\rm)^{k}(\vp-\tv)),  (L_{1}(\rm))^k (\tr-\rm) ,(L_{2}+L_3)(\rm)^k (\tv-\vm)\rangle_{\H_{\rm}}.
\]
is still positive. Finally, we use the positivity of $I$ to estimate the 
two remaining terms on the right. Precisely, using the properties \eqref{choose-chi} of the multiplier $\chi$ in the definition of $(\tr,\tv)$
as well as the ellipticity of $L_1$, respectively $L_2+L_3$ in the two components 
of $\H$, we have
\[
I \gtrsim 2^{2h} \| (\rp-\tr,\vp-\tv) \|_{\H^{2k-1}_{\rm}}^2 + 2^{-2h}
\| (\rm-\tr,\vm-\tv) \|_{\H^{2k+1}_{\rm}}^2 
\]
Hence, applying the Cauchy-Schwarz inequality in \eqref{compare-6}
we finally obtain 
\begin{equation}\label{compare-7}
 E^{2k}(\tr, \tv) \leq   E^{2k}(r^0, v^0) + C(M) \epsilon,
\end{equation}
as desired.

This concludes the proof of \eqref{comp-en},  provided that the scales 
$h^+$ and $h^-$ were chosen so that the constraints \eqref{constraint1},\eqref{constraint2},\eqref{constraint3},\eqref{constraint4}
are all satisfied. We recall them all here:
\begin{equation}\label{choose-hpm}
\begin{aligned}
   h^- < \frac{k-1}{k} h&, \qquad  h^+ > 4 h,
   \\
h^-(k-k_0) &\, > \, h(1+\frac{1}{\kappa}),
\\
 h^-(k-k_0 ) <& \ h^+ \, < h^-(k-k_0+1).
\end{aligned}
\end{equation}
Then the parameters $h^+$ and $h^-$ can be chosen e.g. as follows:

\begin{enumerate}[label=(\alph*)]
\item set $h^- = h/2$, 
\item  take $k$ large enough so that the second constraint holds,
\item  choose $h^+$ in the range given by the third constraint.
\end{enumerate}
\end{proof}
\bigskip

\textbf{ 4. Comparing the energies of $(\tr,\tv)$ and $(r,v)$.}
To recall our setting here, the functions $(\tr,\tv)$
are defined in the domain $\Omega^{-}$ and are localized at frequency $\leq 2^{h}$
scale, but cannot be though of as a state because $\tr$ does not vanish on the 
boundary $\Gamma^-$. Instead we have 
\[
|\tr| \lesssim 2^{-2(k-k_0)h^-} \qquad \text{on } \Gamma^-.
\]
To rectify this, we decrease $\tr$ by a small constant and set
\begin{equation}
(r,v)  = (\tr-c,\tv), \qquad c = 2^{-2(k-k_0)h^-},
\end{equation}
so that the level set $\Gamma = \{ r = 0 \}$ is fully contained within $\Omega^-$. Then we aim to prove that the energies do not change much:

\begin{lemma}\label{l:subtract-c}
We have the energy bound
\begin{equation}
E^{2k}(r,v) \lesssim E^{2k}(\tr,\tv) + O_M(\epsilon)   . 
\end{equation}
\end{lemma}
\begin{proof}
We separate a boundary layer $\Omega^{[>h_1]}$, with $h_1 > h$ to be chosen later,
where we verify directly that the norm on the left is $O(\epsilon)$. Outside this layer, we compare directly the associated good variables.

For the first step we use \eqref{thin-layer}, which suffices if we impose 
the constraint \eqref{thin-layer-h}, which we recall here
\[
h_1 > h(1+\frac{1}{\kappa}).
\]
For the second step, we simply note that the good variables are identical except for
the $r$ factors, where we replace $r_1$ by $r_1-c$. Hence it suffices 
to ensure that
\[
c \lesssim \epsilon r_1 \qquad \text{in }\Omega^{[< h_1]},
\]
which yields 
\[
(k-k_0)h^- >  h + h_1.
\]
These two constraints for $h_1$ are again compatible if $k$ is large enough. The proof of the Lemma is concluded.

\end{proof}

Combining now the outcomes of Lemma~\ref{l:first-reg} and Lemma~\ref{l:subtract-c}, it follows that our 
final regularization $(r,v)$ satisfies the bound \eqref{e-growth}. It also satisfies \eqref{epsilon2}
and \eqref{e-high} due to Lemma~\ref{l:trv-reg}; there one can harmlessly substitute
the weight $\rm$ by $r$ since $(r,v)$ are smooth on the $\epsilon^2$ scale, which is larger than $c$. Thus the proof of Proposition~\ref{p:reg-step} is concluded.

\end{proof}

\subsection{Construction of regular exact solutions} \
Here we use the approximate solutions above. Given an initial data $(r_0,v_0)$ so that 
\[
\|(r_0,v_0)\|_{\bfH^{2k}} \leq M
\]
applying the successive iterations above we obtain approximate solutions  $(r^\epsilon,v^\epsilon)$
defined at $\epsilon$ steps, so that
\[
E^{2k}(r^\epsilon,v^\epsilon) ((j+1)\epsilon) \lesssim (1 +C(M) \epsilon)E^{2k}(r^\epsilon,v^\epsilon) (j\epsilon).
\]
 By discrete Gronwall's inequality, it follows that these approximate solutions are defined 
uniformly up to a time $T = T(M)$, with uniform bounds
\begin{equation}\label{unif-app}
\|(r^\epsilon,v^\epsilon)\|_{\bfH^{2k}} \lesssim_M 1, \qquad t \in [0,T].
\end{equation}

On the other hand, in a weaker topology we have
\[
(r^\epsilon,v^\epsilon) ((j+1)\epsilon) - (r^\epsilon,v^\epsilon) (j\epsilon) = O(\epsilon).
\]
Hence by Arzela-Ascoli we get uniform convergence on a subsequence to a function $(r,v)$
in a $C^j$ norm, uniformly in $t$. Passing to the limit in the relation  \eqref{good-iterate1}, it follows that 
$(r,v)$ solves our equation. Finally, taking weak limits in the norms
in \eqref{unif-app} we also obtain an energy bound on $(r,v)$,
\begin{equation}\label{unif-app-lim}
\|(r,v)(t)\|_{\bfH^{2k}} \lesssim_M 1, \qquad t \in [0,T].
\end{equation}

\section{Rough solutions}

Our goal in this section is to construct rough solutions as limits of smooth solutions, and conclude the proof of Theorem~\ref{t:lwp}.  In terms of a general outline, the argument here is relatively standard, and involves 
the following steps:
\begin{enumerate}
    \item We regularize the initial data,
    \item We prove uniform bounds for the regularized solutions,
    \item We prove convergence of the regularized solutions in 
    a weaker topology,
    \item We prove the convergence in the strong topology by combining
    the weak difference bounds with the uniform bounds in a frequency envelope fashion.
\end{enumerate}

The main difficulty we face is that our phase  space is not linear,
and at each stage we have to compare functions on different domains. For a description of the ideas here in a simpler, model setting we refer the reader to the expository paper \cite{IT-primer}.

\subsection{Regularizing the initial data}

Given a rough initial data $(r_0,v_0) \in \bfH^{2k}$,
our first task is to construct an appropriate family of regularized data, depending smoothly of the regularization parameter. Here it suffices to directly use the 
family of regularizations provided by Proposition~\ref{p:reg-state}.

\subsection{ Uniform bounds and the life-span of regular solutions}
Once we have the regularized data sets $(r_0^h,v_0^h)$, we consider 
the corresponding smooth solutions  $(r^h,v^h)$ generated by the smooth data $(r_0^h,v_0^h)$. A-priori these solutions exist on a time interval that 
depends on $h$. Instead, we would like to have a lifespan bound which is independent of $h$. To obtain this, we use a bootstrap argument for our 
control parameter $B$ for $(r^h,v^h)$, which depends on $h$ and $t$.

For a large parameter $B_0$, to be chosen later, we will make the bootstrap assumption
\begin{equation}\label{boot-B}
B(t,h) \leq 2B_0, \qquad t \in [0,T], \quad 0 \leq h \leq h_0.    
\end{equation}
The solutions $(r^h,v^h)$ can be continued for as long as this 
is satisfied. We will prove that we can improve this bootstrap assumption
provided that $T$ is small enough, $T \leq T_0$, but with $T_0$ independent of $h \leq h_0$.
Here $h_0$ is finite but arbitrarily large; its role is simply to ensure that we run the bootstrap argument on finitely many quantities
at once.

Our choice of $T_0$ will be quite straightforward, 
\begin{equation} \label{choose-T}
    T_0 \leq \frac{1}{B_0}.
\end{equation}
In view of our energy estimates in Theorem~\ref{t:energy}
and Gronwall's inequality, this guarantees uniform energy bounds
for the solutions $(r^h,v^h)$ in all integer Sobolev spaces $\H^{2l}$
in $[0,T]$. 

We remark that the bound \eqref{app-uniform} does not directly
propagate unless $k$ is an integer. Indeed, in that case one could 
immediately close the bootstrap at the level of the $\H^{2k}$ norm
using the embeddings \eqref{sobolev-A} and \eqref{sobolev-B}.
The goal of the argument that follows is to establish 
the $\H^{2k}$ bound for noninteger $k$, by working only with 
energy estimates for integer indices.

Combining Theorem~\ref{t:energy} with \eqref{app-high} we obtain
the higher energy bound in $[0.T]$
\begin{equation}\label{app-high-t}
\| (r^h,v^h) \|_{\bfH^{2k+2j}_h}   \lesssim  2^{2hj}  c_h, \qquad j > 0, \quad j+k \in \N.
\end{equation}
Next we consider the bound \eqref{app-diff}, which we reinterpret in a discrete fashion as a difference bound
\begin{equation}\label{app-diff-discrete}
 D((r_0^h,v_0^h),  (r_0^{h+1},v_0^{h+1}))  \lesssim  2^{-4hk}  c^2_h .
\end{equation}
This bound we can also propagate by Theorem~\ref{t:Diff}, to obtain, also in $[0,T]$, the estimate
\begin{equation}\label{app-diff-t}
 D((r^h,v^h),  (r^{h+1},v^{h+1}))  \lesssim  2^{-4hk}  c^2_h .
\end{equation}

Our objective now is to combine the bounds \eqref{app-high-t} and \eqref{app-diff-t} in order to obtain
a uniform $\H^{2k}$ bound
\begin{equation}\label{uniform-2k}
\| (r^h,v^h)\|_{\bfH^{2k}} \lesssim M := \|(r_0,v_0)\|_{\bfH^{2k}}  .  
\end{equation}
To prove this, we would naively like to consider a  representation of the form
\[
(r^h,v^h) = (r^{1},v^1) + \sum_{l=1}^{h-1} (r^{l+1}-r^l,v^{l+1}-v^l)  ,
\]
where we can estimate the successive terms in both $\H$ and $\H^{2N}$. The difficulty we face 
is that these functions have different domains. Hence the first step is to use the bounds \eqref{app-high-t} and \eqref{app-diff-t} in order to compare these domains. 

\begin{lemma}\label{l:domains}
Assume that $r^h$ and $r^{h+1}$ are nondegenerate, and that \eqref{app-diff-t} holds. Then we have
\begin{equation}
d(\Gamma_h,\Gamma_{h+1}) \lesssim      2^{-h(2+\delta)}, \qquad \delta > 0.
\end{equation}    
\end{lemma}

\begin{proof}
We use the uniform nondegeneracy property for the functions $r_h$ in order to compare these domains. If $r = d(\Gamma_h,\Gamma_{h+1})$, then 
we can find a ball $B_{cr}$ in the common domain so that 
\[
r^h, r^{h+1}, |r^h-r^{h+1}| \approx r \qquad \text{ in } B_{cr}.
\]
Then we obtain
\[
r^{d+1+\frac{1}{\kappa}} \lesssim D((r^h,v^h),  (r^{h+1},v^{h+1}))  \lesssim  2^{-4hk}  c^2_h 
\]
or equivalently
\[
r^{2\kappa_0} \lesssim  2^{-4hk}  c^2_h.
\]

Since $k > k_0$, we obtain
\[
r \lesssim 2^{-h(2+\delta)}, \qquad \delta > 0.
\]
\end{proof}

Now we return to our expansion for $(r^h,v^h)$. In order to compare functions which are defined on a common 
domain, we replace the functions $(r^l,v^l)$ with their regularizations $\Psi^l (r^l,v^l)$. Their domain includes
an additional $2^{-2l}$ boundary layer, which by the previous Lemma~\ref{l:subtract-c} suffices in order to cover the domain $\Omega_h$
for all $h > l$. Then we write
\[
(r^h,v^h) = \Psi^0(r^0,v^0) + \sum \Psi^{l+1}(r^{l+1},v^{l+1}) - \Psi^l(r^l,v^l) + (I - \Psi^h)(r^h,v^h),
\]
and claim that this decomposition is as in Lemma~\ref{l:interp-spaces}.

The first term is trivial. For the last one we use the boundedness of $\Psi^h$ in $\H^{2k}$ and 
the bound \eqref{Lk-bd5} integrated in $h$ to write
\[
\| (I - \Psi^h)(r^h,v^h)\|_{\H^{2N}} \lesssim \| (r^h,v^h)\|_{\H^{2N}},
\]
respectively 
\[
\| (I - \Psi^h)(r^h,v^h)\|_{\H} \lesssim 2^{-2Nh} \| (r^h,v^h)\|_{\H^{2N}},
\]
for a fixed large enough integer $N$,
which together suffice in order to place this term into $(s^h,w^h)$, with norm $c_h$.

For later use, we state the remaining bound for intermediate $l$ as a separate result:

\begin{lemma}\label{l:diff-reg}
For any nondegenerate $r$ with $|r - r^l| \ll 2^{-2l}$ we have the difference bounds
\begin{equation}
\|  \Psi^{l+1}(r^{l+1},v^{l+1}) - \Psi^l(r^l,v^l)\|_{\H_r} \lesssim   2^{-2l k} c_l,
\end{equation}
\begin{equation}
\|  \Psi^{l+1}(r^{l+1},v^{l+1}) - \Psi^l(r^l,v^l)\|_{\H_r^{2N}} \lesssim   2^{2l (N-k)} c_l.
\end{equation}
\end{lemma}
As a corollary of this lemma, we remark that via Sobolev embeddings we also get uniform difference bounds:

\begin{corollary}\label{c:diff-reg}
In the region $\tOmega^{[l]}$ have
\begin{equation}
\|  \Psi^{l+1}(r^{l+1},v^{l+1}) - \Psi_l(r^l,v^l)\|_{C^{\frac32} \times C^1} \lesssim   2^{-2\delta l}, \qquad \delta > 0.
\end{equation}
\end{corollary}
This will serve later in the study of convergence of the regularized solutions.

\begin{proof}
We split
\[
\Psi^{l+1}(r^{l+1},v^{l+1}) - \Psi^l(r^l,v^l) = (\Psi^{l+1}-\Psi^l)(r^{l+1},v^{l+1}) - \Psi^l(r^{l+1}-r^l,v^{l+1}-v^l) .
\]
For the first term we use again the boundedness of $\Psi^l$ and then \eqref{Lk-bd5} to conclude 
that 
\[
\| (\Psi^{l+1}-\Psi^l)(r^{l+1},v^{l+1}) \|_{\H^{2N}} \lesssim \| (r^{l+1},v^{l+1})\|_{\H^{2N}} \lesssim  2^{2l (N-k)} c_l,
\]
and 
\[
\| (\Psi^{l+1}-\Psi^l)(r^{l+1},v^{l+1}) \|_{\H} \lesssim 2^{-2lN} \| (r^{l+1},v^{l+1})\|_{\H^{2N}}\lesssim 2^{-2l k} c_l
\]
as needed.

For the second term we use again the $\H^{2N}$ boundedness, but for the $\H$ bound
we use instead the difference bound \eqref{app-diff-t} together with the $\H$ bound
\[
\| \Psi^l(r^{l+1}-r^l,v^{l+1}-v^l)\|_{\H}^2 \lesssim D((r^{l+1},v^{l+1}),(r^l,v^l)),
\]
and conclude using \eqref{app-diff-t}.
\end{proof}

By the above Lemma we can place the telescopic term  into $(s^l,w^l)$, with norm $c_l$
and thus, by Lemma~\ref{l:interp-spaces}, we obtain the desired bound \eqref{uniform-2k}
and conclude our bootstrap argument.

\subsection{The limiting solution} 

Here we show that the limit 
\begin{equation}
(r,v) = \lim_{h \to \infty}   (r^h,v^h)  
\end{equation}
exists, first in a weaker topology and then in the strong $\bfH^{2k}$ topology.

As before, the smooth solutions $(r^h,v^h)$ do not have common domains.
However, by Lemma~\ref{l:domains} the limit 
\[
\Omega = \lim_{h \to \infty} \Omega_h
\]
exists, has a Lipschitz boundary $\Gamma$, and further we have 
\[
d(\Gamma,\Gamma_h) \lesssim 2^{-h(2+\delta)}.
\]
For this reason, it is convenient to consider instead the limit
\[
(r,v) = \lim_{h \to \infty}   \Psi^h(r^h,v^h) ,
\]
where the functions on the right are all defined in $\Omega$.
Indeed, by Lemma~\ref{l:diff-reg} we see that we have convergence
in $\H$, and, by interpolation, in $\H^{2k_1}$ for all $k_1 < k$.

To obtain convergence in $\H^{2k}_r$, we write
\[
(r,v) = \Psi^0(r^0,v^0) + \sum_{j=0}^\infty   \Psi^{l+1}(r^{l+1},v^{l+1}) - \Psi^l(r^l,v^l),
\]
and view the telescopic sum as a generalized Littlewood-Paley decomposition of $(r,v)$.
Then Lemma~\ref{l:diff-reg} shows that $(r,v)$ is in $\H^{2k}$, with norm
\begin{equation}
\| (r,v)\|_{\H^{2k}} \lesssim \| c_h\|_{l^2} .   
\end{equation}
We also see that we have convergence in $\H^{2k}$, namely
\begin{equation}
\|\Psi^l(r^l,v^l) -  (r,v)\|_{\H^{2k}} \lesssim \| c_{\geq l}\|_{l^2} \to 0   . 
\end{equation}

We also show that we have strong convergence of $(r^h,v^h)$ in $\H^{2k}$ in the sense 
of Definition~\ref{d:convergence}. Indeed, it suffices to compare it with the constant
sequence $\Psi^l(r^m,v^m)$. Then for $l \geq m$ we have
\begin{equation}
\|(r^l,v^l) -  \Psi^m(r^m,v^m)\|_{\H^{2k}} \lesssim \| c_{\geq m}\|_{l^2} \to 0    
\end{equation}
The same relations also show the continuity of $(r,v)$ in $\bfH^{2k}$ as functions of time.

\subsection{Continuous dependence}

We consider a sequence of initial data $(r^{(n)}_0,v^{(n)}_0)$ which converges to $(r_0,v_0)$ in $\bfH^{2k}$
in the sense of Definition~\ref{d:convergence}, and will show that the corresponding solutions
$(r^{(n)},v^{(n)})$ converge to $(r,v)$.

The first observation is that the  $\bfH^{2k}$ convergence implies $\bfH^{2k}$ uniform boundedness
for $(r^{(n)}_0,v^{(n)}_0)$, which in turn implies a uniform lifespan bound for the solutions
as well as a uniform bound in $\bfH^{2k}$. 

Our strategy to prove convergence is to compare this family of solutions with the limit $(r,v)$
via the regularizations used in the construction of rough solutions. Precisely, denote by 
$(r^{(n),h}_0,v^{(n),h}_0)$ respectively $(r^h_0,v^h_0)$ the regularized data sets, for which we have the 
obvious convergence
\[
(r^{(n),h}_0,v^{(n),h}_0) \to (r^h_0,v^h_0) \qquad \text{in } C^\infty.
\]
These are also uniformly bounded in $\bfH^{2k}$ and thus have a uniform lifespan.

Denoting by $c_h^{(n)}$ corresponding frequency envelopes for $(r^{(n)}_0,v^{(n)}_0)$, we have the 
difference bounds
\[
\| \Psi^h (r^{(n),h}, v^{(n),h}) -  (r^{(n)}, v^{(n)})\|_{\H^{2k}} \lesssim c^{(n)}_{\geq h}.
\]

To finish the proof we need to establish two facts:

\begin{itemize}
    \item For each $\epsilon > 0$, the frequency envelopes $c^{(n)}_h$ can be chosen so that\footnote{One can do better 
    than that and ensure that the limit is zero, but that is not needed for our argument.}
    \[
    \limsup_{h \to \infty} \sup_n c^{(n)}_{\geq h} \leq \epsilon.
    \]
    
    \item We have the $C^\infty$ convergence
    \[
    \Psi^h (r^{(n),h}, v^{(n),h}) \to  \Psi^h (r^{h}, v^{h}).
    \]
\end{itemize}

\bigskip

\emph{(i) Equicontinuity of frequency envelopes.}
This is easily achieved via the decomposition 
\[
(r^{(n)}_0,v^{(n)}_0) = (r_0^{smooth},v_0^{smooth}) + O_{\H^{2k}}(\epsilon),
\]
which holds for each $\epsilon$. The smooth part yields envelopes which are uniformly decreasing, 
and the error term yields $\epsilon$ sized envelopes.

\bigskip

\emph{(ii) $C^{\infty}$ convergence.}
Here we have uniform $\H^{2N}$ bounds for the sequence $(r^{(n),h}, v^{(n),h})$, as well as weak convergence, 
in the sense that
\[
D((r^{(n),h}, v^{(n),h}),(r^{h}, v^{h})) \to 0.
\]
The last property implies domain convergence. Then we have $L^2$ convergence away from a $2^{-2h}$ boundary layer, which in turn shows convergence of the regularizations in $C^\infty$.

\subsection{The lifespan of rough solutions}

Here we complete the proof of our last result in Theorem~\ref{t:continuation}. Thus, we 
consider a rough initial data $(r_0,v_0)\in \bfH^{2k}$ and a corresponding solution $(r,v)$
in a time interval $[0,T)$ with the property that 
\begin{equation}\label{intB-have}
\int_0^T B(t)\, dt = C < \infty    .
\end{equation}
By the local well-posedness result, in order to prove the theorem it suffices to show that we have a uniform bound 
\begin{equation}
 \sup_{t \in [0,T]} \| (r,v)\|_{\bfH^{2k}}   < \infty .
\end{equation}

We consider the regularized data $(r_0^h,v_0^h)$ and the corresponding solutions $(r^h,v^h)$. 
By the continuous dependence theorem we know that these solutions converge to $(r,v)$ in $[0,T)$, and in particular their lifespans $T^h$ satisfy  
\[
\liminf_{h \to \infty} T^h \geq T.
\]
What we do not have is a uniform bound for their corresponding control parameters $B^h$. 
To rectify this, we consider a large parameter $h_0$, to be chosen later, and we will show that,
for $h > h_0$, the solutions $(r^h,v^h)$ persist up to time $T$ with uniform bounds 
\begin{equation}\label{intB-want}
\int_{0}^T  B^h(t) dt  \leq 2C, \qquad h \geq h_0.
\end{equation}
If that were the case, then by the local well-posedness proof it follows that the solutions 
$(r^h,v^h)$ remain uniformly bounded in $\bfH^{2k}$ and converge to $(r,h)$, thereby concluding the proof.

To establish the bound \eqref{intB-want} we will run a bootstrap argument. Precisely, we assume 
that on a time interval $[0,T_0]$ with $T_0 < T$ we have a uniform bound
\begin{equation}\label{intB-boot}
\int_{0}^{T_0}  B^h(t) dt  \leq 4C, \qquad h \geq h_0.
\end{equation}
Then we will show that in effect we must have the better bound
\begin{equation}\label{intB-get}
\int_{0}^{T_0}  B^h(t) dt  \leq 2C, \qquad h \geq h_0.
\end{equation}
That would suffice, for then the local well-posedness argument would yield a uniform 
for $(r^h,v^h)$ in $\bfH^{2k}$ and thus allow us to expand the interval $[0,T_0]$ on which the bootstrap assumption holds, uniformly with respect to $h \geq h_0$.

Our goal now is to compare $B^h$ and $B$. Precisely, we aim to show that 
\begin{equation}\label{B-comp}
B^h \lesssim C_1 B + C_2 2^{-\delta h}, \qquad \delta > 0,
\end{equation}
with a universal constant $C_1$ but $C_2$ depending both on the initial data size and on $C$ above.
This suffices in order to establish \eqref{intB-get}, because we are allowed to choose the threshold 
$h_0$ sufficiently large, depending on parameters which are fixed in the problem.

The tools we have at our disposal are 

(i)  a high frequency bound, provided by our energy estimates in \eqref{t:energy}, namely \eqref{app-high-t},

(ii) the difference bound \eqref{app-diff-t}. 

The constants in both bounds depend exactly on the $\bfH^{2k}$ norm initial data and on $C$ above.

The difficulty we have in comparing $B^h$ and $B$ is that the two solutions are supported in different domains
$\Omega$ respectively $\Omega_h$.  However, the difference bound \eqref{app-diff-t} allows us to apply 
Lemma~\ref{l:domains} to conclude that the two domains are at distance $\lesssim 2^{-h(2+\delta)}$.
Thus, rather than comparing $(r,v)$ and $(r^h,v^h)$, it is better to compare their regularizations
$\Psi^h (r,v)$ and $\Psi^h (r^h,v^h)$, which are defined on $2^{-2h}$ enlargements of the domains, which in particular cover the union of $\Omega$ and $\Omega_h$. By a slight abuse of notation, we will identify their domains.

We begin with $\Psi^h (r,v)$, for which we have the straightforward bound
\begin{equation}\label{B-comp-a}
\| \nabla \Psi^h (r,v) \|_{\tC \times L^\infty} \leq C_1 B.
\end{equation}
This is where the universal constant $C_1$ appears. 

Next we compare $\Psi^h (r,v)$ and $\Psi^h (r^h,v^h)$. Here we take a telescopic sum,
\[
\Psi^h (r,v) - \Psi^h (r^h,v^h) = \sum_{l = h}^\infty \Psi^h (r^{l+1},v^{l+1}) - \Psi^h (r^l,v^l).
\]
Using the difference bound \eqref{app-diff-t}, we can estimate the successive terms in 
all $\H^{2m}_{r^h}$ norms,
\[
\| \Psi^h (r^{l+1},v^{l+1}) - \Psi^h (r^l,v^l) \|_{\H^{2m}_{r^h}} \lesssim c_l 2^{-2kl} 2^{2mh}, \qquad m \geq 0,
\]
which after summation yields
\[
\| \Psi^h (r,v) - \Psi^h (r^l,v^l) \|_{\H^{2m}_{r^h}} \lesssim c_l 2^{-2kh} 2^{2mh}, \qquad m \geq 0.
\]
Now we can use Sobolev embeddings to conclude that\footnote{From Sobolev embeddings we get in effect a $\dot C^\frac12$ bound for the first component.}
\begin{equation}\label{B-comp-b}
\| \nabla(\Psi^h (r,v) - \Psi^h (r^h,v^h))\|_{\tC  \times L^\infty} \lesssim 2^{-\delta h}.
\end{equation}

Finally, using  \eqref{app-high-t}, we compare $\Psi^h (r^h,v^h)$ with $(r^h,v^h)$, estimating also the low 
frequencies,
\[
\| (r^h,v^h) - \Psi^h (r^h,v^h)\|_{\H^{2m}_{r^h}} \lesssim c_l 2^{-2kh} 2^{2mh}  \qquad m \geq 0.
\]
Using Sobolev embeddings again, we conclude that 
\begin{equation}\label{B-comp-c}
  \| \nabla((r^h,v^h) - \Psi^h (r^h,v^h))\|_{  \tC \times L^\infty} \lesssim 2^{-\delta h} .
\end{equation}

Now \eqref{B-comp} is obtained by combining \eqref{B-comp-a}, \eqref{B-comp-b} and \eqref{B-comp-c}, 
and the proof of the theorem is concluded.

\bibliographystyle{abbrv}


\end{document}